\documentclass[11pt]{amsart}

\usepackage{graphicx}
\usepackage[english]{babel}
\usepackage[T1]{fontenc}
\usepackage[latin1]{inputenc}
\usepackage{amsfonts}
\usepackage{amssymb}
\usepackage{amsthm}
\usepackage{amsmath}
\usepackage{tikz}
\usepackage{float}
\usepackage{enumerate}
\usepackage{accents}
\usepackage{mathtools}
\usepackage{mathrsfs}
\usepackage{comment}
\usepackage{afterpage}
\usepackage{bbm}
\usepackage{dsfont}
\usepackage{stmaryrd}
\usepackage{hyperref}
\usepackage{mleftright}
\usepackage{subfig}
\usepackage{soul}
\usepackage{tikz-cd} 
\usepackage{fullpage}

\theoremstyle{plain}
\newtheorem{thm}{Theorem}[section]
\newtheorem*{thm*}{Theorem}
\newtheorem{prop}[thm]{Proposition}
\newtheorem*{prop*}{Proposition}
\newtheorem{cor}[thm]{Corollary}
\newtheorem*{cor*}{Corollary}
\newtheorem{lem}[thm]{Lemma}

\newtheorem{thmintro}{Theorem}

\newtheorem{corintro}[thmintro]{Corollary}

\theoremstyle{definition}
\newtheorem{defn}[thm]{Definition}
\newtheorem*{defn*}{Definition}
\newtheorem{ex}[thm]{Example}
\newtheorem{rmk}[thm]{Remark}
\newtheorem*{rmk*}{Remarks}
\newtheorem{cstr}[thm]{Construction}

\newtheorem*{conj*}{Conjecture}
\newtheorem*{quest*}{Question}

\newtheoremstyle{blue-environment}{}{}{}{}{\color{blue}\bfseries}{.}{ }{}
\theoremstyle{blue-environment}

\newcommand{\acts}{\curvearrowright}
\newcommand{\ra}{\rightarrow}
\newcommand{\Ra}{\Rightarrow}

\newcommand{\sq}{\subseteq}

\newcommand{\wt}{\widetilde}

\newcommand{\x}{\times}

\newcommand{\id}{\mathrm{id}}

\newcommand{\mc}{\mathcal}
\newcommand{\mf}{\mathfrak}
\newcommand{\mscr}{\mathscr}

\newcommand{\R}{\mathbb{R}}
\newcommand{\Z}{\mathbb{Z}}

\newcommand{\Q}{\mathbb{Q}}

\newcommand{\s}{\sigma}

\newcommand{\Om}{\Omega}

\newcommand{\g}{\gamma}
\newcommand{\G}{\Gamma}
\newcommand{\CAT}{{\rm CAT(0)}}
\newcommand{\Out}{{\rm Out}}
\newcommand{\Aut}{{\rm Aut}}

\newcommand{\X}{\mc{X}}

\DeclareMathOperator{\hull}{Hull} 
\DeclareMathOperator{\diam}{diam} 

\DeclareMathOperator{\lk}{lk}
\DeclareMathOperator{\St}{st}
\DeclareMathOperator{\Fix}{Fix}

\newcommand{\diagentry}[1]{\mathmakebox[1.8em]{#1}}
\newcommand{\xddots}{%
  \raise 4pt \hbox {.}
  \mkern 6mu
  \raise 1pt \hbox {.}
  \mkern 6mu
  \raise -2pt \hbox {.}
}

\numberwithin{equation}{section}

\begin{document}

\title{Coarse cubical rigidity} 

\author[E. Fioravanti]{Elia Fioravanti}
\address{Max Planck Institute for Mathematics}
\email{fioravanti@mpim-bonn.mpg.de} 
\thanks{Fioravanti thanks the Max Planck Institute for Mathematics in Bonn for their hospitality and financial support. He was also partially supported by Emmy Noether grant 515507199 of the Deutsche Forschungsgemeinschaft (DFG)}

\author[I. Levcovitz]{Ivan Levcovitz}
\address{Tufts University}
\email{Ivan.Levcovitz@gmail.com} 

\author[M. Sageev]{Michah Sageev}
\address{Technion}
\email{sageevm@technion.ac.il}
\thanks{Sageev was supported by ISF grant no.\ 660/20}

\begin{abstract}
We show that for many right-angled Artin and Coxeter groups, all cocompact cubulations coarsely look the same: they induce the same coarse median structure on the group. These are the first examples of non-hyperbolic groups with this property.

For all graph products of finite groups and for Coxeter groups with no irreducible affine parabolic subgroups of rank $\geq 3$, we show that all automorphism preserve the coarse median structure induced, respectively, by the Davis complex and the Niblo--Reeves cubulation. As a consequence, automorphisms of these groups have nice fixed subgroups and satisfy Nielsen realisation.
\end{abstract}

\maketitle
%\tableofcontents

\addtocontents{toc}{\protect\setcounter{tocdepth}{1}}
\section{Introduction}

Cubulating a group $G$ --- constructing a proper $G$--action on a $\CAT$ cube complex --- is a particularly effective strategy to study many algebraic properties of $G$. For instance, the existence of a cocompact cubulation implies that $G$ is biautomatic \cite{NR98}, has finite asymptotic dimension \cite{Wright-asdim} and satisfies the Tits alternative \cite{SW05}.

More recently, cubulations have also proven particularly fruitful in the study of group automorphisms. For a right-angled Artin group $A_{\G}$, a certain space of particularly nice cubulations of $A_{\G}$ provides a remarkably simple (rational) classifying space for $\Out(A_{\G})$ \cite{CCV,CSV,BCV}, extending the classical construction of Outer Space for $\Out(F_n)$ \cite{CV}. In addition, for a general cocompactly cubulated group $G$, automorphisms that ``coarsely preserve'' a given cubulation of $G$ always enjoy particularly good properties \cite{Fio10a}, for instance their fixed subgroups are finitely generated and undistorted in $G$.

In light of these results, it is tempting to study spaces of cubulations in greater generality. One promising feature is that cubulations are often completely determined by their length function \cite{BF1,BF2}. However,  the space of \emph{all cubulations} of a group can be extremely vast and flexible. For instance, if $S$ is a closed hyperbolic surface, every finite filling collection of simple closed curves on $S$ gives rise to a cocompact cubulation of $\pi_1(S)$ \cite{Sag95}, and there are additional cubulations originating from subsurfaces of $S$. Even right-angled Artin groups always admit many more cubulations than those appearing in the Outer Spaces mentioned above (see e.g.\ Example~\ref{hexagon RAAG ex} below).

This leads to two natural questions, which serve as our main motivation for this work. 
\begin{enumerate}
\setlength\itemsep{.25em}
\item Do all cocompact cubulations of a group $G$ have anything in common? 
\item Is there a ``best'' cubulation for $G$? 
\end{enumerate}

In relation to Question~(1), it is natural to wonder whether all cubulations of $G$ give rise to the same notion of ``convex-cocompactness''. Specifically, we say that a subgroup $H\leq G$ is \emph{convex-cocompact} in a cubulation $G\acts X$ if there exists an $H$--invariant convex subcomplex $C\sq X$ on which $H$ acts cocompactly. 

It is well-known that, when $G$ is word-hyperbolic, all cocompact cubulations of $G$ give rise to the same convex-cocompact subgroups; in fact, convex-cocompact subgroups are precisely quasi-convex ones \cite[Theorem~H]{Haglund-GD}. However, this property can fail rather drastically for non-hyperbolic groups, whose cubulations can induce infinitely many distinct notions of convex-cocompactness.

For instance, let $\Z^2\acts\R^2$ be the standard $\Z^2$--action on the standard square tiling of $\R^2$, with one generator $x$ translating horizontally by $1$ and the other generator $y$ translating vertically by $1$. The only convex-cocompact maximal cyclic subgroups for this action are $\langle x\rangle$ and $\langle y\rangle$. Precomposing the action $\Z^2\acts\R^2$ with an element $A\in{\rm SL}_2(\Z)$ produces another action of $\Z^2$ on the standard tiling of $\R^2$, where the subgroups $\langle A^{-1}x\rangle$ and $\langle A^{-1}y\rangle$ become convex-cocompact. These two actions have the same convex-cocompact subgroups if and only if $A$ is a signed permutation matrix. Thus we get infinitely many types of actions in terms of which subgroups are convex-cocompact.

As a matter of fact, both our motivating questions are best phrased in terms of \emph{coarse median structures}. Rather than the general notion introduced by Bowditch \cite{Bow-cm}, the following very particular instance will suffice for most of our purposes in this paper. Recall that Chepoi--Roller duality states that every $\CAT$ cube complex $X$ is equipped with a \emph{median operator} $m\colon X^3\ra X$ and that, conversely, every discrete median algebra uniquely arises in this way  \cite{Chepoi,Roller}. If $x,y,z\in X$ are vertices, then the median $m(x,y,z)$ is the only vertex of $X$ that lies on a geodesic in the $1$--skeleton of $X$ between any two of the three vertices $x,y,z$.

\begin{defn*}
A \emph{cubical coarse median} on $G$ is a map $\mu_X\colon G^3\ra G$ obtained by pulling back the median operator of a cocompact cubulation $G\acts X$ via an equivariant quasi-isometry $G\ra X$ (see Subsection~\ref{subsect:cms} for details).
Two cocompact cubulations are said to induce the same \emph{coarse median structure} if their corresponding cubical coarse medians on $G$ are at uniformly bounded distance from each other.
\end{defn*}

In other words, two cubulations induce the same coarse median structure on $G$ if they are $G$--equi\-va\-riant\-ly quasi-isometric via a map that coarsely respects medians. Additional motivation for studying these ``coarsely median preserving'' quasi-isometries, even beyond the world of cubulated groups, comes from recent work of Petyt. He showed that mapping class groups are coarsely median preserving quasi-isometric to $\CAT$ cube complexes \cite{Hagen-Petyt,Petyt}, albeit certainly not to any that admit geometric group actions, since mapping class groups are QI--rigid \cite{BKMM,Ham-QI}.

The fundamental connection between cubical coarse medians and convex-cocompactness is provided by the following result, which will be one of our main tools. It can be deduced from the first author's work in \cite{Fio10e} and we explain in detail how in Theorem~\ref{coarse median vs hyperplanes}. 

\begin{thm*}[\cite{Fio10e}]
Let $G\acts X,Y$ be cocompact cubulations. The following are equivalent:
\begin{enumerate}
\item $G\acts X$ and $G\acts Y$ induce the same coarse median structure on $G$;
\item $G\acts X$ and $G\acts Y$ have the same convex-cocompact subgroups.
\end{enumerate}
\end{thm*}

In view of this result and the above questions, we are interested in the following property.

\begin{defn*}
A group $G$ satisfies \emph{coarse cubical rigidity} if it has a unique cubical coarse median structure. Equivalently, all cocompact cubulations of $G$ induce the same notion of convex-cocompactness.
\end{defn*}

Hyperbolic groups always satisfy this form of rigidity, provided that they are cocompactly cubulated. More generally, an arbitrary hyperbolic group always admits a unique coarse median structure (not necessarily a cubical one), as shown in \cite[Theorem 4.2]{NWZ1}.
By contrast, as exemplified above, for $n\geq 2$ the group $\Z^n$ admits countably many distinct cubical coarse median structures (see Proposition~\ref{cms on Z^n} for a classification), and uncountably many non-cubical ones.

Some exceptional groups satisfy even stronger forms of cubical rigidity. For instance, many Burger--Mozes--Wise groups \cite{Burger-Mozes,Wise-thesis,Caprace-BMW} admit a \emph{unique} cocompact cubulation with no free faces \cite[Proposition~C]{FH}. However, these stronger kinds of rigidity appear to be extremely rare in general, and they are likely to never occur for right-angled Artin groups, see Example~\ref{hexagon RAAG ex}. 

A weaker form of coarse cubical rigidity is the existence of a coarse median structure which is preserved by any automorphism of the ambient group. Thus, while there is not a unique coarse median structure, there is a special one.

\medskip
The main goal of this paper is to develop tools for proving coarse cubical rigidity of groups, and to exhibit many new examples of groups with this property. Right-angled Artin and Coxeter groups provide some natural candidates and they will indeed be our main focus, though we will discuss results for graph products and Coxeter groups generally.

\subsection{Right-angled Artin groups}

Recall that a right-angled Artin group $A_\G$ is obtained from a finite simplicial graph $\G$ via the presentation: 
\[A_\G= \langle \G^{(0)}\mid [v,w]=1 \iff v,w \text{ span an edge in } \G\rangle. \]
We say that $A_{\G}$ is \emph{twistless} if there do not exist any vertices $v,w\in\G$ with $\St(v)\sq\St(w)$; equivalently, $A_{\G}$ does not have any twist automorphisms (see Section 3 for details). As an example, $A_{\G}$ is twistless if the outer automorphism group $\Out(A_{\G})$ happens to be finite. 

In previous work, it was shown that $A_{\G}$ is twistless if and only if all automorphisms of $A_{\G}$ preserve the coarse median structure induced by the Salvetti complex \cite[Proposition~A(3)]{Fio10a}. Thus, twistlessness is a necessary condition for $A_{\G}$ to have a unique cubical coarse median structure. Our first main result is that this condition is also sufficient.  

\begin{thmintro}\label{RAAGs intro}
Let $A_{\G}$ be a right-angled Artin group. Then $A_{\G}$ satisfies coarse cubical rigidity if and only if it is twistless.
\end{thmintro}

We emphasise that Theorem~\ref{RAAGs intro} concerns arbitrary cocompact cubulations of $A_{\G}$, not just the ones that make up the Outer Spaces constructed by Charney, Vogtmann and coauthors \cite{CCV,CSV,BCV}. In general, $A_{\G}$ has many perfectly nice cubulations outside these Outer Spaces, even cubulations with the same dimension as the Salvetti complex and without any free faces. We demonstrate this in the case when $\G$ is a hexagon in Example~\ref{hexagon RAAG ex}.

Theorem~\ref{RAAGs intro} can be rephrased as follows: every cocompact cubulation of $A_{\G}$ can be obtained from the standard action on (the universal cover of) the Salvetti complex by blowing up to hyperplanes finitely many walls associated with convex-cocompact subgroups of $A_{\G}$, and then possibly collapsing some hyperplanes of the Salvetti complex. See Theorem~\ref{coarse median vs hyperplanes}(4) for the equivalence with the above formulation of Theorem~\ref{RAAGs intro}. When $A_{\G}$ is not twistless, more complicated procedures are required to reach arbitrary cocompact cubulations (e.g.\ precomposition with an automorphism of $A_{\G}$), because we need to be able to modify the induced coarse median structure.

We will actually prove a more general version of Theorem~\ref{RAAGs intro} that applies to \emph{all} right-angled Artin groups. This is best described in terms of the product decomposition
\[\Aut(A_{\G})=T(A_{\G})\cdot U(A_{\G}),\]
where $T(A_{\G})$ and $U(A_{\G})$ are, respectively, the \emph{twist subgroup} and the \emph{untwisted subgroup}. The subgroup $U(A_{\G})$ consists of all those automorphisms that fix the coarse median structure induced by the Salvetti complex \cite{Fio10a}. In Theorem~\ref{RAAGs main}, we will show that $T(A_{\G})$ acts transitively and with finite stabilisers on cubical coarse median structures with ``decomposable flats'' (Definition~\ref{decomposable flats defn}).

Theorem~\ref{RAAGs intro} is a consequence of this last result: when $A_{\G}$ is twistless, all cocompact cubulations have decomposable flats and $T(A_{\G})$ is trivial. At the opposite end of the spectrum, not all cubulations of $\Z^n$ have decomposable flats and, in fact, there are infinitely many $\Out(\Z^n)$--orbits of cubical coarse median structures on $\Z^n$.

\subsection{Right-angled Coxeter groups}

Recall that the right-angled Coxeter group $W_{\G}$ is the quotient of $A_{\G}$ by the normal closure of the set of squares of standard generators $\{v^2\mid v\in\G^{(0)}\}$.

On the one hand, right-angled Coxeter groups tend to behave more rigidly than general right-angled Artin groups because all their automorphisms  preserve the coarse median structure induced by the Davis complex \cite[Proposition~A(2)]{Fio10a}. On the other, Coxeter groups are granted some additional flexibility by the fact that their standard generators have finite order. One exotic example to keep in mind is the ``$\tfrac{\pi}{4}$--rotated'' action of the product of infinite dihedrals $D_{\infty}\x D_{\infty}$ on the square tiling of $\R^2$: each factor preserves one of the two lines through the origin forming an angle of $\tfrac{\pi}{4}$ or $\tfrac{3\pi}{4}$ with the coordinate axes, and each reflection axis is parallel to one of these two lines. Note that products of dihedrals can sometimes act in this way even within cubulations of larger, directly irreducible Coxeter groups (Example~\ref{exotic example}).

Our second main result provides two natural conditions that prevent these exotic behaviours from occurring. An action on a cube complex $G\acts X$ is \emph{strongly cellular} if, whenever an element $g\in G$ preserves a cube of $X$, it fixes it pointwise.

\begin{thmintro}\label{RACGs intro}
Let $W_{\G}\acts X$ be a cocompact cubulation of a right-angled Coxeter group. Suppose that it satisfies {\bf at least one} of the following:
\begin{enumerate}
\item the action is strongly cellular;
\item for all $x,y\in\G^{(0)}$, the subgroup $\langle x,y\rangle$ is convex-cocompact in $X$.
\end{enumerate}
Then $W_{\G}\acts X$ induces the same coarse median structure on $W_{\G}$ as the Davis complex.
\end{thmintro}

Item~(1) of Theorem~\ref{RACGs intro} holds more generally for all graph products of finite groups (Theorem~\ref{sc rigidity 2}). These groups also have a preferred cocompact cubulation due to Davis \cite{Davis-buildings}, which is always strongly cellular. We will refer to this cubulation as the \emph{graph-product complex} in this paper, to avoid confusion in the Coxeter case\footnote{Viewing a right-angled Coxeter group $W_{\G}$ as a graph product of order--$2$ groups, its ``graph-product complex'' is the first cubical subdivision of the ``usual Davis complex'' for $W_{\G}$, which is the $\CAT$ cube complex having the standard Cayley graph of $W_{\G}$ as its $1$--skeleton. Importantly, the $W_{\G}$--action on the Davis complex is \emph{not} strongly cellular, while the action on the graph-product complex is. We have good reason to work with both complexes associated with $W_{\G}$ in this paper, depending on the section, so we prefer not to call them both ``the Davis complex''.}. 

Item~(2) of Theorem~\ref{RACGs intro} has nontrivial content only when $x$ and $y$ are not adjacent. Item~(2) is the result requiring the most technical argument of the paper, which will occupy the entirety of Sections~\ref{nqc sect} and~\ref{std cms RACG sect}. Importantly, Item~(2) implies that cubical coarse median structures that are standard on maximal virtually-abelian subgroups of $W_{\G}$ must necessarily coincide with the one induced by the Davis complex. 

As a consequence of Theorem~\ref{RACGs intro}(2) and the cubical flat torus theorem \cite{WW}, we obtain many examples of right-angled Coxeter groups satisfying coarse cubical rigidity. Say that an induced square $\Delta\sq\G$ is \emph{bonded} if there exists another induced square $\Delta'\sq\G$ such that the intersection $\Delta\cap\Delta'$ has exactly $3$ vertices.

\begin{corintro}\label{loose cor intro}
If every induced square in $\G$ is bonded, then $W_{\G}$ satisfies coarse cubical rigidity.
%admits a unique cubical coarse median structure.
\end{corintro}

Bonded squares are important because they are precisely those squares $\Delta\sq\G$ for which the standard coarse median structure is forced upon the virtually abelian subgroup $W_{\Delta}\leq W_{\G}$ by its intersection pattern with highest abelian subgroups of $W_{\G}$. On the other hand, when $\Delta$ is not bonded, there are no obvious restrictions on the coarse median structure on $W_{\Delta}$. Thus, it would not be unreasonable to expect the reverse arrow in Corollary~\ref{loose cor intro} to also hold true, though we were unable to show this.

We emphasise that there are plenty of right-angled Coxeter groups that do not satisfy coarse cubical rigidity. A one-ended, irreducible example is provided by the left-hand graph in Figure~\ref{intro fig} (see Example~\ref{exotic example}). However, we do suspect that all right-angled Coxeter groups will satisfy a slightly weaker form of rigidity, namely the following.

\begin{conj*}
Let $W_{\G}$ be a right-angled Coxeter group.
\begin{enumerate}
\item[(a)] There are only finitely many cubical coarse median structures on $W_{\G}$.
\item[(b)] Each cubical coarse median structure on $W_{\G}$ is completely determined by its restriction to maximal parabolic virtually-abelian subgroups.
\end{enumerate}
\end{conj*}

In view of Theorem~\ref{RACGs intro}(2), Item~(b) of the Conjecture should look rather plausible. We show in Remark~\ref{conj rmk} that (b)$\Ra$(a).

\begin{figure}
    \centering
    \includegraphics[scale=.35]{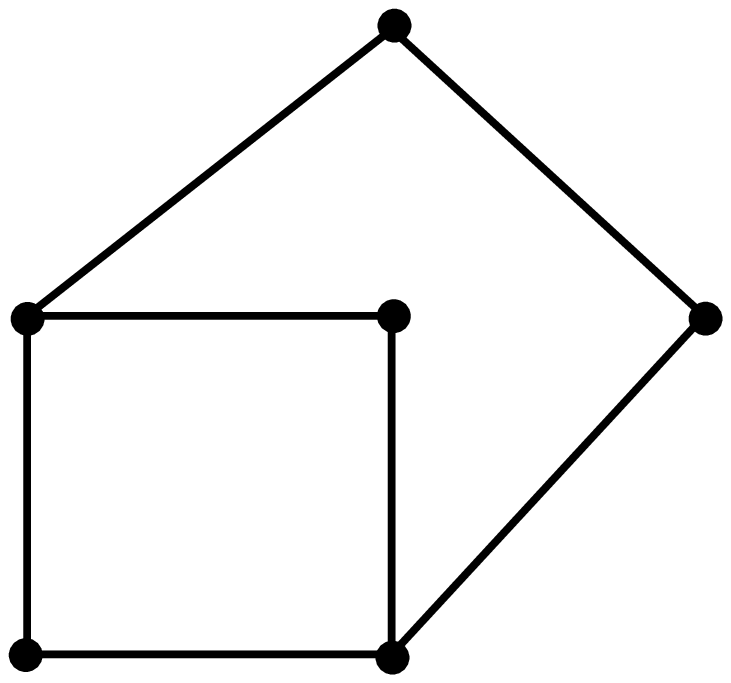}
    \put(-102,32){$\Lambda_1=$}
    \put(180,32){$=\Lambda_2$}
    \hspace{2.5cm}
    \includegraphics[scale=.35]{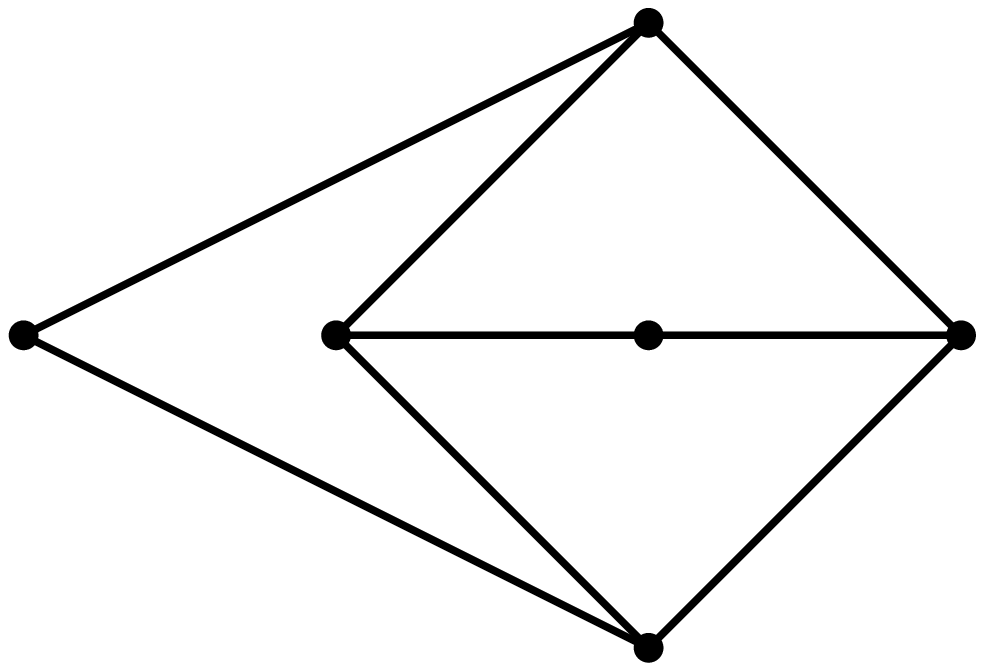}
    \caption{The right-angled Coxeter group $W_{\Lambda_2}$ satisfies coarse cubical rigidity, while $W_{\Lambda_1}$ does not. Every square in $\Lambda_2$ is bonded, while the square in $\Lambda_1$ is ``loose''.}
    \label{intro fig}
\end{figure}

\subsection{Consequences for automorphisms}

For a cocompactly cubulated group $G$, the outer automorphism group $\Out(G)$ naturally permutes cubical coarse median structures on $G$.

Automorphisms that preserve some cubical coarse median structure on $G$ enjoy exceptionally nice properties. For instance, their fixed subgroups are finitely generated, undistorted and cocompactly cubulated \cite[Theorem~B]{Fio10a}, properties that can drastically fail for general automorphisms. In addition, a cubical version of Nielsen realisation works in this context \cite[Corollary~H]{Fio10a}.

All automorphisms of right-angled Coxeter groups preserve a cubical coarse median structure (the one of the Davis complex), while this can fail for right-angled Artin groups. Thus, it is natural to wonder if this weaker form of cubical rigidity extends to two larger classes of groups: graph products of finite groups and general Coxeter groups. 

We show that this is indeed the case for these graph products and for those Coxeter groups $W$ whose Niblo--Reeves cubulation is cocompact \cite{Niblo-Reeves}. By \cite{Williams,Caprace-Muhlherr05}, the latter happens exactly when $W$ does not have any irreducible affine parabolic subgroups of rank $\geq 3$. Note that, when the Niblo--Reeves cubulation is not cocompact, it is not even known whether $W$ admits a coarse median structure (except in a few cases that turn out to be cocompactly cubulated for other reasons\footnote{For instance, the $(2,4,4)$ Coxeter group is cocompactly cubulated, as it coincides with the automorphism group of the standard square tiling of $\R^2$. However, its Niblo--Reeves cubulation is $4$--dimensional and it is not cocompact.}).

\begin{thmintro}\label{autom intro}
\begin{enumerate}
\item[]
\item Let $G$ be a graph product of finite groups. The coarse median structure on $G$ induced by the graph-product complex is fixed by $\Out(G)$.
\item Let $W$ be a Coxeter group with cocompact Niblo--Reeves cubulation. The coarse median structure on $W$ induced by the Niblo--Reeves cubulation is fixed by $\Out(W)$.
\end{enumerate}
\end{thmintro}

It is worth remarking that, if $W$ is not right-angled, the Niblo--Reeves cubulation is never strongly cellular, even after subdivisions. So the techniques used for Theorem~\ref{RACGs intro}(1) cannot be applied here.
We also mention that, when $W$ is \emph{$2$--spherical} (i.e.\ when there are no $\infty$ labels in its Coxeter graph), the group $\Out(W)$ is finite by \cite{FHM,Caprace-Muhlherr07}. However, this does not directly imply that these automorphisms fix a coarse median structure on $W$ without going through Theorem~\ref{autom intro}.

\subsection{A word on proofs}

Here we briefly sketch the proofs of Theorems~\ref{RAAGs intro} and~\ref{RACGs intro}.

Let $G_{\G}$ be a right-angled Artin/Coxeter group, let $G_{\G}\acts\X_{\G}$ be the standard action on the (universal cover of the) Salvetti/Davis complex, and let $G_{\G}\acts Y$ be an arbitrary cocompact cubulation. Theorem~\ref{RAAGs intro} and Theorem~\ref{RACGs intro}(1) can be quickly deduced from the following three ingredients:
\begin{enumerate}
\setlength\itemsep{.25em}
\item[(a)] To prove that $Y$ and $\X_{\G}$ induce the same coarse median structure, it suffices to show that hyperplane-stabilisers of $G\acts\X_{\G}$ are convex-cocompact in $Y$ (Theorem~\ref{coarse median vs hyperplanes}).
\item[(b)] Hyperplane-stabilisers of $G\acts\X_{\G}$ are centralisers of finite-order elements in the Coxeter case, and direct factors of centralisers of infinite-order elements in the Artin case.
\item[(c)] Centralisers of (sufficiently high powers of) convex-cocompact, infinite-order elements are always convex-cocompact (Lemma~\ref{cc properties}(4)). Centralisers of finite-order elements are convex-cocompact in all cocompact cubulations that are also strongly cellular (Lemma~\ref{universally cc}).
\end{enumerate}

The proof of Theorem~\ref{RACGs intro}(2) is significantly more involved. It relies on the following key facts about a general group $G$, though in a less obvious way.
\begin{enumerate}
\setlength\itemsep{.25em}
\item[(d)] Centralisers of finitely generated subgroups of $G$ are \emph{``median-cocompact''} in all cocompact cubulations of $G$ (Proposition~\ref{centralisers are median-cocompact}). That is, they act cofinitely on a median subalgebra of the cube complex. 
\item[(e)] If $G\acts X$ is a cocompact cubulation and $H<G$ is a median-cocompact subgroup that is not convex-cocompact, then there exists a combinatorial ray $r\sq X$ that stays uniformly close to an $H$--orbit $\mc{O}$, but such that convex hulls of all long subsegments of $r$ go very far from $\mc{O}$, and do so uniformly in their length (Theorem~\ref{uniformly non-qc ray}).
\item[(f)] Let $W_{\Delta}\leq W_{\G}$ be right-angled Coxeter groups such that $W_{\Delta}$ is irreducible and has finite centraliser in $W_{\G}$. If a quasi-geodesic $\alpha\sq W_{\Delta}$ spends uniformly bounded time near cosets of proper parabolic subgroups of $W_{\Delta}$, then $\alpha$ is Morse in $W_{\G}$ (Lemma~\ref{Morse path}).
\end{enumerate}

It is worth remarking that, using Chepoi--Roller duality, Fact~(d) has the following general consequence, which seems new and interesting. Rather surprisingly, we are not aware of any proofs of this result that do not rely on any abstract median-algebra or coarse-median techniques.

\begin{corintro}\label{cubulated centralisers intro}
Let $G$ be a cocompactly cubulated group. 
\begin{enumerate}
\item For every finitely generated subgroup $H\leq G$, the centraliser $Z_G(H)$ is cocompactly cubulated.
\item If $G=G_1\x G_2$ and $G_2$ has finite centre, then $G_1$ is cocompactly cubulated.
\end{enumerate}
\end{corintro}

Special cases of this result already appear in the literature. Item~(1) was shown by Haettel in \cite[Theorem~2.1]{Haettel} when $H$ is a sufficiently deep finite-index subgroup of an abelian subgroup of $G$, and by Genevois in \cite[Theorem~5.1]{Genevois-flattorus} when $H\cong\Z$. A special case of Item~(2) under stronger assumptions on $G_2$ follows from work of Kropholler--O'Donnell \cite[Theorem~A]{Kropholler-ODonnell}.

Regarding Item~(1), we emphasise that it is possible that no cubulation of $Z_G(H)$ can be realised as a \emph{convex} subcomplex of any cubulation of $G$. We give one such example in Example~\ref{never-cc centraliser}. Also note that Item~(2) in the corollary can fail if $G_2$ has \emph{infinite} centre: if $W$ is the $(3,3,3)$ Coxeter group, then the group $W\x\Z$ is cocompactly cubulated (Example~\ref{333xZ example}), while $W$ is not \cite{Hagen-cryst}.

\subsection{Structure of the paper}

Section~\ref{prelim sect} collects various basic facts on convex-cocompactness, median-cocompactness and (cubical) coarse median structures. Theorem~\ref{coarse median vs hyperplanes} is the most important result of the section, as it will be our main tool in all proofs of coarse cubical rigidity. We prove Corollary~\ref{cubulated centralisers intro} in Subsection~\ref{subsect:cms}.

Section~\ref{easy sect} is concerned with the simplest cases of coarse cubical rigidity. We prove Theorem~\ref{RACGs intro}(1) for general graph products of finite groups in Subsection~\ref{graph prod subsect}; Theorem~\ref{autom intro}(1) immediately follows from that.
In Subsection~\ref{vab subsect}, we study cubical coarse median structures on virtually abelian groups, also proving that Item~(b) in the Conjecture implies Item~(a). Finally, we prove Theorem~\ref{RAAGs intro} in Subsection~\ref{RAAGs subsect}.

Sections~\ref{nqc sect} and~\ref{std cms RACG sect} are devoted to the proof of Theorem~\ref{RACGs intro}(2). Section~\ref{nqc sect} is only concerned with median subalgebras of cube complexes, its main goal being Theorem~\ref{uniformly non-qc ray}, showing the existence of the ``uniformly non-quasi-convex'' rays mentioned in Item~(e) of the above sketch. Then Section~\ref{std cms RACG sect} applies this result to prove Theorem~\ref{RACGs intro}(2) and deduce Corollary~\ref{loose cor intro}.

Finally, Section~\ref{Coxeter sect} contains the proof of Theorem~\ref{autom intro}(2). Appendix~\ref{appendix} proves two basic lemmas about subsets of metric spaces with cocompact stabilisers.

\subsection{Acknowledgements}

We thank Jason Behrstock and Mark Hagen for interesting comments, Thomas Haettel for mentioning \cite{Haettel,Genevois-flattorus} in relation to Corollary~E in June 2023, and George Shaji for catching a couple of errors in an earlier version of this preprint. We are also grateful to the anonymous referees for their many helpful suggestions, particularly to one of them for proposing the terminology ``bonded square''.

\tableofcontents

%OTHER QUESTIONS
%\begin{quest}
%When is $\Out(W)$ infinite for a general Coxeter group? 
%\end{quest}
%
%\begin{quest}
%Are there general procedures to construct cubical coarse medians structures on RACGs other than the standard one?
%\end{quest}
%
%\begin{quest}
%Are cubical coarse median structures determined by their restriction to (highest) abelian subgroups? What happens if Flat Closing fails?
%\end{quest}

\addtocontents{toc}{\protect\setcounter{tocdepth}{2}}
\section{Preliminaries}\label{prelim sect}

\subsection{Cube complexes}

We will assume a certain familiarity with $\CAT$ cube complexes and group actions on them. For an introduction to the topic, the reader can consult \cite{Roller,CS,Wise-riches,Sageev-notes}. In this subsection, we only fix terminology and notation.

A \emph{cubulation} of a group $G$ is a proper action on a $\CAT$ cube complex. We say that $G$ is \emph{cocompactly cubulated} if it admits a cocompact cubulation. All actions on cube complexes are assumed to be cellular (i.e.\ by cubical automorphisms).

An action $G\acts X$ on a cube complex is \emph{strongly cellular} if the set of fixed points of each element of $G$ is a subcomplex of $X$ (this definition is equivalent to the one given in the Introduction). The action is \emph{non-transverse} if there do not exist a hyperplane $\mf{w}$ and element $g\in G$ such that $\mf{w}$ and $g\mf{w}$ are transverse; when $G$ acts cocompactly, this is equivalent to saying that $X$ $\ell_1$--embeds equivariantly in a finite product of trees. Non-transverse actions are strongly cellular, but the converse does not hold; for instance, since every free action is strongly cellular, one can consider the universal cover of any non-positively curved cube complex with a self-intersecting hyperplane, together with the action by deck transformations.

Every $\CAT$ cube complex admits two natural metrics. The \emph{$\ell_2$--metric} $d_2$ (usually known as the $\CAT$ metric) and the \emph{$\ell_1$--metric} $d_1$, namely the induced path metric obtained by equipping every cube with its $\ell_1$--metric. The restriction of $d_1$ to the $1$--skeleton is also known as the \emph{combinatorial metric}. Whenever we speak of ``geodesics'' without any prefixes, we will always refer to $\ell_1$--geodesics contained in the $1$--skeleton. Note however that, on a few occasions, it will be useful to consider $\ell_1$--geodesics outside the $1$--skeleton. For instance, we will need the following observation in Lemma~\ref{cms on vab}, though this is not involved in the proof of any of the main results of the paper. 

\begin{lem}\label{l_2 geodesics are l_1}
Let $X$ be a $\CAT$ cube complex. Every $\ell_2$--geodesic is an $\ell_1$--geodesic.
\end{lem}
\begin{proof}
Recall that $(X,d_1)$ is a finite-rank median space. Let $\alpha\sq X$ be an $\ell_2$--geodesic. If $\alpha$ were not an $\ell_1$--geodesic, it would cross a median wall $\mf{w}$ at least twice. Every median wall has the form $H\x\{t\}$, where $H$ is a hyperplane with carrier $H\x[-\tfrac{1}{2},\tfrac{1}{2}]$ and $t\in[-\tfrac{1}{2},\tfrac{1}{2}]$.
% this is true with no exceptions (recall that every halfspace is either open or closed)
In particular, $\mf{w}$ is $\ell_2$--convex, so it cannot be crossed twice by $\alpha$.
\end{proof}

Every $\CAT$ cube complex admits a median operator $m\colon X^3\ra X$. The median $m(x,y,z)$ is the only point lying on an $\ell_1$--geodesic between any two of the three points $x,y,z$. With the exception of the proof of Lemma~\ref{cms on vab}, we will only need the restriction of $m$ to the $0$--skeleton of $X$. If $x,y,z$ are vertices, then $m(x,y,z)$ is the only vertex with the property that no hyperplane of $X$ separates $m(x,y,z)$ from two of the vertices $x,y,z$. 

A subset $A\sq X$ is a \emph{median subalgebra} if $m(A,A,A)\sq A$. A full subcomplex $C\sq X$ is \emph{convex} if its $0$--skeleton satisfies the stronger property $m(C^{(0)},C^{(0)},X^{(0)})\sq C^{(0)}$. This notion is equivalent to convexity with respect to either the $\ell_1$-- or the $\ell_2$--metric \cite{Haglund-ss}. If $A\sq X^{(0)}$, we speak of the \emph{convex hull} of $A$, referring to the vertex set of the smallest convex subcomplex of $X$ containing $A$. We also speak of the median subalgebra \emph{generated} by $A$, referring to the smallest median subalgebra of $X$ containing $A$. We will need the following lemma stating that convex hulls and generated subalgebras can be constructed by taking medians a bounded number of times.

\begin{lem}\label{lem:bounded_iteration}
    Let $X$ be a CAT(0) cube complex of dimension $\delta$ and let $A\sq X^{(0)}$ be a subset.
    \begin{enumerate}
        \item The convex hull of $A$ is equal to the set $\mc{J}^{\delta}(A)$, where the operators $\mc{J}^i(\cdot)$ are defined inductively by $\mc{J}^0(A)=A$ and $\mc{J}^{i+1}(A)=m(\mc{J}^i(A),\mc{J}^i(A),X^{(0)})$.
        \item The median subalgebra generated by $A$ is equal to $\mc{M}^{2\delta}(A)$, where the operators $\mc{M}^i(\cdot)$ are defined inductively by $\mc{M}^0(A)=A$ and $\mc{M}^{i+1}(A)=m(\mc{M}^i(A),\mc{M}^i(A),\mc{M}^i(A))$.
    \end{enumerate}
\end{lem}

Part~(1) of the lemma is \cite[Lemma~5.5]{Bow-cm}. The bound mentioned in part~(2) was obtained in Bowditch's book \cite[Proposition~8.2.4]{Bow-book}; also see \cite[Proposition~4.1]{Bow-hulls} or \cite[Proposition~4.2]{Fio10a} for previous results of the same type.

\subsection{Convex-cocompactness}\label{subsec:cc}

Let $X$ be a $\CAT$ cube complex.

\begin{defn}
Let $H\acts X$ be an action. We say that $H$ is \emph{convex-cocompact in $X$} if there exists an $H$--invariant convex subcomplex $C\sq X$ that is acted upon cocompactly by $H$.
\end{defn}

We define the \emph{rank} of a finitely generated, virtually abelian group as the rank of its finite-index free abelian subgroups. If $G$ is a group, we say that a finitely generated, virtually abelian subgroup $A\leq G$ is \emph{highest} if it is not virtually contained in a free abelian subgroup of strictly higher rank. 

The following lemma collects the main properties of convex-cocompactness from \cite{WW}, \cite{Fio10a} and \cite{Fio10e}.

\begin{lem}\label{cc properties}
Let $G\acts X$ be a cocompact cubulation.
\begin{enumerate}
\item Finite intersections of convex-cocompact subgroups of $G$ are convex-cocompact in $X$. 
\item Convex-cocompactness in $X$ is a commensurability invariant for subgroups of $G$.
\item All highest virtually abelian subgroups of $G$ are convex-cocompact in $X$.
\item If $H\leq G$ is convex-cocompact in $X$ and the action $H\acts X$ is non-transverse, then the normaliser $N_G(H)$ is convex-cocompact in $X$.
% Using \cite[Corollary~2.12(1)]{Fio10e}, one could also prove that Z_G(H) is convex-cocompact, *IF* we knew that the centre of a convex-cocompact subgroup is always convex-cocompact. This should certainly be true, but looks a little tricky to prove (and we don't need it). Maybe it can be deduced from the cubical flat torus theorem, but it's not so clear that the centre is a finite intersection of highest abelian subgroups.
\end{enumerate}
\end{lem}
\begin{proof}
Part~(1) is \cite[Lemma~2.7]{Fio10e}. Part~(2) is immediate from the equivalence between convex-cocompactness and coarse median quasi-convexity (Proposition~\ref{prop:cc=qc} below). Part~(3) follows from the cubical flat torus theorem \cite[Theorem~3.6]{WW}. Part~(4) is \cite[Corollary~2.12(3)]{Fio10e}.
\end{proof}

We caution the reader that the non-transversality assumption really is necessary for Lemma~\ref{cc properties}(4) to hold; see \cite[Example~2.11]{Fio10e}.

\begin{rmk}\label{rmk:cc_products}
Consider a product of $\CAT$ cube complexes $X_1\x X_2$ and an element $g\in\Aut(X_1\x X_2)$ preserving each factor of the splitting. Then $\langle g\rangle$ is convex-cocompact in $X_1\x X_2$ if and only if it is elliptic in one of $X_1,X_2$ and convex-cocompact in the other. 

Indeed, every $\langle g\rangle$--invariant convex subcomplex $C\sq X_1\x X_2$ splits as a product $C_1\x C_2$ of $\langle g\rangle$--invariant convex subcomplexes $C_i\sq X_i$. If $g$ were loxodromic in both $X_i$, then both $C_i$ would contain bi-infinite geodesics, hence $C$ would contain a copy of $\R^2$, barring cocompactness.
\end{rmk}

The following technical lemma is an important ingredient in the proof of Theorems~\ref{RAAGs intro} and~\ref{RACGs intro}(2).

\begin{lem}\label{lem:cc_products}
Consider a cocompact cubulation $G\acts X$ satisfying all of the following conditions:
\begin{itemize}
\item $G=H\x K$ and $H\simeq\Z$;
\item $H$ is convex-cocompact in $X$;
\item $K$ is generated by $k_1,\dots,k_n$, where each $\langle k_i\rangle$ is convex-cocompact in $X$.
\end{itemize}
Then $K$ is convex-cocompact in $X$.
\end{lem}
\begin{proof}
In the Caprace--Sageev terminology \cite[Section~3.3]{CS}, every hyperplane of $X$ is either $H$--essential, $H$--half-essential, or $H$--trivial. Of the two halfspaces of $X$ associated with an $H$--half-essential hyperplane, one is $H$--deep and the other is $H$--shallow. 

Since $H$ is normal in $G$, the set of $H$--deep halfspaces bounded by $H$--half-essential hyperplanes is $G$--invariant. The intersection $Y$ of all these halfspaces is nonempty, as it contains any axis for $H\simeq\Z$. Thus, $Y\sq X$ is a $G$--invariant convex subcomplex on which $G$ acts properly and cocompactly. Replacing $X$ with $Y$, we can safely assume that there are no $H$--half-essential hyperplanes.

Now, observing that every $H$--essential hyperplane is transverse to every $H$--trivial hyperplane, we obtain a splitting of $X$ as a product of cube complexes $X_1\x X_0$ \cite[Lem\-ma~2.5]{CS}. Every hyperplane of $X_1$ is $H$--essential, while every hyperplane of $X_0$ is $H$--trivial. Since $G$ normalises $H$, it takes $H$--essential hyperplanes to $H$--essential hyperplanes. It follows that the $G$--action respects the splitting $X=X_1\x X_0$. 

We only need three observations to conclude the proof.

\smallskip
{\bf Claim~1:} \emph{$X_1$ is a quasi-line.}

\smallskip
Since all hyperplanes of $X_1$ are $H$--essential, every $H$--invariant convex subcomplex of $X$ is a union of fibres $X_1\x\{\ast\}$. Thus, since $H$ is convex-cocompact in $X$, it must act cocompactly on $X_1$. Since all hyperplanes of $X_0$ are $H$--trivial, the action $H\acts X_0$ has a global fixed point (e.g.\ by \cite[Proposition~B.8]{CFI}), hence $H$ must act properly on $X_1$. In conclusion, the action $H\acts X_1$ is proper and cocompact, proving that $X_1$ is a quasi-line.

\smallskip
{\bf Claim~2:} \emph{there exists a $G$--invariant $\ell_2$--geodesic line $\alpha\sq X_1$.}

\smallskip
Indeed, the union of all $H$--axes in $X_1$ (for the $\ell_2$--metric) splits as a product of $\CAT$ spaces $\R\x Z$ (e.g.\ by \cite[Theorem~II.7.1]{BH}). This subset, as well as its splitting, is $G$--invariant. Since $X_1$ is a quasi-line, the $\CAT$ space $Z$ is bounded, hence the induced $G$--action on $Z$ has a global fixed point $z$. The fibre $\R\x\{z\}$ is the required axis $\alpha$.

\smallskip
{\bf Claim~3:} \emph{each generator $k_i\in K$ is elliptic in $X_1$.}

\smallskip
This is clear if $k_i$ has finite order; suppose instead that $k_i$ has infinite order. Note that $k_i$ must be elliptic in either $X_0$ or $X_1$ by Remark~\ref{rmk:cc_products}, since $\langle k_i\rangle$ is convex-cocompact in $X$ by hypothesis. However, no infinite-order $k_i$ can be elliptic in $X_0$. Otherwise, since $H$ is elliptic in $X_0$, the subgroup $H\x\langle k_i\rangle$ would also be elliptic in $X_0$. This would force $H\x\langle k_i\rangle$ to act properly on $X_1$, which cannot happen, since $H\x\langle k_i\rangle\simeq\Z^2$ while $X_1$ is a quasi-line. 

\smallskip
We now finish the proof using the above claims. Since each $k_i$ commutes with $H$, preserves the $H$--axis $\alpha$, and is elliptic in $X_1$, we conclude that each $k_i$ fixes $\alpha$ pointwise. This implies that the entire $K$ is elliptic in $X_1$. Since $G=H\x K$ acts cocompactly on $X$, we deduce that $K$ must act cocompactly on $X_0$. Since $X_0$ embeds $K$--equivariantly in $X$ as a convex subcomplex, this shows that $K$ is convex-cocompact in $X$.
\end{proof}

We emphasise that it really is important in Lemma~\ref{lem:cc_products} that $K$ be generated by convex-cocompact elements. Indeed, any epimorphism $\alpha\colon K\ra\Z$ gives an automorphism of $G=H\x K$ fixing $H\cong\Z$ pointwise and taking $K$ to the subgroup $K'=\{(\alpha(k),k)\mid k\in K\}$. So $G$ can have many product splittings $G=H\x K'$ where $K'$ is a subgroup isomorphic to $K$, though only one of these subgroups will be convex-cocompact in $X$.

The following lemma is needed in the proof of Theorem~\ref{RACGs intro}(2).

\begin{lem}\label{cc products}
Let $X$ be a finite-dimensional CAT(0) cube complex. Consider an action $G\acts X$ where $G=H\x K$ is finitely generated. If $H$ and $K$ are convex-cocompact in $X$, then so is $G$.
\end{lem}
\begin{proof}
We will only use this result in the case when $X$ is locally finite and $G$ acts properly, where one could actually give a more direct proof using \cite{CS} instead of \cite{Fio10b}. Nevertheless, we choose to prove the general statement, as this will likely be useful elsewhere.

Up to subdividing $X$, we can assume that $G$ acts without hyperplane-inversions. By Theorem~3.16 and Remark~3.17 in \cite{Fio10b}, the reduced core $\overline{\mc{C}}(H,X)$ is a nonempty, $G$--invariant, convex subcomplex of $X$. In addition, there is a $G$--invariant splitting $\overline{\mc{C}}(H,X)=C_1\x C_0$, where every hyperplane of $C_1$ is skewered by an element of $H$, while the action $H\acts C_0$ factors through a finite group (see Lemma~3.22, Corollary~4.6 and Remark~3.9 in \cite{Fio10b}). The fact that $H\acts C_0$ factors through a finite group, rather than simply being elliptic, is the key difference from \cite{CS}.

If $H$ is convex-cocompact in $X$, the action $H\acts C_1$ is cocompact. Indeed, if $Y\sq X$ is an $H$--cocompact, convex subcomplex, we can project $Y$ first to $\overline{\mc{C}}(H,X)$ and then to $C_1$. The result is an $H$--cocompact, convex subcomplex $Y'\sq C_1$, which must be the entire $C_1$, since every hyperplane of $C_1$ is $H$--skewered. Thus, $H\acts C_1$ is cocompact.

Since the splitting $C_1\x C_0$ is $G$--invariant, we have an action $G\acts C_0$. With another application of the above results from \cite{Fio10b}, we obtain a $G$--invariant splitting $\overline{\mc{C}}(K,C_0)=C_{01}\x C_{00}$. Here $C_{01}\sq C_0$ is a $G$--invariant, convex subcomplex all whose hyperplanes are skewered by elements of $K$. Again, if $K$ is convex-cocompact in $X$, the action $K\acts C_{01}$ is cocompact.

In conclusion, we have found a $G$--invariant, convex subcomplex $C_1\x C_{01}\sq X$ such that the actions $H\acts C_1$ and $K\acts C_{01}$ are cocompact, and the action $H\acts C_{01}$ factors through a finite group. 

Now, let $A\sq C_1$ and $B\sq C_{01}$ be compact fundamental domains for the actions of $H$ and $K$, respectively. Denote by $H\cdot B$ the union of all $H$--translates of $B$; this is again a compact set, since the action $H\acts C_{01}$ factors through a finite group. Finally, the set $A\x (H\cdot B)$ is a compact fundamental domain for the action $G\acts C_1\x C_{01}$, showing that $G$ is convex-cocompact in $X$.
\end{proof}

\begin{rmk}
    Although we will not be needing this, it is interesting to note that the proof of Lemma~\ref{cc products} is actually showing a bit more. The group $G$ acts cocompactly on a convex subcomplex of $X$ with a $G$--invariant product splitting $C=C_1\x C_2$ such that the actions $H\acts C_1$ and $K\acts C_2$ are cocompact and essential, while the action $H\acts C_2$ factors through a finite group. Without stronger assumptions we cannot say more, for instance the action $K\acts C_1$ needs not even be elliptic (the lack of symmetry between $H$ and $K$ is due to the fact that the splitting of $C$ is not canonical). Indeed, we might well have $H\cong K$ and the $G$--action might be given by a homomorphism $\rho\colon H\x H\ra\Aut(X)$ with $\rho(h,1)=\rho(1,h)$ for all $h\in H$, in which case $K\acts C_1$ is essential and $C_2$ is a single point (not having any hyperplanes, $C_2$ is nevertheless $K$--essential as well).
\end{rmk}

\subsection{Median-cocompactness}\label{median cocpt subsect}

Let $X$ be a $\CAT$ cube complex. 

In some cases, one can prove convex-cocompactness of a subgroup $H\leq\Aut(X)$ by first establishing a significantly weaker property --- \emph{median-cocompactness} --- and then promoting this to genuine convex-cocompactness.

\begin{defn}
Let $H\acts X$ be an action. We say that $H$ is \emph{median-cocompact in $X$} if there exists an $H$--invariant median subalgebra $M\sq X^{(0)}$ that is acted upon cofinitely by $H$.
\end{defn}

As an example, if $X$ is the standard square tiling of $\R^2$, the diagonal $\{(n,n)\mid n\in\Z\}$ is a median subalgebra, although the only convex subcomplex containing it is the whole $X$. Thus, considering the standard action $\Z^2\acts\R^2$, the diagonal subgroup $\langle(1,1)\rangle\simeq\Z$ is median-cocompact, but not convex-cocompact.

If $G\acts X$ is a cocompact cubulation and $H\leq G$ is a median-cocompact subgroup, then $H$ is necessarily finitely generated, undistorted and cocompactly cubulated \cite[Lemma~4.12]{Fio10a}. However, the cubulation of $H$ is ``abstractly'' provided by Chepoi--Roller duality and it cannot be realised as a convex subcomplex of $X$ in general.

A fundamental observation is that median-cocompactness always comes for free for \emph{centralisers}, while this is not true of convex-cocompactness without strong additional assumptions (compare Lemma~\ref{cc properties}(4) above and Lemma~\ref{universally cc} below).

\begin{prop}\label{centralisers are median-cocompact}
Let $G\acts X$ be a cocompact cubulation. For every finitely generated subgroup $H\leq G$, the centraliser $Z_G(H)$ is median-cocompact in $X$.
\end{prop}

The proof of this result is not hard, but it is best phrased in terms of coarse median structures, so we postpone it until the next subsection (Lemma~\ref{Z in general cms}). Corollary~\ref{cubulated centralisers intro} will quickly follow from Proposition~\ref{centralisers are median-cocompact}, as we also explain in the next subsection.

\subsection{Cubical coarse median structures}\label{subsect:cms}

Let $G$ be a finitely generated group equipped with a word metric. The specific choice of word metric is inconsequential in the following discussion.

A \emph{coarse median} is a map $\mu\colon G^3\ra G$ for which there exists a constant $C\geq 0$ such that, for all $a,b,c,x\in G$, the following hold:
\begin{enumerate}
\item $\mu(a,a,b)=a$ and $\mu(a,b,c)=\mu(b,c,a)=\mu(b,a,c)$;
\item $\mu(\mu(a,x,b),x,c)$ and $\mu(a,x,\mu(b,x,c))$ are at distance $\leq C$;
\item $d(\mu(a,b,c),\mu(x,b,c))\leq Cd(a,x)+C$;
\item $x\mu(a,b,c)$ and $\mu(xa,xb,xc)$ are at distance $\leq C$.
\end{enumerate}
Coarse medians were introduced by Bowditch \cite{Bow-cm}, while the above formulation is due to Niblo--Wright--Zhang \cite{NWZ1}. Condition~(4) is sometimes omitted, but it is important for our purposes.

Two coarse medians $\mu,\nu$ are \emph{at bounded distance} from each other if the points $\mu(a,b,c)$ and $\nu(a,b,c)$ are at uniformly bounded distance for $a,b,c\in G$. A \emph{coarse median structure} on $G$ is an equivalence class $[\mu]$ of coarse medians pairwise at bounded distance.

Every cocompact cubulation $G\acts X$ induces a specific coarse median structure $[\mu]$ on $G$. This consists of the maps $\mu\colon G^3\ra G$ for which there exists a constant $C$ such that:
\[ d(\mu(g_1,g_2,g_3)v,m(g_1v,g_2v,g_3v))\leq C,\ \ \forall g_1,g_2,g_3\in G,\ \ \forall v\in X^{(0)}.\]
That such maps $\mu$ exist is a straightforward consequence of the Milnor--Schwarz lemma. Coarse median structures arising this way are much better-behaved than arbitrary ones, so they deserve a special name. 

\begin{defn}
A coarse median structure on $G$ is \emph{cubical} if induced by a cocompact cubulation. Let $\mc{CM}(G)$ be the set of coarse median structures on $G$, and $\mc{CM}_{\square}(G)$ the subset of cubical ones.
\end{defn}

If $[\mu]\in\mc{CM}(G)$ and $\varphi\in\Aut(G)$, then the map 
\[(a,b,c)\mapsto \varphi(\mu(\varphi^{-1}(a),\varphi^{-1}(b),\varphi^{-1}(c)))\] 
is also a coarse median. If $\varphi$ is an inner automorphism, then this procedure does not alter the coarse median structure $[\mu]$. For this to hold, it is important that we included condition~(4) in the above definition of coarse median, and also that we identify coarse medians at bounded distance from each other (indeed, conjugation by an element $g\in G$ displaces the actual coarse median $\mu$ by a linear function of the word length of $g$). In conclusion, the group $\Aut(G)$ acts on the set of coarse medians on $G$, and the subgroup of inner automorphisms fixes pointwise the quotient of coarse median structures $\mc{CM}(G)$. Thus, we obtain a natural $\Out(G)$--action on $\mc{CM}(G)$ and $\mc{CM}_{\square}(G)$. 

\begin{ex}
\begin{enumerate}
\item[]
\item If $G$ is word-hyperbolic, $\mc{CM}(G)$ is a singleton by \cite[Theorem~4.2]{NWZ1}. In particular, all cocompact cubulations of $G$ induce the same (cubical) coarse median structure.
\item Instead, $\mc{CM}_{\square}(\Z^2)$ is countably infinite (see Proposition~\ref{cms on Z^n}) and $\mc{CM}(\Z^2)$ is uncountable.
\end{enumerate}
\end{ex}

A fundamental fact is that the information of which subgroups of $G$ are convex-cocompact or median-cocompact in a given cubulation is entirely encoded in the induced coarse median structure.

\begin{defn}\label{qc defn}
Consider $[\mu]\in\mc{CM}(G)$ and a subgroup $H\leq G$.
\begin{enumerate}
\item $H$ is \emph{$[\mu]$--quasi-convex} if $\mu(H,H,G)$ is at finite Hausdorff distance from $H$.
\item $H$ is \emph{$[\mu]$--quasi-submedian} if $\mu(H,H,H)$ is at finite Hausdorff distance from $H$.
\end{enumerate}
\end{defn}

When the coarse median structure is understood, we will sometimes simply speak of \emph{quasi-convex} or \emph{quasi-submedian} subgroups. We emphasise that there is no connection between quasi-submedian subgroups and the ``quasi-median graphs'' studied e.g.\ in \cite{Bandelt-Mulder-Wilkeit,CCHO,Genevois-PhD}.

\begin{prop}\label{prop:cc=qc}
Let $G\acts X$ be a cocompact cubulation, inducing $[\mu_X]\in\mc{CM}_{\square}(G)$. Then: 
\begin{enumerate}
\item $H\leq G$ is convex-cocompact in $X$ if and only if it is $[\mu_X]$--quasi-convex;
\item $H\leq G$ is median-cocompact in $X$ if and only if it is $[\mu_X]$--quasi-submedian.
\end{enumerate}
\end{prop}
\begin{proof}
Both parts follow from Lemma~\ref{lem:bounded_iteration}, which guarantees that both the convex hull of an $H$--orbit in $X$ and the median subalgebra generated by it can be constructed by taking medians a bounded number of times. More details on the argument can be found e.g.\ in \cite[Lemma~3.2]{Fio10a} or in the proof of \cite[Theorem~4.10]{Fio10a}. (Note that, in \cite{Fio10a}, quasi-submedian subgroups are referred to as ``approximate median subalgebras''.)
\end{proof}

In view of Proposition~\ref{prop:cc=qc}, the fact that centralisers of finitely generated subgroups are median-cocompact in all cocompact cubulations (Proposition~\ref{centralisers are median-cocompact}) becomes an immediate consequence of the following observation about general coarse median structures.

\begin{lem}\label{Z in general cms}
If $[\mu]\in\mc{CM}(G)$ and $H\leq G$ is finitely generated, then $Z_G(H)$ is $[\mu]$--quasi-submedian.
\end{lem}
\begin{proof}
For simplicity, if $x,y\in G$ are at distance $\leq D$ in the chosen word metric, we write $x\approx_D y$. We also write $|x|$ for the word length of $x$. Finally, let $C$ be the constant for which $\mu$ satisfies the four conditions in the definition of coarse medians at the beginning of this subsection.

Set $Z:=Z_G(H)$. Consider a point $x\in\mu(Z,Z,Z)$, say $x=\mu(z_1,z_2,z_3)$ with $z_1,z_2,z_3\in Z$. If $h\in H$, using Conditions~(4) and~(3) in this order, we obtain:
\[ hx=h\mu(z_1,z_2,z_3)\approx_C\mu(hz_1,hz_2,hz_3)=\mu(z_1h,z_2h,z_3h)\approx_{3C(1+|h|)}\mu(z_1,z_2,z_3)=x.\]
It follows that $x^{-1}hx$ lies in the finite subset of $G$ with word length at most $4C+3C|h|$. Hence $x$ lies in a finite union of right cosets of $Z_G(h)$, call it $R(h)$.

Since $x$ was arbitrary, we have $\mu(Z,Z,Z)\sq R(h)$ for every $h\in H$. If $h_1,\dots,h_n$ are a finite generating set for $H$, the intersection $R(h_1)\cap\dots\cap R(h_n)$ is a finite union of right cosets of $Z_G(h_1)\cap\dots\cap Z_G(h_n)=Z$, and it contains $\mu(Z,Z,Z)$. This shows that $\mu(Z,Z,Z)$ is contained in a bounded neighbourhood of $Z$. Since $\mu(Z,Z,Z)$ trivially contains $Z$, it is at finite Hausdorff distance from $Z$, proving the lemma.
\end{proof}

Now that the proof of Proposition~\ref{centralisers are median-cocompact} is completed, we take the chance to quickly deduce Corollary~\ref{cubulated centralisers intro} from it.

\begin{proof}[Proof of Corollary~\ref{cubulated centralisers intro}]
Let $G\acts X$ be a cocompact cubulation. If $H\leq G$ is finitely generated, Proposition~\ref{centralisers are median-cocompact} guarantees that $Z_G(H)$ acts cofinitely on a median subalgebra $M\sq X^{(0)}$. By Chepoi--Roller duality, there exists an action on a $\CAT$ cube complex $Z_G(H)\acts Y$ such that $M$ and $Y^{(0)}$ are equivariantly isomorphic as median algebras. It follows that the action $Z_G(H)\acts Y$ is a cocompact cubulation. 
% Properness of the action / local finiteness of $Y$ might look worrying, but note that $M$ can be taken edge-connected and then we're fine, because we're a (non-convex) subcomplex of $X$. Alternatively see the argument in \cite[Lemma~4.12]{Fio10a}.
This proves part~(1) of the corollary.

If $G=G_1\x G_2$ and $G_2$ has finite centre, then $G_1$ is a finite-index subgroup of the centraliser $Z_G(G_2)$, which is cocompactly cubulated by part~(1). Hence $G_1$ is also cocompactly cubulated, proving part~(2).
\end{proof}

The proof of Corollary~\ref{cubulated centralisers intro} also shows that, if a product $G_1\x G_2$ is cocompactly cubulated, then $G_1\x\Z^n$ is cocompactly cubulated for some $n\geq 1$. In general, however, this does not imply that $G_1$ is itself cocompactly cubulated, as the next example demonstrates.

\begin{ex}\label{333xZ example}
Consider the $(3,3,3)$ Coxeter group $W$, which is not cocompactly cubulated by \cite{Hagen-cryst}. Let us see that, instead, $W\x\Z$ acts properly and cocompactly on the cubical tiling of $\R^3$.

Consider the elements $\rho,\s\in{\rm O}(3,\Z)$ and the translations $T_0,T_1$ defined as follows:
\begin{align*}
\rho(x,y,z)&=(x,z,y), & T_0(x,y,z)&=(x+1,y+1,z+1), \\
\s(x,y,z)&=(y,z,x), & T_1(x,y,z)&=(x-2,y+1,z+1).
\end{align*}
All four elements preserve the standard cubical tiling of $\R^3$, and $T_0$ commutes with $\rho,\s,T_1$. The plane $P=\{(x,y,z) \mid x+y+z=0\}$ is orthogonal to the axes of $T_0$ and it is preserved by the reflection $\rho$ and the rotation $\s$. In particular,  $\s|_P$ is an order--$3$ rotation around the origin and $\rho|_P$ is a reflection in a line through the origin that is also an axis for $T_1$. This shows that $\langle\rho,\s\rangle$ is an order--$6$ dihedral group, and the reflection axes of the restrictions to $P$ of $\rho$, $\s\rho\s^{-1}$, $(T_1\s^2)\rho(T_1\s^2)^{-1}$ intersect forming an equilateral triangle.

In conclusion, $\langle \rho,\s\rho\s^{-1},(T_1\s^2)\rho(T_1\s^2)^{-1}\rangle$ is a copy of $W$ acting properly and cocompactly on the plane $P$. This group commutes with $T_1$, which translates perpendicularly to $P$, providing the required geometric action $W\x\Z\acts\R^3$. 

We emphasise that, although this example is probably well-known, the action $W\acts\R^3$ considered here is \emph{not} the Niblo--Reeves cubulation of $W$. The latter is also an action of $W$ on the cubical tiling of $\R^3$ preserving the same exact plane $P$, but the two actions differ in the following key way.
\begin{itemize}
    \item In the current example, each order--$2$ element of $W$ acts as a reflection in a plane orthogonal to $P$. In particular, the $W$--action preserves each of the two connected components of $\R^3-P$.
    \item In the Niblo--Reeves cubulation, each order--$2$ element of $W$ acts as a rotation by $\pi$ around a line contained in $P$. Each of these elements swaps the two connected components of $\R^3-P$.  
\end{itemize}
As a consequence, the Niblo--Reeves cubulation $W\acts\R^3$ \emph{cannot} be extended to an action $W\x\Z\acts\R^3$: in order to commute with all order--$2$ elements of $W$, the $\Z$--factor would have to translate in a direction that is parallel to all their fixed lines, but no such direction exists.
%$W$ is generated by reflections in planes orthogonal to $P$ in the current example, while the same elements of $W$ do not act as reflections in the Niblo--Reeves cubulation: they are rather rotations of $\pi$ around fixed lines orthogonal to $P$. Each of these rotations actually swaps the two sides of the plane $P$, while in the current example the $W$--action preserves both sides of $P$. In particular, the Niblo--Reeves cubulation $W\acts\R^3$ cannot be extended to an action $W\x\Z\acts\R^3$: the $\Z$--factor 
\end{ex}

We now turn to an important tool to check that two cubulations of a group $G$ induce the same coarse median structure, namely Theorem~\ref{coarse median vs hyperplanes} below. Before stating the theorem, we need to recall a classical construction. 

If $G\acts X$ is an action on a $\CAT$ cube complex and $\mc{U}\sq\mscr{W}(X)$ is a $G$--invariant set of hyperplanes, we can collapse the hyperplanes in $\mc{U}$ and obtain a new action on a $\CAT$ cube complex $G\acts X(\mc{U})$, along with a $G$--equivariant, surjective, cubical map $\pi_{\mc{U}}\colon X\ra X(\mc{U})$. The action $G\acts X(\mc{U})$ is known as a \emph{restriction quotient} of $G\acts X$, see \cite[Section~2.3]{CS} for details. Note that restriction quotients of cocompact actions will remain cocompact, but restriction quotients of proper actions may well stop being proper.

A map $f\colon M\ra N$ between median algebras (in our case, subalgebras of cube complexes) is a \emph{median morphism} if, for all $x,y,z\in M$, we have $f(m(x,y,z))=m(f(x),f(y),f(z))$. Consider two actions $G\acts X$ and $G\acts Y$ on $\CAT$ cube complexes. As observed in \cite[Proposition~2.20]{Fio10a} (crucially exploiting \cite[Theorem~4.4]{Huang-Kleiner}), a $G$--equivariant surjective cubical map $f\colon X\ra Y$ corresponds to a restriction quotient if and only if $f$ is a median morphism on the $0$--skeleton. In particular, if a cocompact cubulation is a restriction quotient of another cocompact cubulation, then the two induced coarse median structures on $G$ will coincide.

The next result will be our main tool to show that two cubulations of $G$ induce the same coarse median structure. It reduces the problem to checking that finitely many subgroups are convex-cocompact. Note that Item~(3) in Theorem~\ref{coarse median vs hyperplanes} is deliberately asymmetric in $X$ and $Y$.

\begin{thm}\label{coarse median vs hyperplanes}
Let $G\acts X$ and $G\acts Y$ be cocompact cubulations. The following are equivalent:
\begin{enumerate}
\item $G\acts X$ and $G\acts Y$ induce the same coarse median structure on $G$;
\item $G\acts X$ and $G\acts Y$ have the same convex-cocompact subgroups;
\item every $G$--stabiliser of a hyperplane of $X$ is convex-cocompact in $Y$;
\item there exists a third cocompact cubulation $G\acts Z$ of which both $G\acts X$ and $G\acts Y$ are restriction quotients.
\end{enumerate}
\end{thm}
\begin{proof}
The implication (1)$\Ra$(2) follows from Proposition~\ref{prop:cc=qc}, while the implication (2)$\Ra$(3) is obvious. The implication (4)$\Ra$(1) is also clear, since then $Z$ and $X$ induce the same coarse median structure on $G$, and so do $Z$ and $Y$.

We are left to prove that (3)$\Ra$(4), which is the main content of the theorem. Suppose that all stabilisers of hyperplanes of $X$ are convex-cocompact in $Y$. By \cite[Proposition~7.9]{Fio10e}, there exists a nonempty, $G$--invariant, $G$--cofinite median subalgebra $N\sq X^{(0)}\x Y^{(0)}$.

Let $K_X\sq X$ and $K_Y\sq Y$ be finite subcomplexes with $G\cdot K_X=X$ and $G\cdot K_Y=Y$. Consider the diagonal action $G\acts X\x Y$ and the finite subcomplex $K:=K_X\x K_Y\sq X\x Y$. 

Let us show that there exists a $G$--invariant, $G$--cofinite median subalgebra $M\sq X\x Y$ such that $K\sq M$. To begin with, since $K$ is compact, the set $G\cdot K$ is at finite Hausdorff distance from $N$. Since the median operator $m$ is $1$--Lipschitz, it follows that the set of medians $m(G\cdot K,G\cdot K,G\cdot K)$ is at finite Hausdorff distance from $m(N,N,N)=N$. Since the median subalgebra $M$ generated by $G\cdot K$ is obtained by taking medians a bounded number of times (Lemma~\ref{lem:bounded_iteration}), this implies that $M$ is at finite Hausdorff distance from $N$ too. It is clear that $M$ is $G$--invariant and contains $K$. Finally, since $X\x Y$ is locally finite and $N$ is $G$--cofinite, $M$ is also $G$--cofinite.

Now, by Chepoi--Roller duality, the median algebra $M\cap (X^{(0)}\x Y^{(0)})$ is the $0$--skeleton of a $\CAT$ cube complex $Z$ equipped with a proper and cocompact action $G\acts Z$. The restriction to $M$ of the coordinate projection $X\x Y\ra X$ gives a $G$--equivariant cubical map $\pi_X\colon Z\ra X$ that is a median morphism on the $0$--skeleton; in addition, since $K\sq M$ and $G\cdot K_X=X$, the map $\pi_X$ is surjective. As mentioned before the statement of the theorem, this implies that $\pi_X$ corresponds to a restriction quotient. The same argument shows that $G\acts Y$ is a restriction quotient of $G\acts Z$, concluding the proof. 
\end{proof}

\section{First forms of coarse cubical rigidity}\label{easy sect}

In this section, we prove Theorem~\ref{RAAGs intro} (Theorem~\ref{RAAGs main}) and a general version of Theorem~\ref{RACGs intro}(1) for all graph products of finite groups (Theorem~\ref{sc rigidity 2}). Theorem~\ref{autom intro}(1) immediately follows from the latter. In addition, in Subsection~\ref{vab subsect} we study cubical coarse medians on virtually abelian groups.

\subsection{Strongly cellular cubulations}

In this subsection, we obtain a criterion guaranteeing that all strongly cellular, cocompact cubulations of certain groups induce the same coarse median structure (Proposition~\ref{centraliser hyp-stab}). 

We first need the following simple lemma. This should be compared to Proposition~\ref{centralisers are median-cocompact} stating that centralisers of finite sets are median-cocompact in all cocompact cubulations. Here we exploit stronger assumptions on the centraliser and the action to obtain the stronger conclusion of convex-cocompactness.

\begin{lem}\label{universally cc}
Let $G$ be a group and let $Z\leq G$ be the centraliser of a finite set of finite-order elements. Then $Z$ is convex-cocompact in every strongly cellular, cocompact cubulation of $G$.
\end{lem}
\begin{proof}
Let $G\acts Y$ be a strongly cellular, cocompact cubulation of $G$. Let $f\in G$ have finite order.

The subset $\Fix(f)\sq Y$ is nonempty and, since the $\ell_2$--metric on $Y$ is uniquely geodesic, it is $\ell_2$--convex. By definition of strongly cellular action, the subset $\Fix(f)$ is also a subcomplex. We conclude that $\Fix(f)$ is a convex subcomplex (recall that, for subcomplexes, $\ell_2$--convexity is equivalent to $\ell_1$--convexity \cite{Haglund-ss}).

Observe that the action $Z_G(f)\acts\Fix(f)$ is cocompact. In order to see this, consider a point $x\in\Fix(f)$ and let $G_x\leq G$ be the finite subgroup fixing $x$. If $gx\in\Fix(f)$ for some $g\in G$, then $g^{-1}fg\in G_x$. The set $\{g\in G\mid g^{-1}fg\in G_x\}$ is a finite union of right cosets of $Z_G(f)$. Hence $Z_G(f)$ acts cofinitely on every intersection between $\Fix(f)$ and a $G$--orbit. Since there are only finitely many $G$--orbits of vertices in $Y$, it follows that $Z_G(f)$ acts cocompactly on $\Fix(f)$.

In conclusion, for every finite-order element $f\in G$, the centraliser $Z_G(f)$ is convex-cocompact in $Y$. By Lemma~\ref{cc properties}, finite intersections of these subgroups are also convex-cocompact.
\end{proof}

Combining the previous lemma with Theorem~\ref{coarse median vs hyperplanes}, we immediately obtain the following criterion. We emphasise that the action $G\acts X$ in Proposition~\ref{centraliser hyp-stab} is \emph{not} required to be itself strongly cellular.

\begin{prop}\label{centraliser hyp-stab}
Let $G$ be a group with a cocompact cubulation $G\acts X$ where every hyper\-plane-stabiliser is commensurable to the centraliser of a finite set of finite-order elements of $G$. Then all strongly cellular, cocompact cubulations of $G$ induce the same coarse median structure as $G\acts X$. 
\end{prop}
\begin{proof}
Let $G\acts X$ be a cocompact cubulation where hyperplane-stabilisers are commensurable to centralisers of finite sets of finite-order elements. Let $G\acts Y$ be any strongly cellular, cocompact cubulation. Since convex-cocompactness is a commensurability invariant (Lemma~\ref{cc properties}), Lemma~\ref{universally cc} implies that all stabilisers of hyperplanes of $X$ are convex-cocompact in $Y$. Now, Theorem~\ref{coarse median vs hyperplanes} implies that $X$ and $Y$ induce the same coarse median structure on $G$.
\end{proof}

\subsection{Graph products of finite groups}\label{graph prod subsect}

In this subsection, we use Proposition~\ref{centraliser hyp-stab} to deduce that all strongly cellular, cocompact cubulations of graph products of finite groups induce the same coarse median structure (Theorem~\ref{sc rigidity 2}). In particular, this will prove Theorem~\ref{RACGs intro}(1) and Theorem~\ref{autom intro}(1).

Every graph product of finite groups admits a particularly nice cocompact cubulation: its \emph{graph-product complex}. This was shown in \cite{Davis-buildings} (also see \cite{Ruane-Witzel} and \cite[Theorem~2.27]{Genevois-Martin}), but we briefly recall here the construction.\footnote{For graph products of order--$2$ groups (i.e.\ right-angled Coxeter groups), the complex that we are about to describe is the first cubical subdivision of what is usually known as the Davis complex.}

Let $G$ be the graph product determined by the data $(\G,\{F_v\}_{v\in\G})$. Here $\G$ is a finite simplicial graph and $\{F_v\}_{v\in\G}$ is a collection of finite groups indexed by the vertices of $\G$. The group $G$ is the quotient of the free product of the $F_v$ by the normal subgroup generated by the commutator sets $[F_w,F_{w'}]$ such that $[w,w']$ is an edge of $\G$. If $\Lambda\sq\G$ is a subgraph, we denote by $G_{\Lambda}$ the subgroup of $G$ generated by the union of the $F_v$ with $v\in\Lambda$. If $\Lambda$ is the empty-set, then we let $G_{\Lambda}$ denote the trivial subgroup of $G$.

The $0$--skeleton of the graph-product complex $\mc{D}$ is identified with the collection of cosets $gG_c$, where $g\in G$ and $c\sq\G$ is a clique (possibly empty). We add edges $[gG_c,gG_{c'}]$ when $c\sq c'$ and $\#c=\#c'-1$. Fixing $g\in G$, if $c\sq c'$ are cliques with $\#c=\#c'-k$, then the subgraph of $\mc{D}$ spanned by the vertices $gG_{c''}$ with $c\sq c''\sq c'$ is isomorphic to the $1$--skeleton of a $k$--cube. We complete the construction of $\mc{D}$ by glueing a cube of the appropriate dimension to each such subgraph, with obvious identifications ensuring that each cube is uniquely determined by its vertex set. 

Note that $G$ permutes left cosets by left multiplication. This gives a $G$--action on the $0$--skeleton of $\mc{D}$, which naturally extends to a cellular action on the whole $\mc{D}$. 

\begin{prop}\label{Davis complex}
Let $G$ be a graph product of finite groups.
\begin{enumerate}
    \item The graph-product complex $\mc{D}$ is a $\CAT$ cube complex.
    \item The action $G\acts\mc{D}$ is proper, cocompact and strongly cellular.
    \item Hyperplane-stabilisers are precisely conjugates of the subgroups $G_{\lk(v)}$ for $v\in\G^{(0)}$.
\end{enumerate}
\end{prop}
\begin{proof}
Checking (1) and (2) is straightforward and we leave it to the reader. We prove (3).

Let $\mf{w}\sq\mc{D}$ be the hyperplane dual to the edge $e=[gG_c,gG_cF_v]$, where $g\in G$ is an element and $c\cup\{v\}$ is a clique. Suppose $f\sq\mc{D}$ is an edge opposite to $e$ in a square containing $e$. Then $f$ takes one of the following two forms:
\begin{itemize}
\item $f=[gG_cF_w,gG_cF_vF_w]$, where $c\cup\{v,w\}$ is a clique of $\G$;
\item $f=[ghG_{c-\{w\}},ghG_{c-\{w\}}F_v]$, where $w\in c$ and $h\in F_w$.
\end{itemize}
From this, we deduce that the edges of $\mc{D}$ dual to $\mf{w}$ are precisely those of the form:
\[ [g'G_{c'},g'G_{c'}F_v], \]
where $c'\sqcup\{v\}$ is a clique in $\G$ and the element $g^{-1}g'$ is a product of elements of those $F_w$ for which $[v,w]$ is an edge of $\G$. 

Now, if $h\in G$, we have $h\mf{w}=\mf{w}$ if and only if the edge $[hgG_c,hgG_cF_v]$ is of the above form. This happens exactly when $g^{-1}hg$ lies in the subgroup $G_{\lk(v)}$, as required.
\end{proof}

\begin{rmk}\label{commensurable to centraliser}
For every vertex $v\in\G$ and every $g\in F_v-\{1\}$, we have $G_{\lk(v)}\leq Z_G(g)\leq G_{\St(v)}$ (see e.g.\ \cite[Theorem~32]{Barkauskas}). Since $G_{\St(v)}=F_v\x G_{\lk(v)}$, this shows that $G_{\lk(v)}$ is always commensurable to the centraliser of a finite-order element.
\end{rmk}

\begin{thm}\label{sc rigidity 2}
Let $G$ be a graph product of finite groups. All strongly cellular, cocompact cubulations of $G$ induce the same coarse median structure on $G$.
\end{thm}
\begin{proof}
It suffices to apply Proposition~\ref{centraliser hyp-stab} to the graph-product complex of $G$. Proposition~\ref{Davis complex}(3) and Remark~\ref{commensurable to centraliser} ensure that hyperplane-stabilisers have the required form.
\end{proof}

Theorem~\ref{autom intro}(1) immediately follows from Theorem~\ref{sc rigidity 2} and the fact that the action on the graph-product complex is strongly cellular.

\subsection{Virtually abelian groups}\label{vab subsect}

In this subsection, we completely classify cubical coarse median structures on free abelian groups $\Z^n$ (Proposition~\ref{cms on Z^n}) and on products of infinite dihedrals $D_{\infty}^n$ (Proposition~\ref{cms on D^n}). While $\Z^n$ has many cubical coarse medians --- infinitely many orbits up to the natural $\Aut(\Z^n)$--action --- there are only finitely many cubical coarse median structures on $D_{\infty}^n$. 

The first result is used in the next subsection to study cubical coarse medians on right-angled Artin groups and prove Theorem~\ref{RAAGs main}, which implies Theorem~\ref{RAAGs intro}. The second result is needed to deduce Corollary~\ref{loose cor intro} from Theorem~\ref{RACGs intro} and it also shows that, in the Conjecture from the Introduction, Item~(b) implies Item~(a) (Remark~\ref{conj rmk}).

\begin{lem}\label{cms on vab}
Let $A$ be a virtually abelian group. Every cubical coarse median structure on $A$ is induced by a proper cocompact $A$--action on the standard cubical tiling of $\R^n$, for some $n\geq 0$.
\end{lem}
\begin{proof}
Let $A\acts X$ be a proper cocompact action on a $\CAT$ cube complex. Up to passing to an invariant convex subcomplex (which does not alter the induced coarse median structure), the cubical flat torus theorem \cite[Theorem~3.6]{WW} allows us to assume that $X$ is a product of quasi-lines $L_1\x\dots\x L_n$. 

By the $\CAT$ flat torus theorem \cite[Corollary~II.7.2]{BH}, there exists an $A$--invariant $\ell_2$--convex subspace $F\sq X$ that is $\ell_2$--isometric to a Euclidean flat. The projection of $A$ to each factor $L_i$ is an $\ell_2$--geodesic line $\alpha_i$. Each $\alpha_i$ is also an $\ell_1$--geodesic by Lemma~\ref{l_2 geodesics are l_1}, hence it is a median subalgebra of $L_i$. It follows that $F=\alpha_1\x\dots\x\alpha_n$ is an $A$--invariant median subalgebra of $X$. 

In conclusion, the actions $A\acts X$ and $A\acts F$ induce the same coarse median structure on $A$. In addition, equipped with the restriction of the $\ell_1$--metric on $X$, the flat $F$ is isometric to $\R^n$ with the $\ell_1$--metric. Finally, since $A$ acts on $F$ discretely and permuting the orthogonal directions $\alpha_1,\dots,\alpha_n$, it preserves a tiling of $\R^n$ by cuboids. Scaling to $1$ all edge-lengths of these cuboids, we obtain the standard cubical tiling of $\R^n$ without altering the coarse median structure.
\end{proof}

As we are about to show, cubical coarse medians on $\Z^n$ are parametrised by the following objects.

\begin{defn}
A \emph{virtual basis} of $\Z^n$ is a set $\{C_1,\dots,C_n\}$, where each $C_i\leq\Z^n$ is a maximal cyclic subgroup and $\langle C_1\cup\dots\cup C_n\rangle$ has finite index in $\Z^n$. We denote by $\mc{VB}(\Z^n)$ the set of virtual bases. 
\end{defn}

An example of a virtual basis that is not a basis is given by $C_1=\langle(1,1)\rangle$, $C_2=\langle(1,-1)\rangle$ in $\Z^2$.

\begin{rmk}\label{rmk:VB}
If $A\in {\rm GL}_n(\Q)$ is a matrix with integer entries and $\{C_1,\dots,C_n\}$ is a virtual basis of $\Z^n$, we obtain a new virtual basis $\{C_1',\dots,C_n'\}$ by taking $C_i'$ to be the maximal cyclic subgroup of $\Z^n$ containing $A(C_i)$.
% since matrices in GL_n(Q) preserve linear independence
Multiplying $A$ by a constant does not alter the $C_i'$, so we obtain a natural transitive action:
\[ {\rm PGL}_n(\Q)\acts\mc{VB}(\Z^n).\]

As another way of making sense of this action, note that $\mc{VB}(\Z^n)$ can be equivalently defined as the collection of all cardinality--$n$ general-position subsets of $\Q\mathbb{P}^{n-1}$. The action ${\rm PGL}_n(\Q)\acts\mc{VB}(\Z^n)$ is then directly induced by the standard action ${\rm PGL}_n(\Q)\acts\Q\mathbb{P}^{n-1}$ by projective automorphisms.

The stabiliser of the standard basis of $\Z^n$ is the subgroup $N\leq{\rm PGL}_n(\Q)$ generated by diagonal matrices and permutation matrices (i.e.\ the normaliser of the diagonal subgroup). This shows that we can naturally identify:
\[\mc{VB}(\Z^n)\cong {\rm PGL}_n(\Q)/N\cong {\rm GL}_n(\Q)/\overline N,\]
where $\overline N\leq{\rm GL}_n(\Q)$ is again the normaliser of the diagonal subgroup.
\end{rmk}

\begin{prop}\label{cms on Z^n}
The sets $\mc{CM}_{\square}(\Z^n)$ and $\mc{VB}(\Z^n)$ are in 1-to-1 correspondence, equivariantly with respect to the action $\Out(\Z^n)\acts\mc{CM}_{\square}(\Z^n)$ and the left-multiplication action ${\rm GL}_n(\Z)\acts{\rm GL}_n(\Q)/\overline N$.

The correspondence pairs each coarse median structure $[\mu]$ with the set $\{C_1,\dots,C_n\}$ of maximal, cyclic, $[\mu]$--quasi-convex subgroups of $\Z^n$.
\end{prop}
\begin{proof}
By Lemma~\ref{cms on vab}, every cubical coarse median on $\Z^n$ is induced by a $\Z^n$--action on the standard cubical tiling of $\R^n$, which we denote by $S$. We have $\Aut(S)=T\rtimes {\rm O}(n,\Z)$, where $T\simeq\Z^n$ is the translation subgroup, and ${\rm O}(n,\Z)$ is the group of signed permutation matrices. 

Up to passing to a finite-index subgroup of $\Z^n$ (for instance, the intersection of all subgroups of index $\leq\#{\rm O}(n,\Z)$), we can assume that all our actions $\Z^n\acts S$ factor through $T$. 

We equip $T$ with its standard basis $t_1,\dots,t_n$ (where each $t_i$ translates in only one coordinate direction), as well as the coordinate-wise median operator $m$ associated with this basis. Note that $[m]$ is precisely the coarse median structure on $T$ induced by the action $T\acts S$.

The above discussion shows that every cubical coarse median structure $[\mu]$ on $\Z^n$ can be obtained by pulling back $[m]$ via an embedding $\iota\colon\Z^n\hookrightarrow T$. To each such structure $[\mu]$, we associate the virtual basis $\{\iota^{-1}(\langle t_1\rangle),\dots,\iota^{-1}(\langle t_n\rangle)\}\in\mc{VB}(\Z^n)$, which is precisely the set of maximal, cyclic, $[\mu]$--quasi-convex subgroups of $\Z^n$.

It is straightforward to see that $[\mu]$ is completely determined by this element of $\mc{VB}(\Z^n)$. Finally, every element of $\mc{VB}(\Z^n)$ arises in this way, as every action by translations $H\acts\R^n$ of a finite-index subgroup $H\leq\Z^n$ can be extended to a $\Z^n$--action by translations, which will preserve a tiling of $\R^n$ by sufficiently fine cubes. This concludes the proof.
\end{proof}

Having classified cubical coarse median structures on $\Z^n$, we move on to the case of $D_{\infty}^n$. 

\begin{rmk}\label{cms on D^2}
There are two obvious cubical coarse median structures on $D_{\infty}^2$. The first (or ``standard'') structure is induced by the action on the Davis complex: it is the action on the standard square tiling of $\R^2$ with each reflection axis parallel either to the $x$--axis or to the $y$--axis. The second (or ``$\tfrac{\pi}{4}$--rotated'') structure is induced by the action on the standard tiling of $\R^2$ with all reflection axes meeting the $x$-- and $y$--axes at an angle of $\tfrac{\pi}{4}$ or $\tfrac{3\pi}{4}$. This is just the first action rotated by $\tfrac{\pi}{4}$ and, unlike it, it is not strongly cellular. Indeed, recall from Theorem~\ref{sc rigidity 2} that all strongly cellular cubulations induce the standard coarse median structure.
\end{rmk}

The next result shows that there are no other cubical coarse median structures on $D_{\infty}^2$. Similarly, all cubical coarse medians on $D_{\infty}^n$ are simply obtained by choosing the ``$\tfrac{\pi}{4}$--rotated'' structure on some $D_{\infty}^2$--factors and the ``standard'' one on the rest.

\begin{prop}\label{cms on D^n}
The set $\mc{CM}_{\square}(D_{\infty}^n)$ is finite. More precisely, for every $[\mu]\in\mc{CM}_{\square}(D_{\infty}^n)$, there is a product splitting $D_{\infty}^n=A_1\x\dots\x A_k\x B_1\x\dots\x B_{n-2k}$ for some $0\leq k\leq\tfrac{n}{2}$, where all $A_i$ and $B_j$ are $[\mu]$--quasi-convex, we have $A_i\simeq D_{\infty}^2$ and $B_j\simeq D_{\infty}$, and the restriction of $[\mu]$ to each $A_i$ is the $\tfrac{\pi}{4}$--rotated structure.
\end{prop}
\begin{proof}
Let again $S$ denote the standard cubical tiling of $\R^n$. By Lemma~\ref{cms on vab}, every cubical coarse median on $D_{\infty}^n$ arises from a proper cocompact action $D_{\infty}^n\acts S$. The latter corresponds to a homomorphism $\rho\colon D_{\infty}^n\ra \Aut(S)=T\rtimes {\rm O}(n,\Z)$ with finite-index image, where $T\simeq\Z^n$.

Note that $\rho$ must be faithful. Indeed, the kernel of $\rho$ is necessarily finite, and $D_{\infty}$ (and hence $D_{\infty}^n$) does not have any nontrivial finite normal subgroups.

Now, choose a reflection $r_i$ in each factor of $D_{\infty}^n$ and consider the subgroup $R:=\langle r_1,\dots,r_n\rangle\simeq(\Z/2\Z)^n$. Let $\pi\colon\Aut(S)\ra{\rm O}(n,\Z)$ be the projection with $\ker\pi=T$. Since $T$ is torsion-free and $\rho$ is injective, the homomorphism $\pi\rho$ must be injective on $R$. Thus, $\pi\rho(R)$ is isomorphic to $(\Z/2\Z)^n$ and it is a subgroup of ${\rm O}(n,\Z)$, the group of signed permutation matrices.

Every element of $\pi\rho(R)$ has order $2$, hence it is ${\rm O}(n,\Z)$--conjugate to a block-diagonal matrix, with each block chosen from:
\begin{align*}
&
\begin{pmatrix}
\pm 1
\end{pmatrix},
 &
\pm \begin{pmatrix}
 0 & 1 \\
 1 & 0
\end{pmatrix}.
\end{align*}
% signed permutation matrices admit a block decomposition, with each block corresponding to a cycle in the permutational part; eigenvalues are roots of unity, so only $2$--cycles have only +-1 as their eigenvalues;
Since $\pi\rho(R)$ is abelian, its elements can be simultaneously block-diagonalised. Since $\pi\rho(R)\simeq(\Z/2\Z)^n$, it follows that $\pi\rho(R)$ is ${\rm O}(n,\Z)$--conjugate to a subgroup of the form:
\[ 
\begin{psmallmatrix}
    \diagentry{R_1}\\
    &\diagentry{\xddots} \\
    &&\diagentry{R_k} \\
    &&&\diagentry{\pm1}\\
    &&&&\diagentry{\xddots}\\
    &&&&&\diagentry{\pm 1}
\end{psmallmatrix},
\]
where $0\leq k\leq n/2$ and each $R_{\ell}\simeq(\Z/2\Z)^2$ is the group of $2\x2$ matrices $\{\pm\begin{psmallmatrix} 1 & 0 \\ 0 & 1 \end{psmallmatrix},\pm\begin{psmallmatrix} 0 & 1 \\ 1 & 0 \end{psmallmatrix} \}$.

Finally, the coarse median structure that $\rho$ induces on $D_{\infty}^n$ is completely determined by the datum of which pairs of indices $1\leq i,j\leq n$ are such that $\langle r_i,r_j\rangle$ corresponds to one of the $2\x 2$ blocks $R_{\ell}$ in the above decomposition. Indeed, this datum determines which infinite cyclic subgroups of $D_{\infty}^n$ translate in exactly one coordinate direction in $S\simeq\R^n$ and are thus convex-cocompact. And the latter determines the coarse median structure by applying Proposition~\ref{cms on Z^n} to a free abelian finite-index subgroup of $D_{\infty}^n$.
\end{proof}

\begin{rmk}\label{conj rmk}
Item~(b) implies Item~(a) in the Conjecture from the Introduction.

Indeed, if $W_{\G}$ is a right-angled Coxeter group, $W_{\G}$ has only finitely many maximal virtually-abelian parabolic subgroups $P_1,\dots,P_k$ up to conjugacy. Each of them has a finite-index subgroup $P_i'\simeq D_{\infty}^{m_i}$ for some $m_i\geq 0$. By Lemma~\ref{cc properties}(3) and Proposition~\ref{prop:cc=qc}, all $P_i$ and $P_i'$ are $[\mu]$--quasi-convex with respect to each $[\mu]\in\mc{CM}_{\square}(W_{\G})$, so they inherit cubical coarse median structures from $[\mu]$. By Proposition~\ref{cms on D^n}, there are only finitely many possible restrictions of $[\mu]$ to the $P_i$. Thus, if these restrictions completely determine $[\mu]$, then the set $\mc{CM}_{\square}(W_{\G})$ is finite.
\end{rmk}

\begin{rmk}
    The techniques used in this subsection can be similarly used to completely describe the set of cubical coarse medians $\mc{CM}_{\square}(A)$ for any virtually abelian group $A$. 
    
    Indeed, passing to a normal free abelian subgroup, we can always write
    \[ 1\longrightarrow \Z^n \longrightarrow A\longrightarrow F\longrightarrow 1 , \]
    for some $n\geq 1$ and some finite group $F$. Conjugation gives a monodromy homomorphism $\rho\colon F\ra\Out(\Z^n)={\rm GL}_n(\Z)$. In particular, the subgroup $\rho(F)\leq{\rm GL}_n(\Z)$ acts on $\mc{CM}_{\square}(\Z^n)$, which is identified with $\mc{VB}(\Z^n)\cong {\rm GL}_n(\Q)/\overline N$ (recalling Proposition~\ref{cms on Z^n} and Remark~\ref{rmk:VB}).

    Now, the space $\mc{CM}_{\square}(A)$ is naturally identified with the set of $\rho(F)$--fixed points in $\mc{CM}_{\square}(\Z^n)$. In one direction, every cocompact cubulation of $A$ determines such a fixed point by restriction to the finite-index subgroup $\Z^n$. Conversely, if a cubical coarse median structure on a normal finite-index subgroup is preserved by the conjugacy action of the whole group, then it comes from a cocompact cubulation of the whole group; this was shown in \cite[Corollary~G]{Fio10a} using very similar ideas to those in the proof of Theorem~\ref{coarse median vs hyperplanes} above.

    Thus, the determination of $\mc{CM}_{\square}(A)$ boils down to describing the set of fixed points of a finite subgroup of ${\rm GL}_n(\Z)$ acting on the coset space ${\rm GL}_n(\Q)/\overline N$, where $\overline N=D\rtimes {\rm O}(n,\Z)$ is the normaliser of the subgroup $D$ of diagonal matrices, and ${\rm O}(n,\Z)$ is the group of signed permutation matrices.

    As a special case of this, we see that $A$ is cocompactly cubulated if and only if $\rho(F)$ fixes at least one point in ${\rm GL}_n(\Q)/\overline N$. Equivalently, $\rho(F)\leq{\rm GL}_n(\Z)$ can be conjugated into $\overline N$ by an element of ${\rm GL}_n(\Q)$, in which case $\rho(F)$ can even be conjugated into ${\rm O}(n,\Z)$. This recovers Hagen's characterisation of cocompactly cubulated crystallographic groups from \cite{Hagen-cryst}.
\end{rmk}

The following example, mentioned in the Introduction, shows that Proposition~\ref{centralisers are median-cocompact} cannot be improved. In general, we cannot expect centralisers to be convex-cocompact even if we have the freedom to choose the cubulation.

\begin{ex}\label{never-cc centraliser}
Consider $G=\Z^n\rtimes\mf{S}_n$, where the symmetric group $\mf{S}_n$ acts by permuting the standard basis of $\Z^n$. For $n=3$ or $n\geq 5$, the centraliser $Z_G(\mf{S}_n)$ is not convex-cocompact in any cocompact cubulation of $G$.

Indeed, $Z_G(\mf{S}_n)$ is the infinite cyclic subgroup generated by the element $v:=(1,1,\dots,1)$. If $\langle v\rangle$ were convex-cocompact in a cocompact cubulation of $G$, then, by Proposition~\ref{cms on Z^n}, there would exist a virtual basis $\langle v\rangle,C_1,\dots,C_{n-1}$ of $\Z^n$ that is permuted by $\mf{S}_n$. In particular, the subgroups $C_1,\dots,C_{n-1}$ would be permuted by $\mf{S}_n$, since $v$ is fixed. However, for every maximal infinite cyclic subgroup $C<\Z^n$ with $C\neq\langle v\rangle$, the orbit $\mf{S}_n\cdot C$ contains at least $n$ distinct subgroups. 

Let us prove this last statement. Suppose that $C$ is generated by an element $(x_1,\dots,x_n)$. If the absolute values $|x_i|$ are not all equal, then, without loss of generality, there exist an integer $a>0$ and an index $1\leq k<n$ such that $|x_i|=a$ for $1\leq i\leq k$ and $|x_i|<a$ for $k<i\leq n$. Since $\mf{S}_n$ acts $k$--transitively on $\{1,\dots,n\}$, it follows that the orbit $\mf{S}_n\cdot C$ contains at least $\binom{n}{k}\geq n$ distinct subgroups. If instead all the $|x_i|$ are equal, we can assume that there exists $1\leq k<n$ such that $x_i=1$ for $1\leq i\leq k$ and $x_i=-1$ for $k<i\leq n$. In this case, the orbit $\mf{S}_n\cdot C$ has cardinality $\binom{n}{k}$ if $n\neq 2k$ and $\tfrac{1}{2}\cdot\binom{n}{k}$ if $n=2k$. When $n\geq 5$, even this last quantity is $\geq n$, completing the proof.
\end{ex}

\subsection{Right-angled Artin groups}\label{RAAGs subsect}

Let $A_{\G}$ be a right-angled Artin group. In this subsection, we study the set $\mc{CM}_{\square}(A_{\G})$. The main result is Theorem~\ref{RAAGs main}, which implies Theorem~\ref{RAAGs intro} from the Introduction, but also concerns right-angled Artin groups that are not twistless.

Let $A_{\G}\acts\X_{\G}$ be the standard action on the universal cover of the Salvetti complex. Let $[\mu_{\G}]\in\mc{CM}_{\square}(A_{\G})$ be the induced coarse median structure on $A_{\G}$. We refer to $[\mu_{\G}]$ as the \emph{standard} coarse median structure on $A_{\G}$.

We will obtain Theorem~\ref{RAAGs main} by adapting the proof of Proposition~\ref{centraliser hyp-stab}. The main differences are that: (1) we cannot exploit torsion, and (2) stabilisers of hyperplanes of the Salvetti complex are not commensurable to centralisers. On the other hand, no analogue of the `strongly cellular' assumption will be required. The following notion will play an important role.

\begin{defn}\label{decomposable flats defn}
A cocompact cubulation $A_{\G}\acts X$ has \emph{decomposable flats} if every maximal abelian subgroup $A\leq A_{\G}$ admits a basis $a_1,\dots,a_k$ with each $\langle a_i\rangle$ convex-cocompact in $X$.
\end{defn}

An example of a cubulation that fails to have this property is provided by the $\Z^2$--action on the standard tiling of $\R^2$ where the standard generators of $\Z^2$ translate by $(1,1)$ and $(1,-1)$ respectively.

\begin{rmk}\label{decomposable flats rmk}
Here is some motivation for considering the above condition.
\begin{enumerate}
\setlength\itemsep{.25em}
    \item The standard action $A_{\G}\acts\X_{\G}$ on the universal cover of the Salvetti complex has decomposable flats, since each $v\in\G$ generates a convex-cocompact subgroup.
    \item Whether or not a given cocompact cubulation $A_{\G}\acts X$ has decomposable flats only depends on the induced coarse median structure on $A_{\G}$, since this is true of convex-cocompactness.
    \item The action $\Out(A_{\G})\acts\mc{CM}_{\square}(A_{\G})$ preserves the property of having decomposable flats. 
    
    Indeed, let $A_{\G}\acts X$ be a cocompact cubulation with decomposable flats. Consider $\varphi\in\Aut(A_{\G})$ and denote by $A_{\G}\acts X^{\varphi}$ the action on $X$ precomposed with $\varphi$. If $A\leq A_{\G}$ is a maximal abelian subgroup, then so is $\varphi(A)$, which then admits a basis $a_i$ of convex-cocompact elements for the action $A_{\G}\acts X$. It follows that the elements $\varphi^{-1}(a_i)$ form a basis of $A$ and are convex-cocompact for $A_{\G}\acts X^{\varphi}$. Thus, $A_{\G}\acts X^{\varphi}$ has decomposable flats.
    
    \item For a general cocompact cubulation $G\acts X$ of a general group and a convex-cocompact abelian subgroup $A\leq G$, it is only possible to find a basis $a_1,\dots,a_k$ of a \emph{finite-index} subgroup of $A$ such that each $\langle a_i\rangle$ is convex-cocompact in $X$. Compare Proposition~\ref{cms on Z^n}.
\end{enumerate}
\end{rmk}

\begin{lem}\label{lem:cc_basis}
Let $A_{\G}\acts X$ be a cocompact cubulation with decomposable flats. Let $A\leq A_{\G}$ be a (not necessarily maximal) abelian subgroup that is both convex-cocompact in $X$ and closed under taking roots. Then $A$ admits a basis $a_1,\dots,a_k$ with each $\langle a_i\rangle$ convex-cocompact in $X$.
\end{lem}
\begin{proof}
Let $A'\leq A_{\G}$ be a maximal abelian subgroup containing $A$ and let $x_1,\dots,x_n$ be a basis of $A'$ such that each $\langle x_i\rangle$ is convex-cocompact. If an element $x_{i_1}^{a_1}\cdot\ldots\cdot x_{i_s}^{a_s}$ with $a_1,\dots,a_s\neq 0$ lies in $A$, then the fact that $A$ is convex-cocompact implies that $A$ must contain nontrivial powers of $x_{i_1},\dots,x_{i_s}$ (for instance, by the argument for \cite[Lemma~3.16]{Fio10a}). Since $A$ is closed under taking roots, it must then contain $x_{i_1},\dots,x_{i_s}$ themselves.

This shows that $A$ contains all $x_i$ required to write any of its elements. They provide the required basis of $A$.
\end{proof}

The automorphism group $\Aut(A_{\G})$ is generated by elementary automorphisms described by Laurence and Servatius \cite{Laurence,Servatius}. In the terminology of \cite[Section~2.2]{CSV}, these are known as \emph{graph automorphisms}, \emph{inversions}, \emph{partial conjugations} and \emph{transvections}; the latter are in turn divided into \emph{folds} and \emph{twists}. The only automorphisms that will be important for us are twists.

\begin{defn}\label{defn:twists}
If $v,w\in\G$ are distinct vertices with $\St(v)\sq\St(w)$, there is a well-defined automorphism of $A_{\G}$ that fixes all generators in $\G-\{v\}$ and maps $v\mapsto vw$. We denote this automorphism by $\tau_{v,w}$ and refer to it as a \emph{twist}. The \emph{twist subgroup} $T(A_{\G})\leq\Aut(A_{\G})$ is the subgroup generated by all twists.

We say that $A_{\G}$ is \emph{twistless} if the twist subgroup $T(A_{\G})$ is trivial. Equivalently, we never have $\St(v)\sq\St(w)$ for distinct vertices $v,w\in\G$.
\end{defn}

The right-angled Artin group $A_{\G}$ is twistless precisely when the entire group $\Out(A_{\G})$ fixes the standard coarse median structure $[\mu_{\G}]$ \cite[Proposition~A]{Fio10a}. This is always the case when $\Out(A_{\G})$ is finite.

We are now ready to state the main result of this subsection.

\begin{thm}\label{RAAGs main}
Let $A_{\G}$ be a right-angled Artin group.
\begin{enumerate}
\item The twist group $T(A_{\G})$ acts transitively and with finite stabilisers on the subset of $\mc{CM}_{\square}(A_{\G})$ corresponding to cubulations with decomposable flats.
\item If $A_{\G}$ is twistless, all cocompact cubulations of $A_{\G}$ have decomposable flats.
\end{enumerate}
\end{thm}

\begin{rmk}
Let $U(A_{\G})\leq \Aut(A_{\G})$ be the subgroup generated by all Laurence--Servatius generators except for twists, i.e.\ graph automorphisms, inversions, folds and partial conjugations (including inner automorphisms). This is known as the \emph{untwisted subgroup} \cite{CSV}. As shown in \cite{Fio10a}, $U(A_{\G})$ is precisely the stabiliser of the standard coarse median structure on $A_{\G}$. 

It is clear that $U(A_{\G})$ and $T(A_{\G})$ generate $\Aut(A_{\G})$, but Theorem~\ref{RAAGs main}(1) implies that there is even a product decomposition $\Aut(A_{\G})=T(A_{\G})\cdot U(A_{\G})$ and that the intersection $T(A_{\G})\cap U(A_{\G})$ is finite. Neither $U(A_{\G})$ nor $T(A_{\G})$ is normal in general, so this is not (virtually) a semi-direct product. The existence of this splitting also follows from \cite[Corollary~7.12]{BCV}.
\end{rmk}

We begin the proof of Theorem~\ref{RAAGs main} with a few elementary observations on the twist group.

\begin{rmk}\label{rmk:Aut(Z^n)}
We have $\Aut(\Z^n)=\langle T(\Z^n), {\rm O}(n,\Z)\rangle$, where $T(\Z^n)$ is the twist subgroup, while ${\rm O}(n,\Z)\simeq(\Z/2\Z)^n\rtimes\mf{S}_n$ is the group of signed permutation matrices (here $\mf{S}_n$ denotes the symmetric group). The finite subgroup ${\rm O}(n,\Z)$ permutes the elements of the standard basis of $\Z^n$, possibly inverting some of them. Since ${\rm O}(n,\Z)$ normalises $T(\Z^n)$, we have $\Aut(\Z^n)=T(\Z^n)\cdot {\rm O}(n,\Z)$. The intersection $T(\Z^n)\cap {\rm O}(n,\Z)$ is nontrivial for $n\geq 2$.
\end{rmk}

For every vertex $v\in\G$, denote by $\kappa(v)\sq\G$ the intersection of all maximal cliques of $\G$ that contain $v$. Equivalently, it is easy to see that we have:
\begin{equation}\label{eq:k(v)}
    \kappa(v) = \{ w\in\G^{(0)} \mid \St(v)\sq\St(w) \} .
\end{equation}  

The action $T(A_{\G})\acts A_{\G}$ can be completely described in terms of the behaviour on the free abelian subgroups $A_{\kappa(v)}\leq A_{\G}$. Indeed, by Equation~(\ref{eq:k(v)}) and the definition of twists, we have $\varphi(A_{\kappa(v)})=A_{\kappa(v)}$ for every $v\in\G$ and every $\varphi\in T(A_{\G})$. Clearly, $\varphi$ is the trivial automorphism if and only if it has trivial restriction to all the free abelian subgroups $A_{\kappa(v)}$, so $T(A_{\G})$ embeds in the finite product $\prod_vT(A_{\kappa(v)})$. It follows that the $T(A_{\G})$--stabiliser of the standard coarse median structure on $A_{\G}$ is finite, since this is true in the free abelian case (Remark~\ref{rmk:Aut(Z^n)}).

We now prove a few lemmas from which Theorem~\ref{RAAGs main} can be quickly deduced.

\begin{lem}\label{phi_v}
Let $A_{\G}\acts X$ be a cocompact cubulation with decomposable flats. Let $v\in\G$ be a vertex such that the subgroup $\langle v\rangle$ is not convex-cocompact in $A_{\G}\acts X$ and such that the star $\St(v)\sq\G$ is maximal among stars of such vertices. Then there exists an automorphism $\varphi_v\in T(A_{\G})$ such that:
\begin{itemize}
\item $\langle\varphi_v(v)\rangle$ is convex-cocompact in $A_{\G}\acts X$;
\item $\varphi_v(w)=w$ for all $w\in\G$ such that $\langle w\rangle$ is convex-cocompact in $A_{\G}\acts X$.
\end{itemize}
\end{lem}
\begin{proof}
If $c\sq\G$ is a maximal clique, then $A_c$ is convex-cocompact in $X$ by Lemma~\ref{cc properties}(3). Since $\kappa(v)$ is the intersection of all maximal cliques containing $v$, it follows from Lemma~\ref{cc properties}(1) that $A_{\kappa(v)}$ is convex-cocompact in $X$. 

Write $\kappa(v)=\{v\}\sqcup\{w_1,\dots,w_k\}$ and choose $0\leq s\leq k$ so that $\langle w_{s+1}\rangle,\dots,\langle w_k\rangle$ are convex-cocompact in $X$, while $\langle w_1\rangle,\dots,\langle w_s\rangle$ are not. Also set $\overline\kappa(v):=\{v\}\sqcup\{w_1,\dots,w_s\}$ and note that, by maximality of $\St(v)$ and Equation~(\ref{eq:k(v)}), all vertices in $\overline\kappa(v)$ have the same star. 

Since the action $A_{\G}\acts X$ has decomposable flats, Lemma~\ref{lem:cc_basis} guarantees that there exists a basis $x_0,\dots, x_k$ of $A_{\kappa(v)}$ 
such that all $\langle x_i\rangle$ are convex-cocompact in $X$. Up to permuting and inverting the $x_i$, we can assume that $x_i=w_i$ for $s+1\leq i\leq k$, since $A_{\kappa(v)}\simeq\Z^{k+1}$ has exactly $k+1$ convex-cocompact directions by Proposition~\ref{cms on Z^n}. 

Let $\psi\in\Aut(A_{\kappa(v)})$ map the basis $\{v,w_1,\dots,w_k\}$ to the basis $\{x_0,\dots,x_k\}$ while fixing the elements $w_{s+1},\dots,w_k$. Recalling that the twist subgroup of a free abelian group acts transitively on bases (Remark~\ref{rmk:Aut(Z^n)}), we can take $\psi\in T(A_{\kappa(v)})$. Moreover, $\psi$ is a product of twists $\tau_{a,b}$ with $a\in\overline\kappa(v)$ and $b\in\kappa(v)$, in the notation of Definition~\ref{defn:twists}. Since all elements of $\overline\kappa(v)$ have the same star in $\G$, which is contained in the star of each element of $\kappa(v)$, each of the twists $\tau_{a,b}$ extends to a twist of the whole $A_{\G}$.

In conclusion, the product of these twists of $A_{\G}$ is an automorphism $\varphi_v\in T(A_{\G})$ fixing $\G-\{v,w_1,\dots,w_s\}$ and mapping $v$ to the convex-cocompact element $\psi(v)=x_0$. This is the required automorphism, concluding the proof.
\end{proof}

\begin{lem}\label{lem:T_transitive}
Let $A_{\G}\acts X$ be a cocompact cubulation with decomposable flats. Then there exists $\varphi\in T(A_{\G})$ such that each subgroup $\langle\varphi(v)\rangle$ with $v\in\G$ is convex-cocompact in $A_{\G}\acts X$.
\end{lem}
\begin{proof}
We prove the lemma by induction on the number $N\geq 0$ of standard generators that fail to be convex-cocompact in $A_{\G}\acts X$. The base step $N=0$ is trivial, taking $\varphi=\id_{A_{\G}}$.

For the inductive step, suppose that the lemma has been proved for cubulations in which at most $N-1$ standard generators are not convex-cocompact. Pick a generator $v\in\G$ such that $\langle v\rangle$ is not convex-cocompact in $A_{\G}\acts X$ and such that $\St(v)$ is maximal among stars of vertices of $\G$ with this property. Let $A_{\G}\acts Y$ denote the action $A_{\G}\acts X$ precomposed with the automorphism $\varphi_v\in T(A_{\G})$ provided by Lemma~\ref{phi_v}.

By Remark~\ref{decomposable flats rmk}(3), $A_{\G}\acts Y$ is again a cocompact cubulation with decomposable flats. In addition, by the choice of $\varphi_v$, there are at most $N-1$ standard generators that are not convex-cocompact in $A_{\G}\acts Y$. By the inductive hypothesis, there exists an automorphism $\psi\in T(A_{\G})$ such that each element $\psi(x)$ with $x\in\G$ is convex-cocompact in $A_{\G}\acts Y$. In other words, the elements $\varphi_v\psi(x)$ are all convex-cocompact in $A_{\G}\acts X$. We conclude by setting $\varphi:=\varphi_v\psi$.
\end{proof}

\begin{lem}\label{lem:cc_vertices}
Let $A_{\G}\acts X$ be a cocompact cubulation. If $\langle v\rangle$ is convex-cocompact in $X$ for every $v\in\G^{(0)}$, then the induced coarse median structure on $A_{\G}$ is the standard one.
\end{lem}
\begin{proof}
Hyperplane-stabilisers of the standard $A_{\G}$--action on the universal cover of the Salvetti complex are of the form $A_{\lk(v)}$ with $v\in\G^{(0)}$, so our goal is to show that these subgroups are all convex-cocompact for $A_{\G}\acts X$. Then we can apply Theorem~\ref{coarse median vs hyperplanes}.

First, we show that all centralisers $A_{\St(v)}=Z_{A_{\G}}(v)$ are convex-cocompact in $X$. In order to see this, note that $v^n$ must act non-transversely on $X$ for some $n\geq 1$ (e.g.\ by \cite[Pro\-po\-si\-tion~2.7(5)]{BF1}). The subgroup $\langle v\rangle$ is convex-cocompact by hypothesis, so Lemma~\ref{cc properties}(4) implies that $Z_{A_{\G}}(v^n)$ is convex-cocompact in $X$ (since it is commensurable to $N_{A_{\G}}(v^n)$). Finally, since we are in a right-angled Artin group, we have $Z_{A_{\G}}(v^n)=Z_{A_{\G}}(v)$.

Now, we know that both $\langle v\rangle$ and $A_{\St(v)}=\langle v\rangle\x A_{\lk(v)}$ are convex-cocompact in $X$. In addition, $A_{\lk(v)}$ is generated by elements $w\in\G$, which all generate convex-cocompact subgroups by hypothesis. Lemma~\ref{lem:cc_products} thus implies that $A_{\lk(v)}$ is convex-cocompact in $X$, as required.
\end{proof}

\begin{proof}[Proof of Theorem~\ref{RAAGs main}]
The combination of Lemmas~\ref{lem:T_transitive} and~\ref{lem:cc_vertices} shows that $T(A_{\G})$ acts transitively on the set of coarse median structures induced by cocompact cubulations with decomposable flats. Proving that $T(A_{\G})$ acts with finite stabilisers then amounts to proving that the $T(A_{\G})$--stabiliser of the \emph{standard} structure is finite. As observed above, this follows from the fact that $T(A_{\G})$ embeds in the finite product of the twist groups of the convex-cocompact free abelian groups $A_{\kappa(v)}$, each of which has only a finite subgroup stabilising the standard coarse median structure on $A_{\kappa(v)}$. This proves part~(1).

Regarding part~(2), note that, when $A_{\G}$ is twistless, all cliques $\kappa(v)\sq\G$ are singletons, because of the description of $\kappa(v)$ given in Equation~(\ref{eq:k(v)}). We have already observed that Lemma~\ref{cc properties} implies that $A_{\kappa(v)}$ is convex-cocompact in all cocompact cubulations of $A_{\G}$. Thus, each $\langle v\rangle$ is convex-cocompact in all cocompact cubulations of $A_{\G}$, which implies that all these cubulations have decomposable flats.
\end{proof}

Theorem~\ref{RAAGs intro} might (wrongly) lead us to believe that, under the right assumptions, right-angled Artin groups have only finitely many distinct cocompact cubulations. A natural guess would be that this holds when $\Out(A_{\G})$ is finite, if we restrict to cubulations with no free faces. After all, it was shown in \cite[Proposition~C]{FH} that many Burger--Mozes--Wise groups have a \emph{unique} cubulation with no free faces (while they have many essential ones).

The next example shows that such guesses are incorrect and right-angled Artin groups are way too flexible for this kind of result to hold. 

\begin{figure}
\centering
  \includegraphics[width=\textwidth]{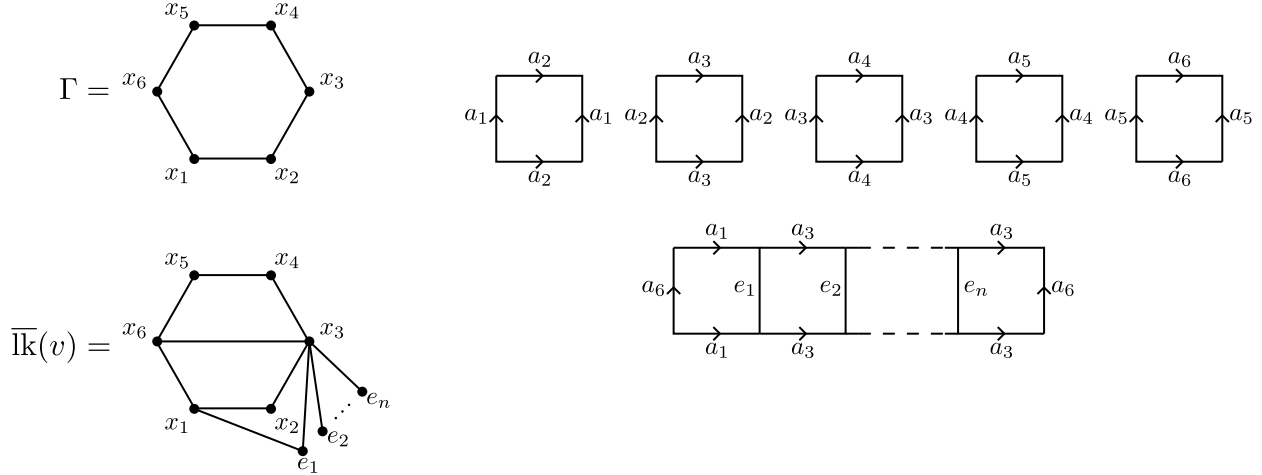}
  \caption{Top left: the hexagon graph $\G$. Right: the $n+6$ squares making up the square complex $C_n$. Bottom left: the reduced link of the single vertex of $C_n$. For clarity, the vertex of $\overline{\lk}(v)$ labelled $e_1$ has degree $2$, while the vertices labelled $e_2,\dots,e_n$ all have degree $1$.} 
   \label{hexagon}
\end{figure}

\begin{ex}\label{hexagon RAAG ex}
Let $\G$ be a hexagon, as in Figure~\ref{hexagon}. The group $\Out(A_{\G})$ is finite and $A_{\G}$ satisfies coarse cubical rigidity by Theorem~\ref{RAAGs intro}. Nevertheless, $A_{\G}$ has infinitely many $2$--dimensional cocompact cubulations with no free faces, as we are about to show.

The rough idea is that there should exist cocompact cubulations of $A_{\G}\acts X$ such that the action of each free group $\langle x_{i-1},x_{i+1}\rangle$ on its own essential core in $X$ is an arbitrary tree in the Outer Space of $\langle x_{i-1},x_{i+1}\rangle$ (modulo some compatibility conditions). We only prove a special case of this, where four of these trees are standard and the remaining two have a simple form.

For each $n\geq 0$, let $C_n$ be the finite square complex described on the right-hand side of Figure~\ref{hexagon}. Note that $C_n$ has a single vertex (call it $v$) and its $1$--skeleton is a rose with $n+6$ petals, which we name $a_1,\dots,a_6$ and $e_1,\dots, e_n$. It is clear that $C_n$ is $2$--dimensional and does not have any free faces. In addition, $C_0$ is simply the Salvetti complex of $A_{\G}$. 

Each complex $C_n$ is non-positively curved. In order to see this, it suffices to check that the link of the only vertex $v$ does not contain any $3$--cycles. Note that $\lk(v)$ contains a pair of non-adjacent vertices for each edge of $C_n$. The graph $\overline{\lk}(v)$ obtained from $\lk(v)$ by collapsing these pairs of vertices (and identifying any resulting pairs of edges with the same endpoints) is pictured in Figure~\ref{hexagon}, below on the left. Any $3$--cycle in $\lk(v)$ would give rise to a $3$--cycle in $\overline{\lk}(v)$, but the latter does not contain any, proving that $C_n$ is non-positively curved.

The description of $C_n$ given in Figure~\ref{hexagon} yields the following presentation for its fundamental group:
\begin{align*}
\pi_1(C_n,v)=\langle a_1,\dots,a_6 \mid [a_1,a_2],\ [a_2,a_3],\ [a_3,a_4],\ [a_4,a_5],\ [a_5,a_6],\ [a_6,a_1a_3^n]\rangle.
\end{align*}
This can be rewritten as follows, replacing $a_1$ with $\overline a_1:=a_1a_3^n$:
\begin{align*}
\pi_1(C_n,v)&=\langle \overline a_1,a_2\dots,a_6 \mid [\overline a_1a_3^{-n},a_2],\ [a_2,a_3],\ [a_3,a_4],\ [a_4,a_5],\ [a_5,a_6],\ [a_6,\overline a_1]\rangle \\
&=\langle \overline a_1,a_2\dots,a_6 \mid [\overline a_1,a_2],\ [a_2,a_3],\ [a_3,a_4],\ [a_4,a_5],\ [a_5,a_6],\ [a_6,\overline a_1]\rangle,
\end{align*}
where we have used the fact that $[a_2,a_3]$ is one of the relators. This shows that there exists an isomorphism $\varphi_n\colon A_{\G}\ra\pi_1(C_n,v)$ with $\varphi(x_1)=a_1a_3^n$ and $\varphi(x_i)=a_i$ for $2\leq i\leq 6$.

The deck-transformation actions $A_{\G}\acts \wt C_n$ given by the isomorphisms $\varphi_n$ are the required 2--dimensional cocompact cubulations of $A_{\G}$ with no free faces.

In order to check that these cubulations are truly pairwise distinct, we consider the essential core of the free group $\langle x_1,x_3\rangle$. First, notice that the loops $a_1,a_3$ form a \emph{convex} rose in $C_n$, which lifts to a convex tree in $\wt C_n$. Thus, the action of the free group $\langle a_1,a_3\rangle$ on its essential core in $\wt C_n$ coincides with the action of $\langle a_1,a_3\rangle$ on its standard Cayley graph (corresponding to the generating set $\{a_1,a_3\}$). Now, since $\varphi_n(x_1)=a_1a_3^n$ and $\varphi_n(x_3)=a_3$, the action of the free group $\langle x_1,x_3\rangle$ on its essential core in $\wt C_n$ is rather the point of Outer Space obtained by twisting the standard Cayley graph (corresponding to the generating set $\{x_1,x_3\}$) by the outer automorphism $(x_1,x_3)\mapsto (x_1x_3^n,x_3)$. This completes our example.

One could also wonder about the essential cores of the other free groups $\langle x_{i-1},x_{i+1}\rangle$ in $\wt C_n$. For $\langle x_2,x_4\rangle,\langle x_3,x_5\rangle,\langle x_4,x_6\rangle,\langle x_6,x_2\rangle$, the essential core is just the standard Cayley graph, whereas for $\langle x_5,x_1\rangle$ it is the subdivision of the standard Cayley graph where one of the two orbits of edges gets subdivided into $n+1$ orbits of edges. To see the latter, note that $a_1,a_3,a_5$ also form a convex rose in $C_n$. Thus, the essential core of $\langle a_1,a_3,a_5\rangle$ in $\wt C_n$ is a copy of the standard Cayley graph of $F_3$, and, in there, the minimal subtree for $\langle a_5,a_1a_3^n\rangle$ gives the essential core of $\langle x_5,x_1\rangle$.

Finally, we emphasise that the construction described in this example applies more generally to the case when $\G$ is a cycle of length at least $6$.
\end{ex}

\section{Uniformly non-quasi-convex rays}\label{nqc sect}

Having shown Theorem~\ref{RAAGs intro} and Theorem~\ref{RACGs intro}(1), we now embark in the proof of Theorem~\ref{RACGs intro}(2), which will occupy the current section and the next. This section is devoted to proving the following result, which is the most important ingredient in the proof of Theorem~\ref{RACGs intro}(2).

Recall that quasi-convexity was introduced in Definition~\ref{qc defn}. If $A\sq X$ is a subset of a $\CAT$ cube complex, we denote by $\hull_X(A)$ the smallest convex subcomplex of $X$ containing $A$.

\begin{thm}\label{uniformly non-qc ray}
Let $G\acts X$ be a cocompact cubulation. Let $M\sq X^{(0)}$ be a median subalgebra that is preserved and acted upon cofinitely by a subgroup $H\leq G$. If $M$ is not quasi-convex, then there exists a (combinatorial) ray $r\colon [0,+\infty)\ra X$ such that:
\begin{enumerate}
\item $r$ stays at bounded distance from $M$;
\item there exists a constant $K\geq 0$ such that, for all $t\geq s\geq 0$, the set $\hull_X(r|_{[s,t]})$ contains points at distance $\geq\lfloor\tfrac{t-s}{K}\rfloor$ from $M$.
\end{enumerate}
\end{thm}

This result should look rather plausible, but its proof will require significant technical work. To motivate the reader through it and emphasise that things are more delicate than they may seem, we recall the following classical example. It shows that Theorem~\ref{uniformly non-qc ray} badly fails if $M$ is replaced with just a connected subcomplex of $X$ (or its vertex set); in general, lack of quasi-convexity is not witnessed by a ray, let alone by a ``uniformly non-quasi-convex'' one. Thus, the assumption that $M$ be a median subalgebra is key and we will have to significantly rely on it.

\begin{ex}\label{ex:mapping_torus}
 Let $G=\pi_1(S)\rtimes\Z$ be the fundamental group of a fibred closed hyperbolic $3$--manifold. Let $G\acts X$ be a cocompact cubulation, whose existence is provided by \cite{Kahn-Markovic,Bergeron-Wise}. Let $H=\pi_1(S)$ be the fibre subgroup and define $M\sq X$ as a thickening of an $H$--orbit, so that $M$ is a connected subcomplex on which $H$ acts cocompactly; this is possible since $H$ is finitely generated. However, $M$ is not a median subalgebra of $X$, as $H$ is distorted in $G$.

 Since $X$ is hyperbolic, coarse-median quasi-convexity and hyperbolic quasi-convexity are the same notion. In particular, $M$ is not quasi-convex in $X$, as the subgroup $H$ is not quasi-convex in $G$. At the same time, we certainly cannot find any non-quasi-convex rays near $M$: since $X$ is hyperbolic, every ray in $X$ is quasi-convex.
\end{ex}

\subsection{Terminology, notation and conventions}

Let $X$ be a finite-dimensional $\CAT$ cube complex and $M\sq X^{(0)}$ a median subalgebra. 

All distances, neighbourhoods and geodesics should always be understood to correspond to the combinatorial metric on the $0$--skeleton of $X$. The notation $d(\cdot,\cdot)$ will always refer to this metric. We write $\mc{N}_R(A)$ for the (closed) $R$--neighbourhood in $X^{(0)}$ of a subset $A\sq X^{(0)}$.

Equipped with the restriction of the median operator of $X$, the subalgebra $M$ is a \emph{median algebra} (see \cite{Roller,Bow-cm} for some background) and we will make use of the corresponding terminology. Specifically, a subset $A\sq M$ is said to be \emph{convex} (in $M$) if $m(A,A,M)\sq A$. A subset $\mf{h}\sq M$ is a \emph{halfspace} if both $\mf{h}$ and $\mf{h}^*:=M-\mf{h}$ are convex and nonempty. A \emph{wall} of $M$ is an unordered pair $\{\mf{h},\mf{h}^*\}$, where $\mf{h}\sq M$ is a halfspace. See Remark~\ref{edge-connected rmk}(2) below for the relation between halfspaces of $M$ and $X$.

We denote by $\mscr{H}(M)$ and $\mscr{W}(M)$ the sets of halfspaces and walls of $M$, respectively. Similarly, $\mscr{H}(X)$ and $\mscr{W}(X)$ are the sets of halfspaces and hyperplanes of the cube complex $X$, or, equivalently, the halfspaces and walls of the median subalgebra $X^{(0)}$. If $A,B\sq X$ are subsets, we write:
\[\mscr{H}(A|B):=\{\mf{h}\in\mscr{H}(X) \mid A\sq\mf{h}^*,\ B\sq\mf{h}\},\]
and denote by $\mscr{W}(A|B)\sq\mscr{W}(X)$ the set of hyperplanes bounding these halfspaces. Note that $\mf{h}\in\mscr{H}(A|\mf{h})$ and $\mf{h}^*\in\mscr{H}(\mf{h}|A)$ for every subset $A\sq\mf{h}^*$.

Halfspaces $\mf{h},\mf{k}\in\mscr{H}(M)$ are \emph{transverse} if all four intersections $\mf{h}\cap\mf{k}$, $\mf{h}^*\cap\mf{k}$, $\mf{h}\cap\mf{k}^*$, $\mf{h}^*\cap\mf{k}^*$ are nonempty. In this case, we also say that the corresponding walls are \emph{transverse}. Two sets $\mc{A},\mc{B}$ of halfspaces/walls are \emph{transverse} if every element of $\mc{A}$ is transverse to every element of $\mc{B}$. If $\mf{h}\sq\mf{k}$ or $\mf{k}\sq\mf{h}$, we say that $\mf{h}$ and $\mf{k}$ are \emph{nested}.

By Chepoi--Roller duality, the median algebra $M$ is canonically isomorphic to the $0$--skeleton of a $\CAT$ cube complex $\square(M)$, equipped with its natural median operator. The sets $\mscr{W}(\square(M))$ and $\mscr{H}(\square(M))$ are naturally identified with $\mscr{W}(M)$ and $\mscr{H}(M)$. In general, $\square(M)$ is not a subcomplex of $X$, let alone a convex one\footnote{For instance, any subset of $\Z$ is a median subalgebra, if we view $\Z$ as the vertex set of the standard cubical structure on the real line.}. Note that $M$ inherits its own intrinsic metric from $\square(M)$, but this will not cause any ambiguity: in our cases of interest, the metric inherited from $\square(M)$ coincides with the restriction of the combinatorial metric of $X$ (see Remark~\ref{edge-connected rmk}(3)).

In fact, we will mostly deal with median algebras satisfying the following additional property (studied e.g.\ in \cite[Subsection~4.4.1]{Fio10a}).

\begin{defn}
A subset $A\sq X^{(0)}$ is \emph{edge-connected} if, for all $x,y\in A$, there exists a sequence of points $x_1,\dots,x_n\in A$ such that $x_1=x$, $x_n=y$ and, for all $i$, the points $x_i$ and $x_{i+1}$ are joined by an edge of $X$.
\end{defn}

\begin{rmk}\label{edge-connected rmk}
The following three observations explain our interest in edge-connected subalgebras.
\begin{enumerate}
\setlength\itemsep{.25em}
\item Suppose that $H\leq\Aut(X)$ is a finitely generated subgroup that preserves $M$ and acts cofinitely on it. Suppose further that $X$ is locally finite. Then $M$ is contained in an \emph{edge-connected}, $H$--invariant, $H$--cofinite subalgebra $M'\sq X^{(0)}$. 

Indeed, since $H$ is finitely generated, its orbits in $X$ are coarsely connected. Thus, since $M$ is $H$--cofinite, there exists $R\geq 0$ such that $\mc{N}_R(M)$ is an edge-connected subset of $X$. By \cite[Lemma~4.21]{Fio10a}, the subalgebra $M'$ generated by this subset is also edge-connected. Finally, $M'$ is at finite Hausdorff distance from $M$ since it is constructed by taking medians a bounded number of times (Lemma~\ref{lem:bounded_iteration}), hence it is $H$--cofinite because $X$ is locally finite.

\item Let $\mscr{H}_M(X)\sq\mscr{H}(X)$ be the subset of halfspaces $\mf{h}$ such that both $\mf{h}\cap M$ and $\mf{h}^*\cap M$ are nonempty. We have a map:
\begin{align*}
{\rm res}_M&\colon\mscr{H}_M(X)\ra\mscr{H}(M), & {\rm res}_M(\mf{h})&:=\mf{h}\cap M.
\end{align*}

This map is always \emph{surjective} by \cite[Lemma~6.5]{Bow-cm}. It is also clear that it is a morphism of pocsets, i.e.\ that ${\rm res}_M(\mf{h}^*)={\rm res}_M(\mf{h})^*$ and $\mf{h}\sq\mf{k}\Ra{\rm res}_M(\mf{h})\sq {\rm res}_M(\mf{k})$. On the other hand, ${\rm res}_M$ is \emph{injective} if and only if $M$ is edge-connected \cite[Lemma~4.20]{Fio10a}. 

Note that, even when $M$ is edge-connected, ${\rm res}_M$ does not preserve transversality: halfspaces $\mf{h},\mf{k}\in\mscr{H}_M(X)$ can be transverse even if ${\rm res}_M(\mf{h})$ and ${\rm res}_M(\mf{k})$ are not. In fact, this always happens, unless $M$ is the vertex set of a \emph{convex} subcomplex of $X$.

\item Denote for a moment by $d_M$ the metric on $M$ corresponding to the combinatorial metric on $\square(M)$, the cube complex given by Chepoi--Roller duality. Recall that $d$ is instead the restriction of the combinatorial metric of $X$.

The metrics $d$ and $d_M$ coincide on $M$ if and only if $M$ is edge-connected (in general, we only have $d_M\leq d$). This is an immediate consequence of Item~(2) above, recalling that, if $x,y\in M$, the distance $d_M(x,y)$ is the number of walls in $\mscr{W}(M)$ separating $x$ and $y$, while $d(x,y)$ is the number of hyperplanes in $\mscr{W}(X)$ separating them.

In addition, if $M$ is edge-connected, the dual cube complex $\square(M)$ can be realised as a subcomplex of $X$: it is simply the union of all cubes of $X$ whose $0$--skeleton is contained in $M$. In general, $\square(M)$ is not a convex subcomplex of $X$. Nevertheless, if an edge path is contained in $\square(M)$, then it is a geodesic in $X$ if and only if it is a geodesic in $\square(M)$.
\end{enumerate}
\end{rmk}

We conclude this subsection with some more notation. First, the map ${\rm res}_M\colon\mscr{H}_M(X)\ra\mscr{H}(M)$ introduced in Remark~\ref{edge-connected rmk}(2) has an obvious twin
\[{\rm res}_M\colon\mscr{W}_M(X)\ra\mscr{W}(M),\]
where $\mscr{W}_M(X)\sq\mscr{W}(X)$ is the subset of hyperplanes bounding halfspaces in $\mscr{H}_M(X)$.

We will speak of \emph{$X$--transverse} or \emph{$M$--transverse} hyperplanes/halfspaces depending on the pocset of interest: $\mscr{H}(X)$ or $\mscr{H}(M)$. This will be important in order to avoid confusion, since the maps ${\rm res}_M$ are bijections in our cases of interest and they do not preserve transversality.

For a hyperplane $\mf{w}\in\mscr{W}(X)$, we denote by $C_X(\mf{w})\sq X$ its \emph{$X$--carrier}: this is the convex subcomplex of $X$ spanned by all edges crossing $\mf{w}$. On the other hand, for a wall $\mf{u}\in\mscr{W}(M)$, we denote by $C_M(\mf{u})\sq M$ its \emph{$M$--carrier}: this is the intersection between the $0$--skeleton of $\square(M)$ and the carrier in $\square(M)$ of the hyperplane of $\square(M)$ determined by $\mf{w}$. In other words, a point $x\in M$ lies in $C_M(\mf{u})$ when there exists $y\in M$ such that $\mf{u}$ is the only wall of $M$ separating $x$ and $y$.

Finally, if $\mf{h}$ is a halfspace bounded by a hyperplane/wall $\mf{w}$, it is convenient to write $C_X(\mf{h})$ and $C_M(\mf{h})$ for the intersections $C_X(\mf{w})\cap\mf{h}$ and $C_M(\mf{w})\cap\mf{h}$.

\begin{rmk}\label{different carriers}
For every $\mf{w}\in\mscr{W}_M(X)$, we have $C_M({\rm res}_M(\mf{w}))\sq C_X(\mf{w})\cap M$, but in general this inclusion can be very far from an equality, even at the level of $0$--skeleta. 

For instance, let $X$ be the strip $\R\x[0,1]$ with its standard decomposition into squares, let $\mf{w}$ be the hyperplane corresponding to the $[0,1]$--factor, and let $M$ be the $0$--skeleton of the subcomplex $(\R\x\{0\})\cup(\{0\}\x[0,1])$. Then $C_X(\mf{w})\cap M=M$, while $C_M({\rm res}_M(\mf{w}))=\{(0,0),(0,1)\}$.

Luckily, when $X$ and $M$ have geometric group actions, the distinction between $C_M({\rm res}_M(\mf{w}))$ and $C_X(\mf{w})\cap M$ will not be quite so drastic. See Lemma~\ref{similar carriers 2} below.
\end{rmk}

We record here also the following observation for later use.

\begin{rmk}\label{qc hulls rmk}
Let $A\sq X^{(0)}$ be a subset of a $\CAT$ cube complex. Set $\delta:=\dim X$.
\begin{enumerate}
\item If $m(A,A,X)\sq\mc{N}_R(A)$, then $\hull_X(A)\sq\mc{N}_{2^{\delta}R}(A)$. This is a straightforward consequence of the fact that the median operator is $1$--Lipschitz, together with the fact that hulls are constructed by taking medians a bounded number of times (Lemma~\ref{lem:bounded_iteration}).
\item In particular, $A\sq X$ is quasi-convex (in the sense of Definition~\ref{qc defn}) if and only if $A$ is at finite Hausdorff distance from $\hull_X(A)$.
\end{enumerate}
\end{rmk}

\subsection{Quadrants, grids and quasi-convexity}

Let $X$ be a finite-dimensional $\CAT$ cube complex and let $M\sq X^{(0)}$ be a median subalgebra. 

The goal of this subsection is to rephrase (failure of) quasi-convexity for the subalgebra $M$ in terms of ``grids'' of hyperplanes of $X$. The important result in this regard is Corollary~\ref{non-qc via grids}.

Just like a \emph{convex} subcomplex of $X$ is the complement of the union of \emph{halfspaces} disjoint from it, an (edge-connected) median \emph{subalgebra} of $X$ is the complement of the union of \emph{quadrants} disjoint from it. The next lemma proves a stronger version of this fact.

Recall that the subset $\mscr{H}_M(X)\sq\mscr{H}(X)$ was introduced in Remark~\ref{edge-connected rmk}(2).

\begin{lem}\label{separate C and M}
Let $C\sq X$ be a convex subcomplex with $C\cap\hull_X(M)\neq\emptyset$. 
\begin{enumerate}
\item If $C\cap M=\emptyset$, there exist distinct halfspaces $\mf{h},\mf{k}\in\mscr{H}_M(X)$ with $C\sq\mf{h}\cap\mf{k}$ and $M\sq\mf{h}^*\cup\mf{k}^*$.
\item If, in addition, $M$ is edge-connected, then $\mf{h}$ and $\mf{k}$ are $X$--transverse.
\end{enumerate}
\end{lem}
\begin{proof}
The pocset $\mscr{H}_M(X)\sq\mscr{H}(X)$ is naturally identified with $\mscr{H}(\hull_X(M))$, and a halfspace $\mf{h}\in\mscr{H}_M(X)$ contains $C$ if and only if it contains $C\cap\hull_X(M)$. Thus, up to replacing $X$ with $\hull_X(M)$ and $C$ with $C\cap\hull_X(M)$, we can assume that $\mscr{H}_M(X)=\mscr{H}(X)$.

Now, we prove part~(1). Consider the subset $\s_C\sq\mscr{H}(M)$ defined as follows:
\[\s_C:=\{\mf{h}\cap M \mid \mf{h}\in\mscr{H}(X) \text{ and } C\sq\mf{h}\}.\]

Note that $\s_C$ cannot contain any infinite descending chains $\mf{h}_1\cap M\supsetneq\mf{h}_2\cap M\supsetneq\dots$. Indeed, the halfspaces $\mf{h}_i\in\mscr{H}(X)$ would be pairwise distinct, giving infinitely many elements of $\mscr{H}(p|C)$. This is impossible since $d(p,C)$ is finite.

Also note that the elements of $\s_C\sq\mscr{H}(M)$ cannot pairwise intersect. Otherwise, since $\s_C$ does not contain any infinite descending chains, there would be a point $p\in M$ that lies in all elements of $\s_C$ (this is clear viewing $M$ as the $0$--skeleton of the cube complex $\square(M)$ via Chepoi--Roller duality). By the definition of $\s_C$, we would also have $p\in C$, contradicting the assumption that $C\cap M=\emptyset$. 

In conclusion, there exist two disjoint elements of $\s_C$. This means that there exist halfspaces $\mf{h},\mf{k}\in\mscr{H}(X)=\mscr{H}_M(X)$ such that $C\sq\mf{h}\cap\mf{k}$ and $M\cap\mf{h}\cap\mf{k}=\emptyset$. This proves part~(1).

Regarding part~(2), recall that $\mf{h}$ and $\mf{k}$ intersect $M$, as they lie in $\mscr{H}_M(X)$. Since $M\cap\mf{h}\cap\mf{k}=\emptyset$, the intersections $\mf{h}\cap\mf{k}^*$ and $\mf{h}^*\cap\mf{k}$ must intersect $M$, hence they are nonempty. The intersection $\mf{h}\cap\mf{k}$ is also nonempty, as it contains $C$. 

Finally, if $M$ is edge-connected, it contains the vertex set of an edge-path in $X$ joining a point of $M\cap\mf{h}\cap\mf{k}^*$ to a point of $M\cap\mf{h}^*\cap\mf{k}$. Since this path cannot intersect $\mf{h}\cap\mf{k}$, which is disjoint from $M$, it must intersect $\mf{h}^*\cap\mf{k}^*$. This proves that $\mf{h}^*\cap\mf{k}^*$ is nonempty, hence $\mf{h}$ and $\mf{k}$ are $X$--transverse.
\end{proof}

The next lemma shows that, if a convex subcomplex and a median subalgebra of $X$ are far from each other, then they are separated by a large grid of hyperplanes of $X$. This configuration is depicted in Figure~\ref{img2}.

\begin{lem}\label{get n-grid}
Let $M$ be edge-connected. Let $C\sq X$ be a convex subcomplex with $C\cap\hull_X(M)\neq\emptyset$. Set $\delta:=\dim X$. If $d(C,M)>(2^\delta+1) \delta n$ for some $n\geq 0$, then there exist halfspaces $\mf{h}_0\supsetneq\dots\supsetneq\mf{h}_n$ and $\mf{k}_0\supsetneq\dots\supsetneq\mf{k}_n$ in $\mscr{H}_M(X)$ such that each $\mf{h}_i$ is $X$--transverse to every $\mf{k}_j$, and $C\sq\mf{h}_n\cap\mf{k}_n$ and $M\sq\mf{h}_0^*\cup\mf{k}_0^*$.
\end{lem}
\begin{proof}
Set $C':=\hull_X(\mc{N}_{\delta n}(C))$. Since the median operator is $1$--Lipschitz in each coordinate, we have $m(\mc{N}_{\delta n}(C),\mc{N}_{ \delta n}(C),X)\sq\mc{N}_{2 \delta n}(C)$, hence Remark~\ref{qc hulls rmk} guarantees that $C'\sq \mc{N}_{(2^\delta+1)\delta n}(C)$. It follows that $C'\cap M=\emptyset$.

By Lemma~\ref{separate C and M}, there exist $X$--transverse halfspaces $\mf{h},\mf{k}\in\mscr{H}_M(X)$ such that $C'\sq\mf{h}\cap\mf{k}$ and $M\sq\mf{h}^*\cup\mf{k}^*$. Since $\mc{N}_{\delta n}(C)\sq C'$, we have $d(C,\mf{h}^*)> \delta n$ and $d(C,\mf{k}^*)>\delta n$. 

By Dilworth's lemma, the sets $\mscr{H}(\mf{h}^*|C)$ and $\mscr{H}(\mf{k}^*|C)$ each contain a chain of halfspaces of length $n+1$.
% here these sets **should** also include \mf{h} and \mf{k}
Also note that $\mscr{H}(\mf{h}^*|C)$ and $\mscr{H}(\mf{k}^*|C)$ are contained in $\mscr{H}_M(X)$, since $\mf{h},\mf{k}\in\mscr{H}_M(X)$ and $C\cap\hull_X(M)\neq\emptyset$. Finally, each $\mf{j}\in\mscr{H}(\mf{h}^*|C)$ is $X$--transverse to each $\mf{j}'\in\mscr{H}(\mf{k}^*|C)$: the proof of this fact is identical to that of Lemma~\ref{separate C and M}(2). This concludes the proof of the lemma.
\end{proof}

\begin{cor}\label{non-qc via grids}
Let $M$ be edge-connected. Then $M$ fails to be quasi-convex if and only if, for every $n\geq 0$, there exist halfspaces $\mf{h}_0\supsetneq\dots\supsetneq\mf{h}_n$ and $\mf{k}_0\supsetneq\dots\supsetneq\mf{k}_n$ in $\mscr{H}_M(X)$ such that each $\mf{h}_i$ is $X$--transverse to every $\mf{k}_j$, and $M\sq\mf{h}_0^*\cup\mf{k}_0^*$.
\end{cor}
\begin{proof}
Recall from Remark~\ref{qc hulls rmk} that $M$ is quasi-convex if and only if it is at finite Hausdorff distance from $\hull_X(M)$. Now, if $\hull_X(M)$ contains points arbitrarily far from $M$, then Lem\-ma~\ref{get n-grid} yields the required hyperplane grids. Conversely, if there exist halfspaces $\mf{h}_i,\mf{k}_j$ as in the statement of the corollary, then $\hull_X(M)\cap\mf{h}_n\cap\mf{k}_n\neq\emptyset$ by Helly's lemma and $d(\mf{h}_n\cap\mf{k}_n,M)>n$. Hence $\hull_X(M)$ contains points arbitrarily far from $M$, showing that $M$ is not quasi-convex.
\end{proof}

\begin{figure}
\centering
\begin{minipage}{.5\textwidth}
  \centering
  \includegraphics[scale=.85]{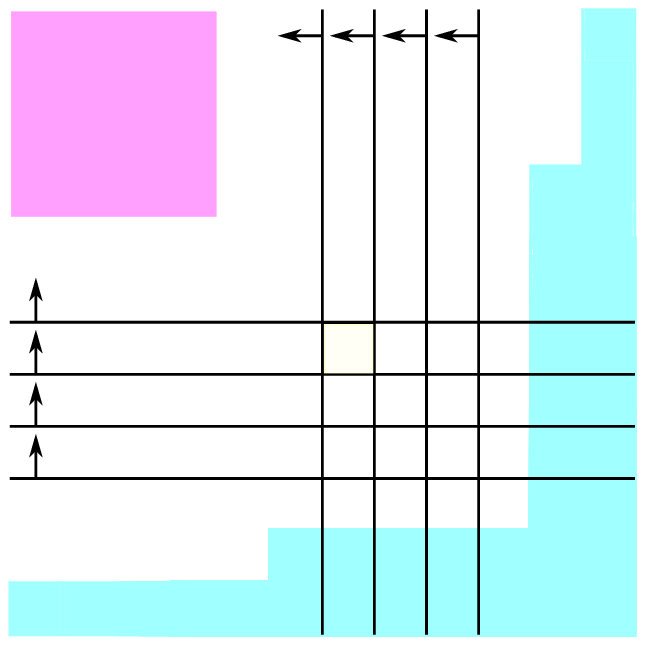}
  \put(-133,125){$C$}
  \put(-20,10){$M$}
  \put(-37,146){$\mf{h}_0$}
  \put(-150,29){$\mf{k}_0$}
  \captionof{figure}{Grid of hyperplanes separating $M$ and $C$ in Lemma~\ref{get n-grid}.}
  \label{img2}
\end{minipage}%
\begin{minipage}{.5\textwidth}
  \centering
  \includegraphics[scale=.85]{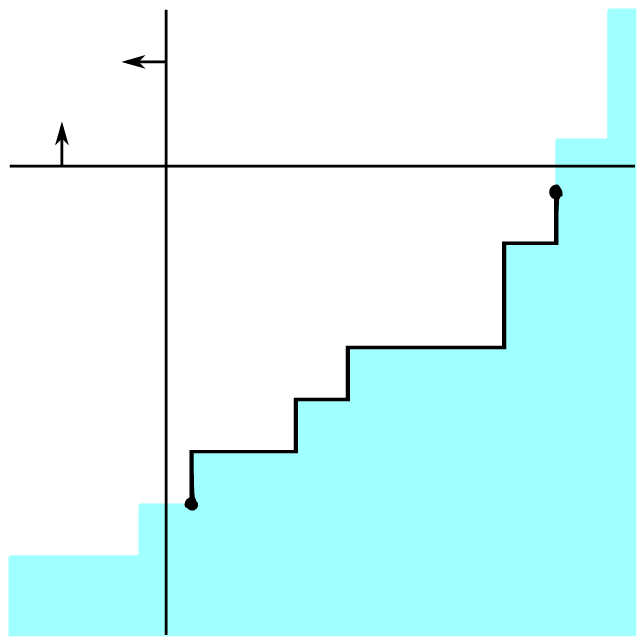}
  \put(-65,75){$\alpha$}
  \put(-30,20){$M$}
  \put(-150,120){$\mf{k}$}
  \put(-124,147){$\mf{h}$}
  \put(-112,24){$q$}
  \put(-15,107){$p$}
  \captionof{figure}{The geodesic $\alpha$ obtained in Construction~\ref{geod construction}.}
  \label{img3}
\end{minipage}
\end{figure}

\subsection{Constructing uniformly non-quasi-convex geodesics}

Let $X$ be a $\CAT$ cube complex with a geometric action $G\acts X$. Let $M\sq X^{(0)}$ be a median subalgebra. We assume that $M$ is edge-connected and that there exists a subgroup $H\leq G$ acting on $M$ with finitely many orbits. 

The following is the main result of this subsection. We will quickly deduce Theorem~\ref{uniformly non-qc ray} from it in the next subsection.

Recall that $\square(M)$ is canonically realised as a (non-convex) subcomplex of $X$, since $M$ is edge-connected (Remark~\ref{edge-connected rmk}(3)). Hyperplanes of $\square(M)$ are precisely intersections with $\square(M)$ of hyperplanes of $X$ (though this is not true of carriers). Also recall that we defined the carrier of a halfspace as the intersection of the halfspace with the carrier of the corresponding wall/hyperplane.

\begin{prop}\label{unqc prop}
There exists a constant $K'\geq 0$ such that the following holds. For all $X$--transverse halfspaces $\mf{h},\mf{k}\in\mscr{H}_M(X)$ with $M\cap\mf{h}\cap\mf{k}=\emptyset$, there exists a geodesic $\alpha\sq\square(M)\sq X$ from a vertex in the $M$--carrier $C_M({\rm res}_M(\mf{k}^*))$ to a vertex in $C_M({\rm res}_M(\mf{h}^*))$ with the following property. For all integers $t\geq s\geq 0$, the set $\hull_X(\alpha|_{[s,t]})$ contains points at distance $>\lfloor\tfrac{t-s}{K'}\rfloor$ from $M$.
\end{prop}

Before embarking in the proof of Proposition~\ref{unqc prop}, we need to make the following key observation. It exploits cocompactness of $M$ and $X$ to deduce that $M$--carriers are not too different from intersections of $X$--carriers with $M$ (recall Remark~\ref{different carriers}). We denote by $I(x,y)$ the \emph{interval} in $X$ with endpoints $x$ and $y$, that is, the union of all combinatorial geodesics in $X$ joining $x$ and $y$.

\begin{lem}\label{similar carriers 2}
There exists a constant $K\geq 0$ with the following property. Consider a halfspace $\mf{h}\in\mscr{H}_M(X)$ and a geodesic $\beta\sq\square(M)\sq X$ connecting a point $x\in C_M({\rm res}_M(\mf{h}))$ to a point $y\in M\cap C_X(\mf{h})$. If $d(x,y)>K$, then $I(x,y)\cap C_M({\rm res}_M(\mf{h}))$ contains a vertex other than $x$.
\end{lem}
\begin{proof}
For every halfspace $\mf{h}\in\mscr{H}_M(X)$, its stabiliser $G(\mf{h})\leq G$ acts cocompactly on $C_X(\mf{h})$. Since $H$ acts cofinitely on $M$, Lemma~\ref{cocompact intersections} shows that the intersection $H\cap G(\mf{h})$ acts cofinitely on the set $M\cap C_X(\mf{h})$. Note that there are only finitely many $H$--orbits of halfspaces in $\mscr{H}_M(X)$, since this set is equivariantly in bijection with $\mscr{H}(M)$. Thus, there exists a constant $K\geq 0$ such that, for every $\mf{h}\in\mscr{H}_M(X)$, all orbits of the action $H\cap G(\mf{h})\acts M\cap C_X(\mf{h})$ are $K$--dense.

The set $M\cap C_X(\mf{h})$ contains the subset $C_M({\rm res}_M(\mf{h}))$, which is $(H\cap G(\mf{h}))$--invariant and non\-empty. Hence, for every $\mf{h}\in\mscr{H}_M(X)$, the subset $C_M({\rm res}_M(\mf{h}))$ is $K$--dense in $M\cap C_X(\mf{h})$.

Now, consider points $x\in C_M({\rm res}_M(\mf{h}))$ and $y\in M\cap C_X(\mf{h})$ with $d(x,y)>K$. By the previous paragraph, there exists a point $z\in C_M({\rm res}_M(\mf{h}))$ with $d(z,y)\leq K$. Setting $w:=m(x,y,z)$, we have $w\in M\cap I(x,y)$ and $w\neq x$. We are left to show that $w$ lies in $C_M({\rm res}_M(\mf{h}))$.

Since $x,z\in C_M({\rm res}_M(\mf{h}))$, there exist $x',z'\in C_M({\rm res}_M(\mf{h}^*))$ with $d(x,x')=d(z,z')=1$. Consider the point $w'=m(x',y,z')$. It is clear that $w'\in M\cap\mf{h}^*$ and that $d(w,w')=1$. This shows that $w\in C_M({\rm res}_M(\mf{h}))$, concluding the proof.
\end{proof}

Now, we begin by describing how to construct the geodesic $\alpha$ appearing in Proposition~\ref{unqc prop}. Then we will show that it satisfies the required properties, which will take a few lemmas. The reader might find Figure~\ref{img3} helpful while working through Construction~\ref{geod construction} and the subsequent Lemma~\ref{properties of alpha}.

\begin{cstr}\label{geod construction}
Consider $X$--transverse halfspaces $\mf{h},\mf{k}\in\mscr{H}_M(X)\sq\mscr{H}(X)$ with $M\cap\mf{h}\cap\mf{k}=\emptyset$. Since ${\rm res}_M(\mf{h})$ and ${\rm res}_M(\mf{k})$ are disjoint halfspaces of $M$, we can consider their \emph{bridge} $\mc{B}\sq\square(M)$. This is the convex subset of $M$ that is the union of all geodesics from $C_M({\rm res}_M(\mf{h}^*))$ to $C_M({\rm res}_M(\mf{k}^*))$; see \cite[Section~2.G]{CFI} or \cite[Section~2.2]{Fio3} for a precise definition and its properties. 

Let $p$ be any point in $\mc{B}\cap C_M({\rm res}_M(\mf{k}^*))$. Let $q\in\mc{B}\cap C_M({\rm res}_M(\mf{h}^*))$ be the unique point such that every hyperplane (of $\square(M)$ or $X$) separating $p$ and $q$ also separates ${\rm res}_M(\mf{h})$ and ${\rm res}_M(\mf{k})$; this point exists by the properties of bridges.

Let $\mc{G}_M(p,q)$ be the set of all geodesics from $p$ to $q$ contained in $\square(M)\sq X$. For $\beta,\beta'\in\mc{G}_M(p,q)$, write $\beta\prec\beta'$ if there exists an integer $0\leq t_0\leq d(p,q)$ such that: 
\begin{align*}
d(\beta(t_0),\mf{k})&<d(\beta'(t_0),\mf{k}), & &\text{and} & d(\beta(t),\mf{k})&=d(\beta'(t),\mf{k}), \text{ for all } 0\leq t< t_0.
\end{align*}
This is a total order on the set of equivalence classes $\mathclose{\mc{G}_M(p,q)}/\mathopen{\sim}$, where we write $\beta\sim\beta'$ if we have $d(\beta(t),\mf{k})=d(\beta'(t),\mf{k})$ for all $0\leq t\leq d(p,q)$. We emphasise that we are interested in distances to $\mf{k}$, not ${\rm res}_M(\mf{k})$.

Since $\mc{G}_M(p,q)$ is finite, there exists a $\prec$--minimal element $\alpha$. This will be our geodesic. We orient $\alpha$ from $p$ to $q$.
\end{cstr}

Observe that, for every halfspace $\mf{j}\in\mscr{H}(p|{\rm res}_M(\mf{h}))\sq\mscr{H}(X)$, we have:
\begin{align*} 
\mf{j}\cap\mf{k}^*&\supset {\rm res}_M(\mf{h})\neq\emptyset, & \mf{j}^*\cap\mf{k}^*&\ni p, & \mf{j}^*\cap\mf{k}&\supset{\rm res}_M(\mf{k})\neq\emptyset,
\end{align*} 
where the third equation is due to the fact that walls of $M$ crossing the bridge $\mc{B}$ must either separate ${\rm res}_M(\mf{h})$ and ${\rm res}_M(\mf{k})$, or be $M$--transverse to both. The fourth intersection $\mf{j}\cap\mf{k}$ can be empty or not, and this gives rise to a partition:
\[\mscr{H}(p|{\rm res}_M(\mf{h}))=\Om_{\perp}\sqcup\Om_{\parallel}.\]
More precisely, for every $\mf{j}\in\Om_{\parallel}$, we have $\mf{j}\cap\mf{k}=\emptyset$ in $X$. Instead, each $\mf{j}\in\Om_{\perp}$ is $X$--transverse to $\mf{k}$ (although we have ${\rm res}_M(\mf{j})\cap{\rm res}_M(\mf{k})=\emptyset$ by the properties of bridges). Note that $\mf{h}\in\Om_{\perp}$.

We say that a segment $\beta\sq\alpha$ is a \emph{$\parallel$--segment} (resp.\ a \emph{$\perp$--segment}) if all halfspaces entered by $\beta$ lie in $\Om_{\parallel}$ (resp.\ in $\Om_{\perp}$). The next lemma collects the key properties of the geodesic $\alpha$. 

\begin{lem}\label{properties of alpha}
The oriented geodesic $\alpha$ obtained in Construction~\ref{geod construction} satisfies the following.
\begin{enumerate}
\item If some $\mf{j}_{\parallel}\in\Om_{\parallel}$ is entered by $\alpha$ before some $\mf{j}_{\perp}\in\Om_{\perp}$, then $\mf{j}_{\parallel}$ is $X$--transverse to $\mf{j}_{\perp}$.
\item All $\parallel$--segments and $\perp$--segments of $\alpha$ have length $\leq K$, with $K$ as in Lemma~\ref{similar carriers 2}.
\item If $\alpha$ contains a $\parallel$--segment $\beta_{\parallel}$ immediately followed by a $\perp$--segment $\beta_{\perp}$, and if $\mf{j}_{\parallel}$ and $\mf{j}_{\perp}$ are any halfspaces entered, respectively, by $\beta_{\parallel}$ and $\beta_{\perp}$, then $M\cap\mf{j}_{\parallel}^*\cap\mf{j}_{\perp}=\emptyset$.
\end{enumerate}
\end{lem}
\begin{proof}
Property~(1) is almost immediate. For every $\mf{j}_{\parallel}\in\Om_{\parallel}$ and $\mf{j}_{\perp}\in\Om_{\perp}$, the intersection $\mf{j}_{\parallel}^*\cap\mf{j}_{\perp}$ is nonempty, as it contains $\mf{k}\cap\mf{j}_{\perp}\neq\emptyset$. Thus, if $\alpha$ enters $\mf{j}_{\parallel}$ before $\mf{j}_{\perp}$, these two halfspaces must be $X$--transverse.

We now prove Property~(2), beginning with some preliminary remarks. 

In Construction~\ref{geod construction}, we introduced an equivalence relation $\sim$ on $\mc{G}_M(p,q)$ and a total order $\prec$ on its set of equivalence classes. What matters for these relations is the function $t\mapsto d(\beta(t),\mf{k})$, where $\beta\in\mc{G}_M(p,q)$. In turn, this function is completely determined by the order in which $\beta$ enters elements of $\Om_{\parallel}$ and $\Om_{\perp}$: entering some $\mf{j}\in\Om_{\perp}$ does not change the value of $d(\beta(t),\mf{k})$, since $\mf{j}$ is $X$--transverse to $\mf{k}$; on the other hand, entering some $\mf{j}\in\Om_{\parallel}$ increases $d(\beta(t),\mf{k})$ by $1$, since $\mf{j}\cap\mf{k}=\emptyset$.

If $\beta\in\mc{G}_M(p,q)$ enters halfspaces $\mf{j}_1,\mf{j}_2$ consecutively and if $\mf{j}_1,\mf{j}_2$ are $M$--transverse, then there exists another geodesic $\beta'\in\mc{G}_M(p,q)$ only differing from $\beta$ in the fact that it enters $\mf{j}_2$ before $\mf{j}_1$. If $\mf{j}_1,\mf{j}_2$ both lie in $\Om_{\parallel}$, or both lie in $\Om_{\perp}$, then $\beta'\sim\beta$. Instead, if $\mf{j}_1\in\Om_{\parallel}$ and $\mf{j}_2\in\Om_{\perp}$, then $\beta'\prec\beta$.

With these observations in hand, we complete the proof of Property~(2) by the following two claims. Let $\alpha_{\parallel}\sq\alpha$ and $\alpha_{\perp}\sq\alpha$ be, respectively, a maximal $\parallel$--segment and a maximal $\perp$--segment. Let $K$ be the constant provided by Lemma~\ref{similar carriers 2}.

\smallskip
{\bf Claim~1.} \emph{The segment $\alpha_{\perp}$ has length $\leq K$.}

\smallskip
\emph{Proof of Claim~1.}
First, suppose that $\alpha_{\perp}$ is an initial segment of $\alpha$. Since all elements of $\Om_{\perp}$ are $X$--transverse to $\mf{k}$, we have $\alpha_{\perp}\sq C_X(\mf{k}^*)$. Every halfspace $\mf{j}$ entered by $\alpha_{\perp}$ satisfies ${\rm res}_M(\mf{j})\cap{\rm res}_M(\mf{k})=\emptyset$, since $\alpha$ joins the points $p$ and $q$ in the bridge $\mc{B}$. Thus, every point of $\hull_X(\alpha_{\perp})$ (other than $p$) lies in the difference $C_X(\mf{k}^*)-C_M({\rm res}_M(\mf{k}^*))$. Lemma~\ref{similar carriers 2} then shows that $\alpha_{\perp}$ has length $\leq K$.

Suppose now instead that $\alpha_{\perp}$ is not an initial segment of $\alpha$. Let $\mf{w}\in\mscr{W}(X)$ be the last hyperplane crossed by $\alpha$ before the start of $\alpha_{\perp}$; let $e\sq\alpha$ be the edge crossing $\mf{w}$. The hyperplane $\mf{w}$ bounds an element of $\Om_{\parallel}$, by maximality of $\alpha_{\perp}$. So, by part~(1), $\mf{w}$ is $X$--transverse to every hyperplane crossed by $\alpha_{\perp}$ and hence $\alpha_{\perp}\sq C_X(\mf{w})$. If $\alpha_{\perp}$ had length $>K$, Lemma~\ref{similar carriers 2} would imply the existence of a point $x\in\hull_X(\alpha_{\perp})\cap C_M({\rm res}_M(\mf{w}))$ other than the initial vertex of $\alpha_{\perp}$. Let $x'\in M$ be the point with $\mscr{W}(x|x')=\{\mf{w}\}$. Then, replacing the segment $e\cup\alpha_{\perp}\sq\alpha$ with a geodesic in $\square(M)\sq X$ passing through $x'$ and $x$, we would find an element $\alpha'\in\mc{G}_M(p,q)$ with $\alpha'\prec\alpha$, a contradiction.
\hfill$\blacksquare$ 

\smallskip
{\bf Claim~2.} \emph{The segment $\alpha_{\parallel}$ has length $\leq K$.}

\smallskip
\emph{Proof of Claim~2.}
This is entirely analogous to the previous proof.

First, suppose that $\alpha_{\parallel}$ is a terminal segment of $\alpha$. Note that all elements of $\Om_{\parallel}$ are $X$--transverse to $\mf{h}$ (by part~(1), since $\mf{h}\in\Om_{\perp}$).
% is this the only spot where we use the (important!) assumption that $\mf{h}$ is $X$--transverse to $\mf{k}$?
So we have $\alpha_{\parallel}\sq C_X(\mf{h}^*)$. Again, by the properties of bridges, every halfspace $\mf{j}$ entered by $\alpha_{\parallel}$ satisfies ${\rm res}_M(\mf{j}^*)\cap{\rm res}_M(\mf{h})=\emptyset$. As in Claim~1, every point of $\hull_X(\alpha_{\parallel})$ (other than $q$) lies in the difference $C_X(\mf{h}^*)-C_M({\rm res}_M(\mf{h}^*))$ and Lemma~\ref{similar carriers 2} shows that $\alpha_{\parallel}$ has length $\leq K$.

Suppose instead that $\alpha_{\parallel}$ is not a terminal segment of $\alpha$. Let $\mf{w}\in\mscr{W}(X)$ be the first hyperplane crossed by $\alpha$ after the end of $\alpha_{\parallel}$; let $e\sq\alpha$ be the edge crossing $\mf{w}$. The hyperplane $\mf{w}$ bounds an element of $\Om_{\perp}$, so, by part~(1), it is $X$--transverse to every hyperplane crossed by $\alpha_{\parallel}$ and $\alpha_{\parallel}\sq C_X(\mf{w})$. If $\alpha_{\parallel}$ had length $>K$, Lemma~\ref{similar carriers 2} would imply the existence of a point $x\in\hull_X(\alpha_{\parallel})\cap C_M({\rm res}_M(\mf{w}))$ other than the terminal endpoint of $\alpha_{\parallel}$. If $x'\in M$ is the point with $\mscr{W}(x|x')=\{\mf{w}\}$, we can form a new geodesic $\alpha'\in\mc{G}_M(p,q)$ by replacing the segment $e\cup\alpha_{\parallel}\sq\alpha$ with a geodesic in $\square(M)\sq X$ that follows $\alpha_{\parallel}$ up to $x$, then crosses $\mf{w}$ to reach $x'$ and finally moves to the vertex in $e-\alpha_{\parallel}$. As before, we have $\alpha'\prec\alpha$, a contradiction.
\hfill$\blacksquare$ 

\smallskip
Finally, we prove Property~(3). Consider $\beta_{\parallel},\beta_{\perp}$ and $\mf{j}_{\parallel},\mf{j}_{\perp}$ as in the statement. Let $x$ be the point where $\beta_{\parallel}$ and $\beta_{\perp}$ meet. Suppose for the sake of contradiction that $M\cap\mf{j}_{\parallel}^*\cap\mf{j}_{\perp}\neq\emptyset$. Then $\mf{j}_{\parallel}$ and $\mf{j}_{\perp}$ are $M$--transverse (since $\mf{j}_{\parallel}$ is entered before $\mf{j}_{\perp}$ by a geodesic contained in $\square(M)$).

Without loss of generality, suppose that $x\in C_M({\rm res}_M(\mf{j}_{\parallel}))\cap C_M({\rm res}_M(\mf{j}_{\perp}^*))$; otherwise it suffices to replace $\mf{j}_{\parallel}$ and $\mf{j}_{\perp}$ by a different pair of $M$--transverse halfspaces entered by $\beta_{\parallel},\beta_{\perp}$, corresponding to walls in $\mscr{W}(M)$ that are closer to $x$. This does not affect the assumption that $\mf{j}_{\parallel}$ and $\mf{j}_{\perp}$ are $M$--transverse.

Up to changing the order in which $\beta_{\parallel}$ and $\beta_{\perp}$ enter the respective halfspaces (which yields $\alpha'\sim\alpha$), we can also assume that $\mf{j}_{\parallel}$ is the last halfspace entered before $x$ and that $\mf{j}_{\perp}$ is the first entered after it. Now, since $\mf{j}_{\parallel}$ and $\mf{j}_{\perp}$ are $M$--transverse, we can swap the order in which these halfspaces are entered to produce $\alpha''\prec\alpha$, a contradiction.
\end{proof}

We only need one last simple combinatorial lemma before proving Proposition~\ref{unqc prop}.

\begin{lem}\label{combinatorial lemma}
Suppose some geodesic in $X$ enters halfspaces $\mf{h}_1,\dots,\mf{h}_N$ in this order (not necessarily consecutively), and also halfspaces $\mf{k}_1,\dots,\mf{k}_N$ in this order. Consider $n\geq 0$. If $N>2n\dim X$, there exist an index $1\leq k\leq N$ and indices $1\leq i_1<\dots<i_n<k$ and $k<j_1<\dots<j_n\leq N$ such that:
\begin{align*}
\mf{h}_{i_1}&\supset\mf{h}_{i_2}\supset\dots\supset\mf{h}_{i_n}\supset\mf{h}_k, & \mf{k}_k\supset\mf{k}_{j_1}&\supset\mf{k}_{j_2}\supset\dots\supset\mf{k}_{j_n}.
\end{align*}
\end{lem}
\begin{proof}
Define functions $f,g\colon\{1,\dots,N\}\ra\{0,\dots,N\}$ as follows. The integer $f(k)$ is the longest length of a chain of $\mf{h}_i$ all strictly containing $\mf{h}_k$, while $g(k)$ is the longest length of a chain of $\mf{k}_i$ all strictly contained in $\mf{k}_k$. We need an index $k$ so that $f(k)\geq n$ and $g(k)\geq n$ hold simultaneously.

For an integer $i$, consider the halfspaces $\mf{h}_j$ with $f(j)=i$. It is clear that they must be pairwise transverse. Hence $\#f^{-1}(i)\leq\dim X$ and, similarly, $\#g^{-1}(i)\leq\dim X$ for every $1\leq i\leq N$. It follows that $\#f^{-1}([0,n-1])\leq n\dim X$ and $\#g^{-1}([0,n-1])\leq n\dim X$, which implies the lemma.
\end{proof}

The next result immediately implies Proposition~\ref{unqc prop}.

\begin{cor}
Let $\alpha$ be the geodesic obtained in Construction~\ref{geod construction}. Consider $n\geq 0$. For all integers $0\leq s\leq t\leq d(p,q)$ with $t-s\geq 2K(2n \dim X +2)$, the segment $\alpha|_{[s,t]}$ enters halfspaces $\mf{h}_0,\dots,\mf{h}_n\in\Om_{\perp}$ and $\mf{k}_0,\dots,\mf{k}_n\in\Om_{\parallel}$ such that, for all $i,j$, we have $M\cap\mf{h}_i\cap\mf{k}_j^*=\emptyset$ and $\mf{h}_i$ is $X$--transverse to $\mf{k}_j$. In particular, $\hull_X(\alpha|_{[s,t]})$ contains points at distance $>n$ from $M$.
\end{cor}
\begin{proof}
After an initial segment of length $\leq K$, the geodesic $\alpha|_{[s,t]}$ contains a $\parallel$--segment $\beta_1$ followed by a $\perp$--segment $\gamma_1$, and so on up to a $\parallel$--segment $\beta_N$ and a $\perp$--segment $\gamma_N$, for some $N\geq 0$. Each $\beta_i$ and $\gamma_i$ has length $\leq K$ by Lemma~\ref{properties of alpha}(2) and the sum of their lengths is $\geq t-s-2K$. Thus $N\geq\lfloor\tfrac{t-s-2K}{2K}\rfloor>2n\dim X$.

Let $\mf{h}_i\in\Om_{\parallel}$ and $\mf{k}_i\in\Om_{\perp}$ be arbitrary halfspaces entered by $\beta_i$ and $\g_i$, respectively. By Lemma~\ref{properties of alpha}(3), we have $M\cap\mf{h}_i^*\cap\mf{k}_i=\emptyset$, hence ${\rm res}_M(\mf{k}_i)\sq {\rm res}_M(\mf{h}_i)$. Lemma~\ref{combinatorial lemma} (applied to the cube complex $\square(M)$) ensures the existence of a chain:
\[ {\rm res}_M(\mf{k}_{j_1})\subset\dots\subset {\rm res}_M(\mf{k}_{j_n}) \subset {\rm res}_M(\mf{k}_k)\subset {\rm res}_M(\mf{h}_k) \subset {\rm res}_M(\mf{h}_{i_n})\subset\dots\subset{\rm res}_M(\mf{h}_{i_1}).\]
It follows that $M\cap\mf{h}_{i_a}^*\cap\mf{k}_{j_b}=\emptyset$ for all indices $a,b$, while Lemma~\ref{properties of alpha}(1) guarantees that $\mf{h}_{i_a}$ and $\mf{k}_{j_b}$ are $X$--transverse. These are the required halfspaces.

Finally, by Helly's lemma, there exists a point $z\in\hull_X(\alpha|_{[s,t]})$ lying in all $\mf{h}_{i_a}^*$ and all $\mf{k}_{j_b}$ (as well as $\mf{h}_k^*$ and $\mf{k}_k$). We clearly have $d(z,M)\geq n+1$, concluding the proof.
\end{proof}

\subsection{Conclusion}

We finally prove Theorem~\ref{uniformly non-qc ray} by combining Proposition~\ref{unqc prop} and Corollary~\ref{non-qc via grids}.

\begin{proof}[Proof of Theorem~\ref{uniformly non-qc ray}]
Since $H$ acts cofinitely on the subalgebra $M$, \cite[Lemma~4.12]{Fio10a} shows that $H$ is finitely generated. Thus, we can argue as in Remark~\ref{edge-connected rmk}(1) and thicken $M$ to an edge-connected subalgebra $M\sq M'\sq X^{(0)}$ that is still $H$--invariant and $H$--cofinite. 

Since $M'$ is at finite Hausdorff distance from $M$, it is not quasi-convex. Hence Corollary~\ref{non-qc via grids} guarantees that, for every $n\geq 0$, there exist $X$--transverse halfspaces $\mf{h}_n,\mf{k}_n\in\mscr{H}_{M'}(X)$ with $d({\rm res}_{M'}(\mf{h}_n),{\rm res}_{M'}(\mf{k}_n))>n$.

Proposition~\ref{unqc prop} yields geodesics $\alpha_n\sq\square(M')$ from $C_{M'}({\rm res}_{M'}(\mf{h}_n^*))$ to $C_{M'}({\rm res}_M(\mf{k}_n^*))$ with the property that, for all $t\geq s\geq 0$, the set $\hull_X(\alpha_n|_{[s,t]})$ contains points at distance $>\lfloor\tfrac{t-s}{K'}\rfloor$ from $M'$. Note that the length of the $\alpha_n$ diverges with $n$, since $d({\rm res}_{M'}(\mf{h}_n),{\rm res}_{M'}(\mf{k}_n))>n$.

Exploiting $H$--cocompactness of $M'$, we can assume that the initial vertices of the $\alpha_n$ all lie in a given finite subset of $M'$. Passing to a subsequence, the $\alpha_n$ converge to a combinatorial ray $r\sq\square(M')$. Every segment of $r$ is a segment of some $\alpha_n$ for large $n$, so it is still true that $\hull_X(r|_{[s,t]})$ contains points at distance $>\lfloor\tfrac{t-s}{K'}\rfloor$ from $M'$.

Finally, since $r\sq\square(M')$ and $M'$ is at finite Hausdorff distance from $M$, the ray $r$ stays at bounded distance from $M$. This concludes the proof.
\end{proof}

\section{The standard coarse median structure of a RACG}\label{std cms RACG sect}

Subsections~\ref{notation subsect2} and~\ref{main argument subsect} are devoted to the proof of Theorem~\ref{RACGs intro}(2), which is Theorem~\ref{std RACG thm} below. Then, in Subsection~\ref{loose subsect}, we deduce Corollary~\ref{loose cor intro}.

\subsection{Notation and a preliminary lemma}\label{notation subsect2}

Let $W_{\G}$ be a right-angled Coxeter group. 

In this case, the graph-product complex $\mc{D}$ defined in Subsection~\ref{graph prod subsect} is the cubical subdivision of a simpler $\CAT$ cube complex, which is usually known as the \emph{Davis complex} and which we denote by $\mc{D}_{\G}$. The $1$--skeleton of $\mc{D}_{\G}$ is naturally identified with the Cayley graph of $W_{\G}$ with respect to the generating set $\G^{(0)}$. Thus, every edge and every hyperplane of $\mc{D}_{\G}$ is labelled by a vertex of $\G$.

Let $[\mu_{\G}]\in\mc{CM}_{\square}(W_{\G})$ be the coarse median structure induced on $W_{\G}$ by $\mc{D}_{\G}$. We refer to $[\mu_{\G}]$ as the \emph{standard} coarse median structure on $W_{\G}$.

We say that a graph $\Delta$ is \emph{irreducible} if it is not a join of two proper subgraphs; equivalently, $W_{\Delta}$ is not a direct product of proper subgroups. For an (induced) subgraph $\Delta\sq\G$, we write: 
\[\Delta^{\perp}:=\{v\in\G^{(0)} \mid \Delta\sq\lk(v) \}.\]

The following is fairly classical, but we were not able to find a proof in the literature.

\begin{lem}\label{Morse path}
Let $\Delta\sq\G$ be an irreducible induced subgraph such that $\Delta^{\perp}$ spans a (possibly empty) clique. Let $\alpha\colon [0,+\infty)\ra\mc{D}_{\Delta}$ be an infinite edge path in $\mc{D}_{\Delta}\sq\mc{D}_{\G}$. If we have
\[K:=\sup_{g\in W_{\Delta},\ v\in\Delta} \diam\alpha^{-1}(g\mc{D}_{\Delta-\{v\}})<+\infty,\]
then $\alpha$ is a Morse quasi-geodesic in $\mc{D}_{\G}$.
\end{lem}
% Note: even ``relative'' Morseness can fail if $W_{\Delta}$ is a product of infinite subgroups, I think. (the factors might commute with other stuff)
\begin{proof}
For every (oriented) edge $e\sq\alpha$, denote by $\g(e)\in\Delta$ its label, by $\mf{w}(e)\in\mscr{W}(\mc{D}_{\G})$ the hyperplane it crosses, and by $\mf{h}(e)\in\mscr{H}(\mc{D}_{\G})$ the halfspace it enters. We say that $e$ is \emph{good} if $\alpha(0)$ lies in $\mf{h}(e)^*$ and the unbounded connected component of $\alpha- e$ is entirely contained in $\mf{h}(e)$.

Furthermore, for $n\geq 0$, let $e_n$ be the edge connecting $\alpha(n)$ and $\alpha(n+1)$.

\smallskip
{\bf Claim~1.} \emph{For every vertex $w\in\Delta$ and every sub-path $\alpha_0\sq\alpha$ of length $>K$, there exists a good edge $e\sq\alpha_0$ with $\g(e)=w$.}

\smallskip
\emph{Proof of Claim~1.}
Recall that the hyperplanes of $\mc{D}_{\Delta}$ labelled by $w$ are pairwise disjoint, and that the connected components of the complement of their union are precisely the translates $g\mc{D}_{\Delta-\{w\}}$ with $g\in W_{\Delta}$ (with some shreds of cubes attached).

Let $[m,n]\sq [0,+\infty)$ be a maximal interval such that the vertices $\alpha(m)$ and $\alpha(n)$ lie in the same translate of $\mc{D}_{\Delta-\{w\}}$. Let $g\mc{D}_{\Delta-\{w\}}$ be this translate, where $g\in W_{\Delta}$. Then $\g(e_n)=w$ and $\g(t)\in\mf{h}(e_n)$ for all $t\geq n+1$. If $m\neq 0$, we also have $\g(e_{m-1})=w$ and $\mf{w}(e_{m-1})\neq\mf{w}(e_n)$; otherwise, $\alpha(m-1)$ and $\alpha(n+1)$ would lie in some other translate $g'\mc{D}_{\Delta-\{w\}}$, contradicting maximality of $[m,n]$. This shows that either $m=0$ or $\alpha(0)\in\mf{h}(e_m)^*\sq\mf{h}(e_n)^*$, hence $e_n$ is always a good edge.

Now, let $\alpha(t)$ be the initial vertex of $\alpha_0$ and let $[m,n]$ be an interval containing $t$ that is maximal in the above sense. By definition of $K$, such an interval exists and we have $n-m\leq K$. Hence $n\leq t+K$ and $e_n$ is contained in $\alpha_0$. This is the required good edge, proving the claim.
\hfill$\blacksquare$

\smallskip
Now, fix a vertex $v\in\Delta$. Since $\Delta$ is irreducible, its complement graph $\Delta^c$ (where two vertices are adjacent if and only if they are not adjacent in $\Delta$) is connected. Let $D$ be its diameter. Choose a sequence of integers $(n_k)_{k\geq 0}$ such that each edge $e_{n_k}$ is good, with $\g(e_{n_k})=v$ and:
\[(2D-1)(K+1)<n_{k+1}-n_k\leq 2D(K+1).\] 
This is possible by Claim~1. Set $N:=\#\Delta^{\perp}$.

Recall that two hyperplanes of a cube complex are said to be \emph{$L$--separated}, for some $L\geq 0$, if they are disjoint and at most $L$ hyperplanes are transverse to both.

\smallskip
{\bf Claim~2.} \emph{For $k\neq k'$, the hyperplanes $\mf{w}(e_{n_k})$ and $\mf{w}(e_{n_{k'}})$ are $N$--separated in $\mc{D}_{\G}$.}

\smallskip
\emph{Proof of Claim~2.}
We have $\mf{h}(e_{n_0})\supseteq\mf{h}(e_{n_1})\supseteq\dots$ by construction. For every $w\in\Delta$ and $k\geq 0$, it is standard to show that the hyperplanes $\mf{w}(e_{n_k})$ and $\mf{w}(e_{n_{k+1}})$ are separated by a hyperplane labelled by $w$ (using Claim~1 and the inequality $n_{k+1}-n_k>(2D-1)(K+1)$, where $D=\diam\Delta^c$). 

Hence, if some $\mf{u}\in\mscr{W}(\mc{D}_{\G})$ is transverse to $\mf{w}(e_{n_k})$ and $\mf{w}(e_{n_{k+1}})$, then $\mf{u}$ is transverse to hyperplanes labelled by all vertices of $\Delta$, hence $\mf{u}$ must be labelled by a vertex of $\Delta^{\perp}$. Since $\Delta^{\perp}$ is a clique, it follows that at most $N$ hyperplanes are transverse to $\mf{w}(e_{n_k})$ and $\mf{w}(e_{n_{k+1}})$, proving the claim.
\hfill$\blacksquare$

\smallskip
For each $k\geq 0$, let $p_k$ be the gate-projection of $\alpha(0)$ to the carrier of $\mf{w}(e_{n_k})$. Let $\beta\sq\mc{D}_{\Delta}$ be a geodesic ray obtained by concatenating geodesics from $\alpha(0)$ to $p_0$ and from each $p_i$ to $p_{i+1}$.

\smallskip
{\bf Claim~3.} \emph{We have $d(p_k,\alpha(n_k))\leq N+2D(K+1)$ for all $k\geq 0$.}

\smallskip
\emph{Proof of Claim~3.}
Denote by $C_k$ the carrier of $\mf{w}(e_{n_k})$ and by $P_k$ the gate-projection of the halfspace $\mf{h}(e_{n_{k-1}})^*$ to $C_k$. Both $p_k$ and $\alpha(n_k)$ lie in $C_k$, with $p_k$ in fact lying in $P_k$. Since $\mf{w}(e_{n_k})$ and $\mf{w}(e_{n_{k-1}})$ are $N$--separated, the convex subcomplex $P_k$ is crossed by at most $N$ hyperplanes, and so it has diameter $\leq N$. By the properties of gate-projections, we have:
\[d(\alpha(n_k),P_k)\leq d(\alpha(n_k),\mf{h}(e_{n_{k-1}})^*)\leq d(\alpha(n_k),\alpha(n_{k-1}))\leq 2D(K+1). \]
Recalling that $p_k\in P_k$ and $\diam(P_k)\leq N$, we obtain the claim.
\hfill$\blacksquare$

\smallskip
Finally, Claim~3 shows that the distances $d(p_k,p_{k+1})$ are uniformly bounded, so Theorem~4.2 and Theorem~2.14 in \cite{Charney-Sultan} imply that $\beta$ is Morse. It is clear that $\alpha$ and $\beta$ are at finite Hausdorff distance from each other, so it follows that $\alpha$ is a Morse quasi-geodesic. 
\end{proof}

\subsection{The main argument}\label{main argument subsect}

Lemma~\ref{Morse path} allows us to translate Theorem~\ref{uniformly non-qc ray} into the following practical result, which will quickly yield Theorem~\ref{RACGs intro}(2).
 
\begin{prop}\label{cc or commutes}
Let $W_{\G}\acts X$ be a cocompact cubulation. Suppose that there exists an irreducible subgraph $\Delta\sq\G$ such that $W_{\Delta}$ is median-cocompact in $X$. Suppose further that, for every $x\in\Delta$, the subgroup $W_{\Delta-\{x\}}$ is convex-cocompact in $X$. Then:
\begin{enumerate}
\item either $W_{\Delta}$ is itself convex-cocompact in $X$,
\item or the centraliser $Z_{W_{\G}}(W_{\Delta})=W_{\Delta^{\perp}}$ is infinite.
\end{enumerate}
\end{prop}
\begin{proof}
Let $M\sq X^{(0)}$ be a median subalgebra on which $W_{\Delta}$ acts cofinitely. Suppose that $W_{\Delta}$ is not convex-cocompact in $X$, that is, that $M$ is not quasi-convex in $X$ (Proposition~\ref{prop:cc=qc}).

Theorem~\ref{uniformly non-qc ray} gives us a ray $r\sq X$ at bounded distance from $M$, and a constant $K$ such that, for every segment $\s\sq r$ of length $\ell$, the set $\hull_X(\s)$ contains points at distance $\geq\lfloor\tfrac{\ell}{K}\rfloor$ from $M$. 

Fix a basepoint $p\in M$ and, for every $x\in\Delta$, denote $\mc{O}_x:=W_{\Delta-\{x\}}\cdot p$. Since $W_{\Delta-\{x\}}$ is convex-cocompact in $X$, Proposition~\ref{prop:cc=qc} guarantees that there exists a constant $R\geq 0$ such that $m(\mc{O}_x,\mc{O}_x,X)\sq\mc{N}_R(\mc{O}_x)$ for every $x\in\Delta$. Setting $\delta:=\dim X$, Remark~\ref{qc hulls rmk} implies that $\hull_X(\mc{O}_x)\sq\mc{N}_{2^{\delta}R}(\mc{O}_x)\sq\mc{N}_{2^{\delta}R}(M)$, hence $\hull_X(g\mc{O}_x)\sq\mc{N}_{2^{\delta}R}(M)$ for every $g\in W_{\Delta}$.

It follows that, for every constant $C\geq 0$, there exists a constant $C'\geq 0$ such that the intersection between $r$ and the $C$--neighbourhood of any $g\mc{O}_x$ with $g\in W_{\Delta}$ has diameter at most $C'$.

Now, by the Milnor--Schwarz lemma, there exists a $W_{\G}$--equivariant quasi-isometry $q\colon\mc{D}_{\G}\ra X$ with $q(\mc{D}_{\Delta})\sq M$. Let $\alpha\sq\mc{D}_{\Delta}$ be a quasi-geodesic edge path such that $d_{\rm Haus}(q(\alpha),r)<+\infty$. By the previous paragraph and the fact that $\alpha$ is a quasi-geodesic, we see that $\alpha$ satisfies the hypothesis of Lemma~\ref{Morse path}. 

If $\Delta^{\perp}$ were a clique, then Lemma~\ref{Morse path} would show that $\alpha$, and hence $r$, is Morse. However, this would imply that $d_{\rm Haus}(r,\hull_X(r))<+\infty$, contradicting the fact that $\hull_X(r)$ contains points arbitrarily far from $M$.

Thus, there must exist vertices $x,y\in\Delta^{\perp}$ that are not connected by an edge. The subgroup $\langle x,y\rangle$ is infinite and it is contained in $Z_{W_{\G}}(W_{\Delta})=W_{\Delta^{\perp}}$, proving the proposition.
\end{proof}

We are finally ready to prove Theorem~\ref{RACGs intro}(2), which is the following result.

\begin{thm}\label{std RACG thm}
Let $W_{\G}\acts X$ be a cocompact cubulation where, for all $x,y\in\G$, the subgroup $\langle x,y\rangle$ is convex-cocompact. Then $W_{\G}\acts X$ induces the standard coarse median structure on $W_{\G}$.
\end{thm}
\begin{proof}
We prove the statement by induction on $\#\G^{(0)}$. The base case where $\G$ is a singleton is obvious. We now discuss the inductive step.

Hyperplane-stabilisers for the action $W_{\G}\acts\mc{D}_{\G}$ are subgroups of the form $W_{\lk(v)}$ with $v\in\G$. In view of Theorem~\ref{coarse median vs hyperplanes}, it suffices to show that these subgroups are all convex-cocompact in $X$. 

Fix a vertex $v\in\G$. 

\smallskip
{\bf Claim~1.} \emph{For every $\Delta\sq\lk(v)$, the subgroup $W_{\Delta}$ is median-cocompact in $X$.}

\smallskip
\emph{Proof of Claim~1.}
Since $W_{\lk(v)}$ has index $2$ in the centraliser of $v$, Proposition~\ref{centralisers are median-cocompact} guarantees that it is median-cocompact in $X$. Let $M\sq X$ be a median subalgebra on which $W_{\lk(v)}$ acts cofinitely. By Chepoi--Roller duality, there exists a cocompact cubulation $W_{\lk(v)}\acts Y$ such that $Y^{(0)}$ is equivariantly isomorphic to $M$ as a median algebra. For every $x,y\in\lk(v)$, the fact that $\langle x,y\rangle$ is convex-cocompact in $X$ implies that it is also convex-cocompact in $Y$.

Since $\lk(v)$ is a proper subgraph of $\G$, the inductive hypothesis implies that $Y$ induces the standard coarse median structure on $W_{\lk(v)}$. Using Proposition~\ref{prop:cc=qc}, this implies that, for every $\Delta\sq\lk(v)$, the subgroup $W_{\Delta}$ is median-cocompact in $Y$, and hence in $X$.
\hfill$\blacksquare$

\smallskip
Now, suppose for the sake of contradiction that $W_{\lk(v)}$ is not convex-cocompact in $X$. Let $\Delta_0\subset\lk(v)$ be a minimal subgraph such that $W_{\Delta_0}$ is not convex-cocompact in $X$. Note that $\Delta_0$ exists and has at least $3$ vertices, by our assumptions.

By Claim~1, $W_{\Delta_0}$ is median-cocompact in $X$ and, by minimality of $\Delta_0$, all subgroups $W_{\Delta_0-\{x\}}$ with $x\in\Delta_0$ are convex-cocompact in $X$. Minimality also implies that $\Delta_0$ is irreducible, because of Lemma~\ref{cc products}. Thus, Proposition~\ref{cc or commutes} guarantees that $W_{\Delta_0^{\perp}}$ is infinite, that is, there exist $z,z'\in\G$ such that $\langle z,z'\rangle\simeq D_{\infty}$ and $\Delta_0\sq\lk(z)\cap\lk(z')$.

\smallskip
{\bf Claim~2.} \emph{The subgroup $W_{\lk(z)\cap\lk(z')}$ is convex-cocompact in $X$.}

\smallskip
\emph{Proof of Claim~2.}
The proof of this fact is almost identical to that of Lemma~\ref{lem:cc_vertices}.

We have $Z_{W_{\G}}(zz')=\langle zz'\rangle\x W_{\lk(z)\cap\lk(z')}$. The subgroup $\langle zz'\rangle\simeq\Z$ is convex-cocompact in $X$ by our assumptions, since it has finite index in $\langle z,z'\rangle$. Choosing $n\geq 1$ such that $(zz')^n$ acts non-transversely on $X$, Lemma~\ref{cc properties}(4) implies that $Z_{W_{\G}}((zz')^n)$ is convex-cocompact in $X$. Note that $Z_{W_{\G}}((zz')^n)=Z_{W_{\G}}(zz')$. In addition, the index--$2$ subgroup of $W_{\lk(z)\cap\lk(z')}$ consisting of words of even length is generated by elements $uu'$ with $u,u'\in\lk(z)\cap\lk(z')$, which all generate infinite cyclic subgroups that are convex-cocompact in $X$, by assumption. Finally, Lemma~\ref{lem:cc_products} implies that $W_{\lk(z)\cap\lk(z')}$ is convex-cocompact in $X$.
\hfill$\blacksquare$

\smallskip
Now, let $C\sq X$ be a convex subcomplex on which $W_{\lk(z)\cap\lk(z')}$ acts cocompactly. Again, since $\lk(z)\cap\lk(z')$ is a proper subgraph of $\G$, the inductive hypothesis implies that $W_{\lk(z)\cap\lk(z')}$ inherits the standard coarse median structure from its action on $C$. Using Proposition~\ref{prop:cc=qc}, it follows that $W_{\Delta_0}$ is convex-cocompact in $C$, and hence in $X$. This is the required contradiction.
\end{proof}

Now that the proof of Theorem~\ref{RACGs intro} is complete, it is interesting to discuss why the Conjecture from the Introduction is harder.

\begin{rmk}
    Say that we have an ``exotic'' cocompact cubulation $W_{\G}\acts X$, i.e.\ one that does \emph{not} induce the coarse median structure of the Davis complex. Suppose that we are interested in showing that a different cubulation $W_{\G}\acts Y$ induces the same coarse median structure, under the assumption that the two cubulations give the same coarse median structure on all abelian subgroups. 

    The only tool we have for this at the moment is Theorem~\ref{coarse median vs hyperplanes}, which would require us to show that the hyperplane-stabilisers of the action $W_{\G}\acts X$ are convex-cocompact with respect to the cocompact cubulation $W_{\G}\acts Y$.

    In the ``standard'' case, we were able to exploit the fact that hyperplane-stabilisers of the Davis complex are centralisers, and the fact that centralisers are universally median-cocompact (Proposition~\ref{centralisers are median-cocompact}). While still far from convex-cocompactness, median-cocompactness gives us powerful information that puts Theorem~\ref{uniformly non-qc ray} in motion (recall the cautionary Example~\ref{ex:mapping_torus} without this assumption).

    For exotic cubulations there is no guarantee that hyperplane-stabilisers will behave as nicely. We can follow this blueprint only if we can find cubulations with ``good'' hyperplane-stabilisers representing our exotic coarse median structures. Here ``good'' would need to mean ``universally median-cocompact'' (i.e.\ median-cocompact in every cocompact cubulation of $W_{\G}$), or significantly new ideas would be required.
\end{rmk}

\subsection{Loose and bonded squares}\label{loose subsect}

In this subsection, we prove Corollary~\ref{loose cor intro} by combining Theorem~\ref{RACGs intro}(2) with the cubical flat torus theorem \cite{WW} and our study of cubical coarse medians on products of dihedrals (Proposition~\ref{cms on D^n}). Rather than with the \emph{bonded squares} mentioned in the introduction, it will be convenient to work with their opposite: \emph{loose squares} (Definition~\ref{defn:loose_squares} below).

Let $\G$ be a finite simplicial graph. For simplicity, we say that a \emph{square} is an induced $4$--cycle $\Delta\sq\G$. A \emph{hyperoctahedron} is an induced subgraph $\Lambda\sq\G$ whose opposite graph is a union of pairwise disjoint edges (in other words, $W_{\Lambda}\simeq D_{\infty}^n$ for some $n\geq 0$).

It is worth remarking that maximal virtually-abelian subgroups of the right-angled Coxeter group $W_{\G}$ are not always \emph{highest} in the sense of Subsection~\ref{subsec:cc}. As a consequence, they might not be convex-cocompact in all cocompact cubulations of $W_{\G}$. Instead, if $\Lambda\sq\G$ is a maximal hyperoctahedron, then the virtually abelian subgroup $W_{\Lambda}\simeq D_{\infty}^n$ is necessarily highest in $W_{\G}$.

The following equivalent conditions characterise loose squares $\Delta\sq\G$.

\begin{lem}\label{lem:loose_equivalence}
For a square $\Delta\sq\G$, the following conditions are equivalent.
\begin{enumerate}
\setlength\itemsep{.1cm}
\item For every maximal induced subgraph $\Lambda\sq\G$ such that $W_{\Lambda}$ is virtually abelian, either $\Delta\sq\Lambda$ or $W_{\Delta\cap\Lambda}$ is finite.
\item For every maximal hyperoctahedron $\Lambda\sq\G$, either $\Delta\sq\Lambda$ or $W_{\Delta\cap\Lambda}$ is finite.
\item For every square $\Delta'\sq\G$, either $W_{\Delta\cup\Delta'}$ is virtually abelian or $W_{\Delta\cap\Delta'}$ is finite.
\item There does not exist a square $\Delta'\sq\G$ such that $\Delta\cap\Delta'$ has exactly $3$ vertices.
\end{enumerate}
\end{lem}
\begin{proof}
We begin with $(1)\Ra(2)$. If $\Lambda\sq\G$ is a maximal hyperoctahedron and $\Lambda'\sq\G$ is a maximal subgraph such that $W_{\Lambda'}$ is virtually abelian and $\Lambda\sq\Lambda'$, then $W_{\Lambda'}=W_{\Lambda}\x(\Z/2\Z)^m$ for some $m\geq 0$. By~(1), either $\Delta\sq\Lambda'$ or $W_{\Delta\cap\Lambda'}$ is finite, which yields the analogous statement for $\Lambda$.

Let us prove $(2)\Ra (3)$, or rather $\neg (3)\Ra\neg(2)$. Let $\Delta'\sq\G$ be a square such that $W_{\Delta\cup\Delta'}$ is not virtually abelian and $W_{\Delta\cap\Delta'}$ is infinite. Choose a maximal hyperoctahedron $\Lambda\supseteq\Delta'$. Then $\Delta\not\sq\Lambda$, since $W_{\Delta\cup\Delta'}$ is not virtually abelian, unlike $W_{\Lambda}$. In addition, $W_{\Delta\cap\Lambda}$ is infinite, since it contains $W_{\Delta\cap\Delta'}$. Thus, $\Lambda$ is a hyperoctahedron witnessing $\neg(2)$.

Regarding the implication $(3)\Ra(4)$, cyclically label by $a,b,c,d$ the vertices of $\Delta$. If there existed a square $\Delta'$ with vertices $a,b,c,x$ and $x\neq d$, the subgroup $W_{\Delta\cap\Delta'}$ would be infinite. Moreover, we would have $W_{\Delta\cup\Delta'}=\langle a,c\rangle\x\langle b,d,x\rangle$, with $\langle b,d,x\rangle$ isomorphic to either $\Z/2\Z\ast\Z/2\Z\ast\Z/2\Z$ or $\Z/2\Z\ast(\Z/2\Z)^2$. In particular, $W_{\Delta\cup\Delta'}$ would not be virtually abelian, contradicting $(3)$.

Finally, we prove $(4)\Ra(1)$. Suppose that $\Lambda\sq\G$ is a maximal induced subgraph such that $W_{\Lambda}$ virtually abelian. Suppose further that $W_{\Delta\cap\Lambda}$ is a proper, infinite subgroup of $W_{\Delta}$. Cyclically labelling the vertices of $\Delta$ by $a,b,c,d$, we can assume that $a,c\in\Lambda$ and $b\not\in\Lambda$. Moreover, we have a splitting $W_{\Lambda}=\langle a,c\rangle \x D_{\infty}^n \x (\Z/2\Z)^m$ with $n+m\geq 1$, by maximality of $\Lambda$.

For a vertex $v\in\Lambda-\{a,c\}$, neither $\{a,b,c,v\}$ nor $\{a,v,c,d\}$ can be the vertex set of square distinct from $\Delta$, because of (4). This implies that either $v=d$ or $v$ commutes with both $b$ and $d$. It follows that $W_{\Delta\cup\Lambda}$ is virtually abelian, which contradicts maximality of $\Lambda$, since $b\not\in\Lambda$. 

This completes the proof of the lemma.
\end{proof}

\begin{defn}\label{defn:loose_squares}
A square $\Delta\sq\G$ is \emph{loose} if it satisfies the equivalent conditions in Lemma~\ref{lem:loose_equivalence}. A square is \emph{bonded} if it is not loose.
%there does {\bf not} exist another square $\Delta'\sq\G$ such that $\Delta\cap\Delta'$ has exactly $3$ vertices.
\end{defn}

This is equivalent to the definition given in the Introduction: a square $\Delta\sq\G$ is bonded if there exists another square $\Delta'\sq\G$ such that the intersection $\Delta\cap\Delta'$ has exactly $3$ vertices.

Corollary~\ref{loose cor intro} claims that $W_{\G}$ satisfies coarse cubical rigidity when the intersection pattern of squares in $\G$ is sufficiently intricate, namely when every square is bonded. We are finally ready to prove this statement.

\begin{proof}[Proof of Corollary~\ref{loose cor intro}]
Suppose $\G$ has no loose squares. Let $W_{\G}\acts Y$ be a cocompact cubulation. 

Let $\Lambda\sq\G$ be a maximal hyperoctahedron. Thus, $W_{\Lambda}\simeq D_{\infty}^n$ for some $n\geq 1$ and $W_{\Lambda}$ is a highest virtually abelian subgroup of $W_{\G}$, as defined in Subsection~\ref{subsec:cc}. By Lemma~\ref{cc properties}(3), $W_{\Lambda}$ is convex-cocompact in $Y$. By Proposition~\ref{cms on D^n}, either $W_{\Lambda}$ inherits the standard coarse median structure from $Y$ (i.e.\ that of the Davis complex for $W_{\Lambda}$), or there exists a square $\Delta=\{x_1,x_2,x_3,x_4\}\sq\Lambda$ such that $W_{\Delta}$ is convex-cocompact in $Y$, but the subgroups $\langle x_1,x_3\rangle\simeq\langle x_2,x_4\rangle\simeq D_{\infty}$ are not.

However, since the square $\Delta$ cannot be loose, Lemma~\ref{lem:loose_equivalence}(2) guarantees the existence of a maximal hyperoctahedron $\Lambda'\sq\G$ such that $W_{\Delta}\cap W_{\Lambda'}$ is a proper, infinite subgroup of $W_{\Delta}$. By Lemma~\ref{cc properties}(3), the subgroup $W_{\Lambda'}$ is convex-cocompact in $Y$ and, by Lemma~\ref{cc properties}(1), so is the intersection $W_{\Delta}\cap W_{\Lambda'}$. Now, this intersection is commensurable to either $\langle x_1,x_3\rangle$ or $\langle x_2,x_4\rangle$, showing that at least one of them is convex-cocompact in $Y$.

Combined with the previous paragraph, this proves that, for every maximal hyperoctahedron $\Lambda\sq\G$, the action $W_{\G}\acts Y$ induces the standard coarse median structure on the subgroup $W_{\Lambda}$. If $x,y\in\G$ are vertices with $\langle x,y\rangle\simeq D_{\infty}$, we certainly have $\{x,y\}\sq\Lambda$ for some maximal hyperoctahedron $\Lambda$, so the subgroup $\langle x,y\rangle$ is convex-cocompact in $Y$. 

Now, Theorem~\ref{std RACG thm} implies that $Y$ induces the standard coarse median structure on $W_{\G}$, proving the corollary.
\end{proof}

We conclude the section by giving an example of a right-angled Coxeter group $W_{\G}$ that fails to satisfy coarse cubical rigidity. Of course, we have seen that $W_{\G}=D_{\infty}^n$ is such a group for $n\geq 2$, and it is also easy to come up with examples splitting over finite subgroups. 

Instead, the group we are about to exhibit is one-ended and directly irreducible. 

\begin{ex}\label{exotic example}
Consider the graph $\Gamma$ from Figure~\ref{fig:example_defining_graph}. Let $\Gamma_1$ be the subgraph spanned by the vertices $\{a, b, c, d\}$, let $\Gamma_2$ be the subgraph spanned by the vertices $\{a, e, f, c\}$, and let $\Gamma_3$ spanned by the vertices $\{a, c\}$.
We have the following amalgamated product decomposition $W_\Gamma = W_{\Gamma_1} \ast_{W_{\Gamma_3}} W_{\Gamma_2}$.

Let $X_1$ be the usual square tiling of $\mathbb{E}^2$. See Figure~\ref{fig:virtual_Z2_action}. Each vertex $x \in \Gamma_1$ acts by reflection about the line $l_x$. Note that the subcomplex $l$ is preserved by $W_{\Gamma_3}$.

Let $C$ be a $2$-cube with opposite pairs of vertices identified, and let $X_2$ be the universal cover of $C$ (see Figure~\ref{fig:virtual_free_action}). We define an action of $W_{\Gamma_2}$ on $X_2$ by letting $e$ and $f$ act by reflections about the lines $l_e$ and $l_f$ respectively. Let $a$ (resp. $c$) act as a reflection 
about the line perpendicular to $l_e$ (resp. $l_f$) and through the vertex $v_a$ (resp. $v_c$). Note that there is an infinite subcomplex $l'$ (shown in blue) that is preserved by $W_{\Gamma_3}$.

Let $T$ be the Bass-Serre tree associated to $W_\Gamma = W_{\Gamma_1} \ast_{W_{\Gamma_3}} W_{\Gamma_2}$. We define a blowup $X$ of $T$. We blowup each vertex $v$ of $T$ corresponding to a coset of $W_{\Gamma_1}$ (resp.\ $W_{\Gamma_2}$) to a copy $X_v$ of the complex $X_1$ (resp.\ $X_2$). For each edge $[u,v]$ of $T$ corresponding to a coset $gW_{\Gamma_3}$, we glue $gl \subset X_u$ to $gl' \subset X_v$ where, up to relabelling, we are assuming that $X_u$ is a copy of $X_1$ and $X_v$ is a copy of $X_2$. Even though we are glueing over non-convex subcomplexes, it is not hard to see that the cube complex $X$ is $\CAT$. The result is a cocompact cubulation $W_{\Gamma}\acts X$ that is not strongly cellular. 

The induced coarse median structure on $W_{\G}$ is not the standard one, since $\langle a,c\rangle$ is not convex-cocompact in $X$, as is evident from Figure~\ref{fig:virtual_Z2_action}.
\end{ex}

\begin{figure}
    \centering
    \includegraphics[scale=.35]{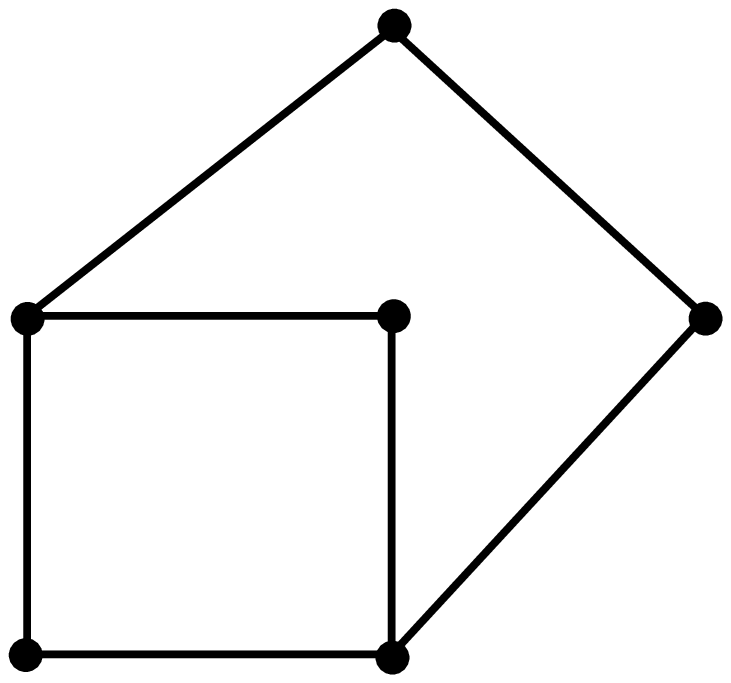}
    \put(-80,0){$d$}
    \put(-80,37){$a$}
    \put(-32,37){$b$}
    \put(-30,0){$c$}
    \put(-35,69){$e$}
    \put(0,35){$f$}
    \caption{The square $abcd$ is loose.}
    \label{fig:example_defining_graph}
\end{figure}

\begin{figure}
\centering
\begin{minipage}{.5\textwidth}
    \centering
    \includegraphics[scale=0.95]{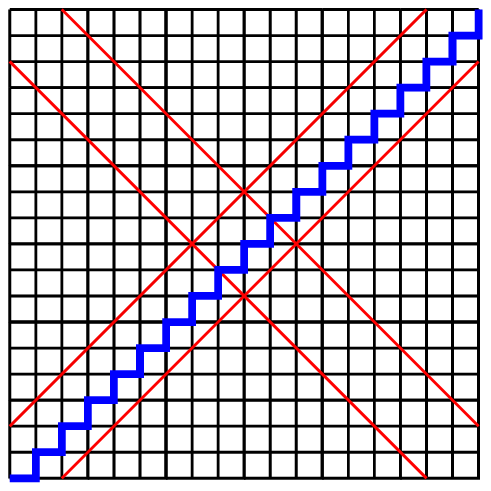}
    \put(-142,15){$l_b$}
    \put(-120,-8){$l_d$}
    \put(-20,-8){$l_a$}
    \put(3,15){$l_c$}
    \put(-140, -5){$l$}
    \caption{Action of $W_{\Gamma_1}$ on $X_1$.}
    \label{fig:virtual_Z2_action}
\end{minipage}%
\begin{minipage}{.5\textwidth}
    \centering
    \includegraphics[scale=1.15]{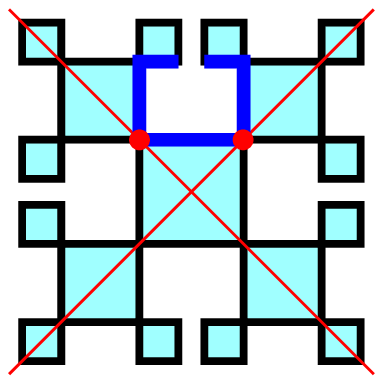}
    \put(-92,69){$v_a$}
    \put(-41,69){$v_c$}
    \put(-122,-12){$l_f$}
    \put(-5,-12){$l_e$}
    \put(-63,85){$l'$}
    \caption{Action of $W_{\Gamma_2}$ on $X_2$.}
    \label{fig:virtual_free_action}
\end{minipage}
\end{figure}

The previous example suggests that it should always be possible to exploit loose squares in $\G$ to produce cocompact cubulations of $W_{\G}$ inducing a coarse median structure other than that of the Davis complex. If this is indeed the case, then Corollary~\ref{loose cor intro} is sharp.

\section{General Coxeter groups}\label{Coxeter sect}

In this section, we prove Theorem~\ref{autom intro}(2).

Let $(W,S)$ be a general Coxeter system. A subgroup of $W$ is \emph{special} if it is generated by a subset of $S$, and it is \emph{parabolic} if it is conjugate to a special subgroup. An element of $W$ is an \emph{involution} if it has order $2$, and it is a \emph{reflection} if it is conjugate to an element of $S$.

Involutions in Coxeter groups were fully classified by Richardson \cite{Richardson}. We record his results in the following lemma. For a more recent account, the reader can also consult \cite[$\S$27.2--27.4]{Kane}.

\begin{lem}\label{involutions}
\begin{enumerate}
\item[]
\item Let $P$ be a finite irreducible Coxeter group. The centre $Z(P)$ is nontrivial if and only if $Z(P)\simeq\Z/2\Z$. In this case, $Z(P)=\langle w_P\rangle$ for an element $w_P\in P$ that is the longest element of $P$ with respect to any Coxeter generating set. 
\item Let $W$ be a general Coxeter group. Every involution in $W$ is the longest element $w_P$ of a parabolic subgroup $P\leq W$ of the form $P=P_1\x\dots\x P_k$, where each $P_i$ is a finite irreducible Coxeter group with nontrivial centre. In particular, we have $w_P=w_{P_1}\cdot\ldots\cdot w_{P_k}$. 
\end{enumerate}
\end{lem}

\begin{rmk}\label{longest element in parabolics}
Let $(W,S)$ be a Coxeter system with $W$ finite. Let $w_o\in W$ be the longest element with respect to $S$. Since $w_o$ is longest, any reduced expression for $w_o$ will involve all elements of $S$, so $w_o$ does not lie in any proper \emph{special} subgroup of $W$. In general, however, $w_o$ can lie in a proper \emph{parabolic} subgroup of $W$. For instance, this happens when $W$ is the dihedral group with $6$ elements.

Things are different if we suppose that $W=P$, where $P$ has the form from Lemma~\ref{involutions}(2). In this case, the longest element $w_P\in P$ is not contained in any proper parabolic subgroup of $P$. Indeed, since $w_P$ is central, it is contained in a proper parabolic subgroup if and only if it is contained in a proper special subgroup of $W$ (keeping a Coxeter generating set fixed).
\end{rmk}

\begin{lem}\label{NR wall-stabilisers}
Let $W$ be a Coxeter group and let $W\acts\X$ denote its Niblo--Reeves cubulation. Stabilisers of hyperplanes of $\X$ are precisely conjugates of centralisers $Z_W(r)$ with $r\in W$ a reflection.
\end{lem}
\begin{proof}
Let $\Sigma$ be the presentation $2$--complex of $W$ associated with the standard presentation. For all $s,t\in S$ with $(st)^m=1$, there is a $W$--orbit of $2m$--gons in $\Sigma$ with edges alternately labelled by $s$ and $t$. There are some natural walls in $\Sigma$, uniquely determined by the following property: if $\mf{w}$ is a wall and $C\sq\Sigma$ is a $2$-cell, then $\mf{w}\cap C$ is either empty or a segment joining midpoints of opposite edges of $C$. The Niblo--Reeves cubulation of $W$ is precisely the $\CAT$ cube complex associated with this collection of walls \cite{Niblo-Reeves}.

For each reflection $r\in S$, there is a unique wall $\mf{w}_r\sq\Sigma$ intersecting the edge $[1,r]$ in its midpoint. Every wall in $\Sigma$ is in the $W$--orbit of one of the walls $\mf{w}_r$. Thus, it suffices to show that the $W$--stabiliser of the wall $\mf{w}_r\sq\Sigma$ is precisely the centraliser $Z_W(r)$.

Note that $r$ fixes the wall $\mf{w}_r$ pointwise and swaps its two sides. In fact, $r$ is the only element of $W$ with this property, since $W$ acts freely on the $0$--skeleton of $\Sigma$. Now, if $g\in W$, we have $g\mf{w}_r=\mf{w}_r$ if and only if $grg^{-1}$ again fixes $\mf{w}_r$ pointwise and swaps its two sides. Thus, $g\mf{w}_r=\mf{w}_r$ if and only if $grg^{-1}=r$, i.e.\ $g\in Z_W(r)$ as required.
\end{proof}

\begin{lem}\label{involution centralisers}
Let $W$ be Coxeter group with cocompact Niblo--Reeves cubulation $W\acts\X$. Then, for every involution $\s\in W$, the centraliser $Z_W(\s)$ is convex-compact in $\X$.
\end{lem}
\begin{proof}
In view of Lemma~\ref{involutions}, we have $\s=w_P$ for a finite, parabolic subgroup $P=P_1\x\dots\x P_k$, where each $P_i$ is finite and irreducible.

We begin by showing that the centraliser $Z_W(w_P)$ is contained in the normaliser $N_W(P)$. Indeed, if $g\in W$ commutes with $w_P$, then $w_P\in P\cap gPg^{-1}$. Note that $P\cap gPg^{-1}$ is a parabolic subgroup of $P$ by \cite[Lemma~5.3.6]{Davis}, but it cannot be a \emph{proper} parabolic subgroup because of Remark~\ref{longest element in parabolics}. We conclude that $P\leq gPg^{-1}$ and, since $P$ and $gPg^{-1}$ are finite groups of the same cardinality, we must have $gPg^{-1}=P$. This shows that $g\in N_W(P)$, as required.

Now, we have a chain of inclusions $Z_W(P)\leq Z_W(w_P)\leq N_W(P)$. Since $P$ is finite, $Z_W(P)$ has finite index in $N_W(P)$, hence $Z_W(P)$ has finite index in $Z_W(w_P)$ as well. 

Finally, observe that $Z_W(P)$ is convex-cocompact in $\X$. Indeed, $P$ is generated by finitely many reflections $r_1,\dots,r_k$. Each centraliser $Z_W(r_i)$ is the stabiliser of a hyperplane of $\X$ by Lemma~\ref{NR wall-stabilisers}, hence it is convex-cocompact in $\X$. Thus $Z_W(P)=\bigcap_i Z_W(r_i)$ is convex-cocompact in $\X$ by Lemma~\ref{cc properties}(1). Lemma~\ref{cc properties}(2) implies that $Z_W(w_P)$ is convex-cocompact, completing the proof.
\end{proof}

\begin{proof}[Proof of Theorem~\ref{autom intro}(2)]
Let $W\acts\X$ be the Niblo--Reeves cubulation of $W$, which is cocompact by the assumptions of the theorem. Consider an automorphism $\varphi\in\Aut(W)$ and let $W\acts\X^{\varphi}$ denote the standard action on the Niblo--Reeves cubulation precomposed with $\varphi$. 

We want to show that $\X$ and $\X^{\varphi}$ induce the same coarse median structure on $W$ by invoking Theorem~\ref{coarse median vs hyperplanes}. By Lemma~\ref{NR wall-stabilisers}, up to conjugacy, stabilisers of hyperplanes of $\X^{\varphi}$ are of the form $\varphi^{-1}(Z_W(r))$, with $r\in W$ a reflection. Observing that $\varphi^{-1}(Z_W(r))=Z_W(r')$ for the involution $r':=\varphi^{-1}(r)$, Lemma~\ref{involution centralisers} shows that all these subgroups are convex-cocompact in $\X$, as required for Theorem~\ref{coarse median vs hyperplanes}.
\end{proof}

\appendix

\section{Cocompactness of intersections}\label{appendix}

A collection $\mc{C}$ of subsets of a metric space $X$ is \emph{locally finite} if every ball in $X$ intersects only finitely many elements of $\mc{C}$.

\begin{lem}\label{cocpt vs lf}
Let $G\acts X$ be a proper cocompact action on a metric space. A closed subset $A\sq X$ is acted upon cocompactly by its $G$--stabiliser if and only if the orbit $G\cdot A$ is locally finite.
\end{lem}
\begin{proof}
The backward arrow is \cite[Lemma~2.3]{Hagen-Susse}, since $X$ is a proper metric space (as it admits a geometric group action).

For the forward arrow, denote by $G_A\leq G$ the stabiliser of $A$ and assume that the action $G_A\acts A$ is cocompact. Suppose for the sake of contradiction that a ball $B\sq X$ intersects infinitely many pairwise distinct translates $g_nA$ with $g_n\in G$. Then the balls $g_n^{-1}B$ all intersect $A$. Since $G_A\acts A$ is cocompact and $G\acts X$ is proper, the elements $g_n^{-1}$ all lie in a product set $G_A\cdot F$ with $F\sq G$ finite. Hence all $g_n$ lie in $F^{-1}\cdot G_A$, contradicting that there are infinitely many distinct sets $g_nA$.
\end{proof}

\begin{lem}[Cocompact Intersections]\label{cocompact intersections}
Let $G\acts X$ be a proper cocompact action on a metric space. Let $A,B\sq X$ be closed subsets that are invariant and acted upon cocompactly by subgroups $H,K\leq G$, respectively. Then the action $H\cap K\acts A\cap B$ is cocompact (possibly, $A\cap B=\emptyset$).
\end{lem}
\begin{proof}
Suppose that $A\cap B\neq\emptyset$. We first prove that $A\cap B$ is acted upon cocompactly by its $G$--stabiliser. In view of Lemma~\ref{cocpt vs lf}, it suffices to show that the orbit $G\cdot (A\cap B)$ is locally finite.

Suppose for the sake of contradiction that $g_n(A\cap B)$ are pairwise distinct translates intersecting a ball $L\sq X$. The collections $\{g_nA\}$ and $\{g_nB\}$ also intersect $L$, so they must be finite by Lemma~\ref{cocpt vs lf}. Hence there are only finitely many possible intersections $g_nA\cap g_nB$, contradicting our assumption.

Now, we know that $A\cap B$ is acted upon cocompactly by its $G$--stabiliser $G_{A\cap B}$. A finite-index subgroup of $G_{A\cap B}$ must stabilise $A$, since $A\cap B$ is contained in all $G_{A\cap B}$--translates of $A$, and there are only finitely many such translates of $A$ by local finiteness. Since $H$ acts cocompactly on $A$, it must have finite index in the $G$--stabiliser of $A$, hence a finite-index subgroup of $G_{A\cap B}$ is contained in $H$ (and, similarly, in $K$). In conclusion, a finite-index subgroup of $G_{A\cap B}$ acts cocompactly on $A\cap B$ and is contained in $H\cap K$, which concludes the proof.
\end{proof}

\bibliography{./mybib}
\bibliographystyle{alpha}

\end{document}